\documentclass[11pt]{amsart}
\usepackage{amsmath,amsthm,amsfonts,amscd,amssymb,eucal,latexsym}
\usepackage{mathrsfs}
\usepackage[numbers]{natbib}
\usepackage[fit]{truncate}
\usepackage{fullpage}
\usepackage{graphicx}
\usepackage{epsfig}
\usepackage{color,soul}
\usepackage[all,cmtip]{xy}
\usepackage[top=1.125in, bottom=1.125in, left=1.25in, right=1.25in]{geometry}
\newcommand{\truncateit}[1]{\truncate{0.8\textwidth}{#1}}
\newcommand{\scititle}[1]{\title[\truncateit{#1}]{#1}}

\theoremstyle{plain}
\newtheorem{theorem}{Theorem}[section]
\newtheorem{corollary}[theorem]{Corollary}
\newtheorem{lemma}[theorem]{Lemma}
\newtheorem{claim}[theorem]{Claim}

\theoremstyle{definition}
\newtheorem{definition}[theorem]{Definition}
\newtheorem{remark}[theorem]{Remark}

\def\de{\delta} \def\dl{\partial}   \def\Si{\Sigma} \def\si{\sigma}
\def\su{\subset}    \def\De{\Delta} \def\Lm{\Lambda}
\def\al{\alpha} \def\be{\beta}      \def\ga{\gamma} \def\Ga{\Gamma}  \def\ra{\rightarrow}\def\om{\omega} \def\OM{\Omega}
\def\ti{\tilde}     \def\z{\times}  \def\x{\cdot}
\def\y{\circ}   \def\ep{\varepsilon}\def\F{\Phi}    
\newcommand{\N}{\mathbb{N}}
\newcommand{\R}{\mathbb{R}}

\newcommand{\Bb}{\mathcal{B}}
\newcommand{\Nn}{\mathcal{N}}

\newcommand{\Tt}{\mathcal{T}}
\newcommand{\Mm}{\mathcal{M}}

\DeclareMathOperator{\diam}{diam}
\DeclareMathOperator{\dist}{distance}
\DeclareMathOperator{\length}{length}

\DeclareMathOperator{\vol}{vol}

\begin{document}
\begin{abstract}
In this paper, we show that for any closed 4-dimensional simply-connected Riemannian manifold $M$ with Ricci curvature $|Ric|\leq 3$, volume $\vol(M)>v>0$, and diameter $\diam(M)\leq D$, the length of a shortest closed geodesic is bounded by a function $F(v,D)$ which only depends on $v$ and $D$.

The proofs of our result are based on a recent theorem of diffeomorphism finiteness of the manifolds satisfying the above conditions proven by J.~Cheeger and A.~Naber.
\end{abstract}

\scititle{Length of a shortest closed geodesic in manifolds of dimension four}
\author{Nan Wu, Zhifei Zhu}
\date{\today}
\maketitle

\section{Introduction}
The first main result of this paper is the following theorem.
\begin{theorem}\label{thm1_l}
Let $M$ be a closed 4-dimensional simply-connected Riemannian manifold with Ricci curvature $|Ric|\leq 3$, volume $\vol(M)>v>0$, and diameter $\diam(M)\leq D$. Then the length of a shortest closed geodesic on $M$ is bounded by a function $F(v,D)$ which only depends on $v$ and $D$.
\end{theorem}

Let us denote by $\Mm(4,v,D)$ below the set of closed 4-dimensional simply-connected Riemannian manifolds with Ricci curvature $|Ric|\leq 3$, volume $\vol(M)>v>0$, and diameter $\diam(M)\leq D$.

\begin{remark}
If $M$ is not simply-connected, then one can always bound the length of the shortest closed geodesic by $2\cdot\diam(M)$. (See, for example, \cite{gromov2007metric}.) Therefore, we only consider the case where $M$ is simply-connected.
\end{remark}

In Theorem~\ref{thm1_l}, we do not have an explicit form for the function $F(v,D)$. The proof of our main theorem relies on an explicit construction of a covering of the manifold $M$ by harmonic balls and a certain type of contractible open sets. The construction of this covering is based on a theorem of diffeomorphism finiteness for manifolds in $\Mm(4,v,D)$, proved by J.~Cheeger and A.~Naber in \cite{cheeger2014regularity}.

The number of these sets in the covering, which plays an important role in our estimation, depends on the constants $\ep(v)$ and $r_0(v)$ in the following ``$\ep$-regularity'' theorem  \cite[Proposition~2.5]{anderson1990convergence}.

\begin{theorem}[\cite{anderson1990convergence}, Proposition~2.5]\label{epr}
Let $M\in \Mm(4,v,D)$ and $B(r)$, $r\leq D$ a geodesic ball in $M$. Then there are positive constants $\ep(v)$ and $r_0(v)$ such that if the curvature satisfies $\int_{B(2r)} |R|^2 <\ep$, then for all $x\in B$, the harmonic radius $r_h(x)$ at $x$ satisfies
$$\frac{r_h(x)}{\dist (x,\dl B)}\geq r_0>0.$$
\end{theorem}

\begin{remark}
The above ``$\ep-$regularity'' theorem holds for any dimension $n$. However, in our paper, we will only use the case of $n=4$.
\end{remark}

In the work \cite{cheeger2014regularity} of Cheeger and Naber, the authors are able to obtain a similar estimate without the integral of the curvature $\int|R|\leq \ep$ condition, using more advanced techniques developed in \cite{cheeger1997structure} and \cite{cheeger2014regularity}. We are going to introduce these results in Section 2. With these estimates about harmonic radius on Riemannian manifolds, we are able to improve our main theorem as the following.

\begin{theorem}\label{thm1_e}
Let $M\in\Mm(4,v,D)$. If for some $\ep(v)$ and $r_0(v)$, the manifold $M$ satisfies the above Theorem~\ref{epr}, then one can write down an explicit expression of $F$ in Theorem~\ref{thm1_l} in terms of $v$, $D$, $\ep$ and $r_0$.
\end{theorem}

Note that from the proof of the Theorem~\ref{epr} (see \cite[Section 2]{anderson1990convergence}), one may not obtain an explicit expression of the constant $r_0$ in terms of $\ep$. In fact, in the work of M.~Anderson \cite{anderson1989ricci}, one can explicitly estimate the constants $\ep$ and $r_0$ in terms of the local Sobolev constant and the second derivative of the Ricci curvature. And the Sobolev constant is explicitly estimated in terms of the volume in \cite{anderson1992thel}.

As a result, if the manifold is Einstein, then the second derivative of the Ricci curvature vanishes and the above $\ep$ is bounded by $C\cdot v^{-1/2}$, where $C$ is a constant that only depends on dimension. This leads to the following corollary which provides an explicit bound for the length of the shortest closed geodesic in Theorem~\ref{thm1_l}.

\begin{corollary}
Let $(M,g)$ be a closed 4-dimensional simply-connected Einstein manifold with Ricci curvature $Ric=kg$, where $-3\leq k \leq 3$ is a constant. Suppose that the volume $\vol(M)>v>0$, and the diameter $\diam(M)\leq D$. Then the length of a shortest closed geodesic on $M$ is bounded by an explicit function $F(v,D)$ which only depends on $v$ and $D$.
\end{corollary}

In this work, we will show first the existence of upper bound for the length of the shortest geodesic (Theorem~\ref{thm1_l}), and then, the existence of an explicit upper bound in terms of $v$, $D$, $\ep$ and $r_0$ (Theorem~\ref{thm1_e}). The proof of the Theorem~\ref{thm1_e} is much harder since we are not assuming any uniform lower bound on the radius of the harmonic balls in the covering.

The question of the length of a shortest closed geodesic was initially asked in the paper of M. Gromov in \cite{gromov1983filling}. Gromov asked whether the length of a shortest periodic geodesic in a $n$-dimensional Riemannian manifold $M^n$ can be bounded by $c(n)\vol (M)^{1/n}$. Similar question can also be asked for the diameter $D$ of the manifold. The fact that each closed Riemannian manifold has at least one closed geodesic was proved by L. Lusternik and A. Fet. (See, for example, \cite{klingenberg2012lectures}).

At present, there is no curvature-free upper bound for the length of the shortest closed geodesic on a general Riemannian manifold. However, various results have been obtained under certain geometric assumptions. (See \cite{croke1988area,maeda1994length,nabutovsky2002length,sabourau2004filling,rotman2006length,liokumovich2014lengths}
for the case of 2-spheres, \cite{treibergs1985estimates} for convex surfaces, \cite{ballmann1983existence} for spheres with $1/4-$pinched metric of positive curvature, and \cite{rotman2000upper,nabutovsky2003upper} for compact Riemannian manifold with sectional curvature bounded from below. Also \cite{croke2003universal} would be a nice introduction to readers who are not familiar with this topic). In this paper, we give an upper bound while assuming $M\in \Mm(4,v,D)$. Our theorem is the first result while assuming bounds on the Ricci curvature.

Let us briefly describe the idea of how to obtain an upper bound for the length of the shortest closed geodesic. Let $\OM_p M$ be the space of loops with fixed base point $p\in M$. For the smallest integer $m$ such that $\pi_{m+1}(M)\neq 0$, if one is able to construct a ``small'' non-contractible sphere of dimension $m$ in $\OM_p M$, in other words, a non-contractible map $S^m\ra \OM^L_p$, where $\OM^L_pM$ is the subspace of $\OM_pM$ whose points are loops of length $\leq L$, then by a standard Morse-type argument, there is a closed geodesic of length $\leq L$ occurred as a critical point of the length functional on the free loop space $\Lm M$. A.~Nabutovsky and R.~Rotman show in \cite{nabutovsky2013length} that the obstruction to these ``small'' non-contractible spheres are some ``short'' closed geodesics on the manifold. (See \cite[Corollary 5.4]{nabutovsky2013length}.)

More specifically, let us introduce the following definition of the depth of a loop (See \cite[Definition~7.1\&7.4]{nabutovsky2013length}.) and the width of a homotopy. We say that a smooth curve $\ga:[0,1]\ra M$ is a loop based at some point $p\in M$, if $\ga(0)=\ga(1)=p$.
\begin{definition}[Depth of a loop]\label{def3_d}
Let $M$ be a closed n-dimensional simply-connected Riemannian manifold with diameter $D$ and $\ga:S^1\ra M$ a loop in $M$ based at $p$. We define the depth $S(\ga)$ of $\ga$ to be the infimum of positive number $S$ such that $\ga$ is contractible by a path homotopy through loops of length $\leq \length(\ga)+S$.
We define $S_p(M,L)$ to be $\sup_{\length(\ga)\leq L} S(\ga)$, where the supremum is taken over all loops $\ga$ of length $\leq L$ based at $p$. In other words, $S_p(M,L)$ is the infimum of $S$ such that every loop $\ga$ of length $\leq L$ based at $p$ can be contracted by a homotopy through loops of length $\leq \length(\ga)+S$.
\end{definition}

\begin{definition}[Width of a homotopy]\label{def1_w}
Let $M$ be a Riemannian manifold and $\ga_i:[0,1]\ra M$, $i=1,2$, be two curves in $M$. Suppose $\ga_1$ and $\ga_2$ are homotopic and $H:[0,1]\z[0,1]\ra M$ is a homotopy between $\ga_1$ and $\ga_2$. For every fixed $s\in [0,1]$, the notation $$H_s:=H(s,\cdot):[0,1]\ra M$$
is a curve in $M$ which describes the trajectory of a point $H(s,0)$ during the homotopy. We define the width $\om_H$ of the homotopy $H$ to be
$$\om_H=\max_{s\in [0,1]} \text{length of } H_s.$$
\end{definition}

In \cite{nabutovsky2013length}, by taking the base point $p=q=x$ in Theorem $7.3$ and applying Corollary $5.4$, Nabutovsky and Rotman proved that

\begin{theorem}\label{thm4_q}
Let $M^n$ be a closed Riemannian manifold of diameter $D$ and $p$ be a point in $M$, and $S\geq0$. Assume that there exists $k\in\N$ such that there is no geodesic loop of length in $((2k-1)D,2kD]$ based at $p$ which is a local minimum of the length functional on $\OM_{p} M$ of depth $>S$. Then for every positive integer $m$ every map $f:S^m\ra \OM_{p}M$ is homotopic to a map $\ti{f}:S^m\ra \OM_{p}^{L+o(1)}M$, where $L=((4k+2)m+(2k-3))D+(2m-1)S$.

In this case, the length of a shortest closed geodesic on $M$ does not exceed $L=((4k+2)m+(2k-3))D+(2m-1)S$.
\end{theorem}

An important observation in \cite{nabutovsky2013length} is that the depth of $\ga$ is related to the width of an optimal homotopy contracting $\ga$. In fact, we have
\begin{theorem}
If for any closed curve $\ga$ of length bounded by $L$, there exists a contraction of $\ga$ with width bounded by some constant $W$, then
$$S_p(M,L)\leq \max \{2L,2W+2D\}.$$
\end{theorem}

The proof of this inequality can be found in \cite{nabutovsky2003upper} or \cite[Section 8]{nabutovsky2013length}. This observation allows us to convert the problem of obtaining an upper bound for the length of a shortest geodesic in $M$ to the problem of estimating the width of an optimal homotopy contracting any curve $\ga$ in $M$. And in the case of $M\in \Mm(4,v,D)$, for every curve $\ga\su M$, we will prove that one can always contract $\ga$ to a point through a homotopy with controlled width.  Our construction is based on the work \cite{cheeger2014regularity} of Cheeger and Naber, where they constructed a ``bubble tree'' decomposition for the manifolds in $\Mm(4,v,D)$. We are going to describe this decomposition in Section~\ref{sec2}.

In conclusion, in order to obtain an upper bound for the length of the shortest geodesic, we are going to prove that

\begin{theorem}\label{thm2_w}
Let $M\in\Mm(4,v,D)$. Then we have:
\begin{itemize}
\item[A.] There exists an increasing function $W(v,D)$ which only depends on $v$ and $D$ such that any closed curve $\ga:S^1\ra M$ can be contracted to a point through a homotopy with width $\om_H\leq W (v,D)$.
\item[B.] If we further assume that there are no non-trivial closed geodesics on $M$ with length bounded by $4D$ and $M$ satisfies the ``$\ep$-regularity'' Theorem~\ref{epr} for some constants $\ep$ and $r_0$, then one can write down an explicit expression of $W$ in terms of $v$, $D$, $\ep$ and $r_0$.
\end{itemize}
\end{theorem}

In fact, from the prove of the above theorem, we see that Theorem~\ref{thm2_w}B is true in any dimension as long as the manifold $M$ satisfies \cite[Theorem~8.6]{cheeger2014regularity}. In Anderson's work \cite[Theorem~2.6]{anderson1990convergence}, in the case of dimension $n$, if one assume that the integral of curvature satisfies
\begin{equation}\label{equ2}
\int_M |R|^{n/2} dV \leq C,
\end{equation}
then one can still obtain the result of \cite[Theorem~8.6]{cheeger2014regularity}. By applying \cite[Lemma~8.61]{cheeger2014regularity}, one can obtain a similar bubble tree decomposition for $M$ as in \cite[Theorem~8.64]{cheeger2014regularity}. Therefore, our main Theorem~\ref{thm1_l} can be generalized as following.
\begin{theorem}
Let $M$ be a closed $n$-dimensional simply-connected Riemannian manifold with Ricci curvature $|Ric|\leq n-1$, volume $\vol(M)>v>0$, and diameter $\diam(M)\leq D$. Suppose that the curvature tensor $R$ of $M$ satisfies (\ref{equ2}), then the length of the shortest closed geodesic on $M$ is bounded by a function $F(v,D)$ which only depends on $v$ and $D$.
\end{theorem}

The idea of the proof of Theorem~\ref{thm2_w} is the following. Given a closed contractible curve $\ga:[0,1]\ra M$, we would like to first contract $\ga$ through a family of curves $\{\ga_j\}$ so that the width of the homotopy between each $\ga_j$ and $\ga_{j+1}$ is bounded in terms of $D$. If the number of the curves in the family $\{\ga_j\}$ is bounded in terms of $v$ and $D$, then we are done. However, in general, the number of the curves is not related to $v$ and $D$. Therefore, we are going to construct a new homotopy through bounded number of curves.

The observation is that by the result of Cheeger and Naber, we may cover the manifold $M$ by finitely many harmonic balls and some (thin) contractible sets. We are going to construct {a graph $\Si$}, which is essentially the 1-skeleton of the nerve of this covering, so that we can find the approximations of the curves $\ga_j$ in this graph $\Si$ with bounded length. Here the approximation of a curve $\ga_j$ in $\Si$ means a homotopy between $\ga_j$ and a curve in $\Si$ with controlled width.

Now for any homotopy that contracts the curve $\ga$, we can find an approximation of this homotopy by looking at the approximation of the curves during this homotopy. The new ``optimal'' homotopy can be obtained by removing the curves with the same approximations in the graph. And then the total number of the curves is bounded in terms of the number of the curves in {the graph $\Si$}, which can be estimated by the number $\ti{N}(v,D)$ of the sets in this covering of $M$.

The difficult part of the proof is to bound the length, or more precisely, the "simplicial length" (see Definition~\ref{def4_s}) of the approximation of the curve $\ga_j$, because, for example, there is no lower bound for the radius of the harmonic balls in the covering of the manifold. In other words, if we are trying to approximate a curve with some short geodesic segments, we may end up with an uncontrolled number of the segments in the approximation.

To solve this problem, our observation is that during the homotopy, if we decompose a curve $\ga$ into a wedge $\vee_{i}\al_i$ of some curves $\al_i$ with a fixed base point and let $G_i$ be the contraction of each $\al_i$, then the width contracting $\ga$ is bounded by $2\cdot \max_i \om_{G_i}$ (See Lemma~\ref{lm8}). In this case, we only need to bound the length of the approximation of each $\al_i$, instead of the entire curve $\ga$. We are going to show in Lemma~\ref{lm6} and Lemma~\ref{lm7} that there is a desired decomposition of the curve $\ga$ so that we can control the length of the approximation of the curves in $\Si$.

\subsection{Structure of this paper.}
In Section~\ref{sec2}, we are going to introduce some definitions and results about non-collapsing manifolds with bounded diameter and Ricci curvature in \cite{cheeger2014regularity}. We will be focusing on the case of dimension 4. We are also going to show some elementary results about the contractibility of certain metric balls which will be used in the rest of our proof.

In Section~\ref{sec3}, we will first construct a {graph} and develop a certain type of the approximation of homotopies in this {graph}  as we mentioned above. We will then show several results about the upper bound of the length of the different type of the curves in the approximation. Some techniques we used in Lemma~\ref{lm1} to Lemma~\ref{lm5} are due to R.~Rotman and her work \cite{rotman2000upper}.

In the last section, we will prove our main results Theorem~\ref{thm2_w}A, B, and Theorem~\ref{thm1_l}. The proof of Theorem~\ref{thm2_w}A and B will be separated and will be based on different methods.

\section{Harmonic radius and finite diffeomorphism type theorem in dimension 4}\label{sec2}
In this section we introduce some definitions and results about non-collapsing manifolds with bounded diameter and Ricci curvature in \cite{cheeger2014regularity}, which will be used to proof our main results Theorem~\ref{thm1_l} and Theorem~\ref{thm2_w}. Note that their work is based on theory of manifolds with Ricci curvature bounded below developed by J.~Cheeger and T.~Colding \cite{colding1997ricci}, \cite{cheeger1996lower}, \cite{cheeger1997structure} and work of M.~Anderson \cite{anderson1989ricci}, \cite{anderson1992thel} and Cheeger and Anderson \cite{anderson1991diffeomorphism}.

We first recall the notion of the harmonic radius. (See \cite[Definition~2.9]{cheeger2014regularity} or \cite[Chapter 10.5]{petersen2006riemannian}.)

\begin{definition}\label{def2_h}
Let $M^n$ be an $n-$dimensional Riemannian manifold and $x$, a point in $M$. We define the harmonic radius $r_h(x)$ to be the largest $r>0$ such that there exists a map $\F: B_r(0^n)\ra M$, where $0^n\in \R^n$ is the origin, such that:
\begin{enumerate}
\item $\F$ is a diffeomorphism onto its image with $\F(0^n)=x$.
\item $\De_g x^l=0$, $l=1,\dots,n$, where $x^l$ are the coordinate functions and $\De_g$ is the Laplace-Beltrami operator.
\item\label{equ1} If $g_{ij}=\F^*(g)$ is the pullback metric on $B_r(0^n)$, then
$$
||g_{ij}-\de_{ij}||_{C^0,B_r(0^n)}+r||\dl_k g_{ij}||_{C^0,B_r(0^n)}\leq 10^{-3}.
$$
\end{enumerate}
\end{definition}

The above map $\F:B_r(0^n)\ra M$ is also called a {\sl harmonic coordinate}. The condition (3) above tells us that $\F$ is a lipschitz map with the lipschitz constant bounded by $1.001$. Therefore, we are able to estimate the ``contractibility radius'' at $x$ in terms of the harmonic radius by the following lemma.

\begin{lemma}\label{lem1_c}
Let $M$ be a Riemannian manifold and $x\in M$. Suppose $r_h(x)>0$ is the harmonic radius at $x$. Let $R(x)=\frac{1}{8\sqrt{1+2\cdot10^{-3}}}\x r_h(x)$. Then the metric ball $B_{R(x)}(x)$ is contractible in $M$.

Furthermore, for any closed curve $\ga:[0,1]\ra B_{R(x)}(x)$, there exists a contraction $H:[0,1]\z [0,1]\ra M$ such that $H(\cdot,0)=\ga$, $H(\cdot,1)=x$ and the width of the homotopy $\om_H\leq D$.
\end{lemma}

\begin{proof}
Let $\F: B_{r_h(x)}(0^n)\ra M$ be the harmonic coordinate at $x\in M$ such that $x=\F(0^n)$. Let $p\in \dl \F(\overline{B_{r_h(x)/2}(0^n)})$ be the point realizing the minimum distance between $x$ and the boundary of the closure $\dl\F(\overline{B_{r_h(x)/2}(0^n)})$. We connect $p$ and $x$ by a minimizing geodesic $\ga$. Note that $\ga$ must be contained in $\F(\overline{B_{r_h(x)/2}(0^n)})$.

By \eqref{equ1} in Definition~\ref{def2_h}, the length of $\ga$ satisfies
$$\length(\ga)\geq \frac{\length(\F^{-1}(\ga))}{\sqrt{1+2\x 10^{-3}}}\geq \frac{r_h(x)}{2\sqrt{1+2\x 10^{-3}}}.$$
Let $R(x)=\frac{r_h(x)}{8\sqrt{1+2\x 10^{-3}}}$. We show that the ball $B_{R(x)}(x)$ can be contracted to $x$ within the ball $\F(\overline{B_{r_h(x)/2}(0^n)})$. Indeed, let $k:\overline{B_{r_h(x)/2}(0^n)}\z [0,1]\ra \overline{B_{r_h(x)/2}(0^n)}\su \R^n$ be the contraction defined by $k(y,t)=yt$. Then $H=\F\y k\y(\F^{-1}\z id)$ is a homotopy contracting $\F(\overline{B_{r_h(x)/2}(0^n)})$ to $x\in M$. For any $y\in \overline{B_{r_h(x)/2}(0^n)}$, the length of the trajectory satisfies
$$\length(\F\y k(y,\cdot))\leq \length\left(\F\y k\left(\frac{r_h(x)\x y}{2|y|},\cdot\right) \right)\leq\frac{r_h(x)}{2\sqrt{1-10^{-3}}}\leq r_h(x) \leq D.$$

Note that $\overline{B_{R(x)}(x)}\su \F(\overline{B_{r_h(x)/2}(0^n)})$. We restrict the homotopy $H$ to $B_{R(x)}(x)$ and the width $\om_H\leq D$.
\end{proof}

In \cite{cheeger2014regularity}, J.~Cheeger and A.~Naber proved the finiteness of the number of diffeomorphism type of manifolds $M$ of dimension $4$ with $|Ric_M|\leq3$, $\vol(M)>v>0$ and $\diam(M)<D$. This theorem is based on the construction of the ``bubble tree'' decomposition of the manifolds (\cite[Theorem 8.64]{cheeger2014regularity}), which decomposes $M$ into a union of body regions and neck regions. The proofs of our main results are also based on this construction. Therefore, let us briefly describe this process below. We first start with the construction of a body region.

Up to rescaling, we cover the manifold $M$ by metric balls $\{B_1(x_i)\}$ such that the balls in $\{B_{1/4}(x_i)\}$ are pairwise disjoint. By a standard volume comparison argument, there are at most $N_0(v,D)$ such balls. In each ball $B_1(x_i)$, there exist scales $r_j^1>r_0(v,D)$, an integer $N_1\leq N(v,D)$ and a collection of balls $\{B_{r_j^1}(x_j^1)\}_{j=1}^{N_1}$ such that
if $x\in B_1(x_i)\setminus \cup_j B_{r_j^1}(x_j^1)$, then the harmonic radius $r_h(x)\geq r_0(v,D)$. Here $r_0$ and $N$ are some constants that only depend on $v$ and $D$. Furthermore, the balls $\{ B_{2r_j^1}(x_j^1)\}$ are disjoint. In total, there are at most $N_0\x N$ such balls.

We define the first body $\Bb^1=M\setminus \cup_j B_{r_j^1}(x_j^1)$. Note that the manifold $M=\Bb^1\cup (\cup_j B_{2r_j^1}(x_j^1))$. Next, we construct the first neck region. In $B_{2r_j^1}(x_j^1)$, there is a scale $\bar{r}_j^1$, and $\ep(v)<0.1$, such that there is a neck region neck $\Nn_j^2$ satisfying
\begin{equation*}
A_{\bar{r}_j^1 /2, 2 r_j^1} (x_j^1)   \subset \Nn_j^2 \subset A_{(1-\ep)\bar{r}_j^1 /2, 2(1+\ep) r_j^1} (x_j^1),
\end{equation*}
where $A_{r,R} (x_j^1)$ is a metric annulus centered at $x_j^1$ in $M$. As proved in Theorem~8.6 and Lemma~8.40 in \cite{cheeger2014regularity}, the geometry of these $\Nn_j^2$ are controlled. In other words, there is a diffeomorphism $\Phi_j^2: A_{\bar{r}_j^1/2,2r_j^1}(0)\ra \Nn_j^2$, where $A_{\bar{r}_j^1/2,2r_j^1}(0)$ is an annulus centered at $0\in \R^4/\Ga_j^2$ for some finite discrete subgroup $\Ga_j^2\su O(4)$. And if $g_{ij}={\Phi_j^2}^*g$ is the pullback metric, then
\begin{equation*}
||g_{ij}-\de_{ij}||_{C^0}+\bar{r}_j^1\cdot|| \dl g_{ij}||_{C^0}\leq \ep(v)<0.1
\end{equation*}
The order of $|\Ga_j^2|$ is bounded by a function $C(v,D)$ which only depends on $v$ and $D$.

We repeat the above construction to each ball $B_{2\bar{r}_j^1}(x_j^1)$ and we define the second body regions $\Bb_j^{2}=B_{2\bar{r}_j^1}(x_j^1)\setminus \cup_i B_{r_i^{2}}(x_i^{2})$ and the second neck region $\Nn_j^2$ that connects $\Bb_j^2$ and $\Bb^1$.

In general, we have the bodies
\begin{equation}\label{equ3}
\Bb_j^{k+1}=B_{2\bar{r}_j^k}(x_j^k)\setminus \cup_i B_{r_i^{k+1}}(x_i^{k+1}),
\end{equation}
such that when $x\in \Bb_j^{k+1}$, then $r_h(x)\geq r_0(v,D)\x \diam(\Bb_j^{k+1})$.
And the neck region $\Nn_j^{k+1}$ that connects the body $\Bb_j^{k+1}$ and  $\Bb_i^{k}$, which satisfy
\begin{equation}\label{equ4}
A_{\bar{r}_j^k /2, 2 r_j^k} (x_j^k)   \subset \Nn_j^{k+1} \subset A_{(1-\ep)\bar{r}_j^k /2, 2(1+\ep) r_j^k} (x_j^k),
\end{equation}
and there is a diffeomorphism
\begin{equation}\label{equ4.1}
\Phi_j^{k+1}: A_{\bar{r}_j^k/2,2r_j^k}(0)\ra \Nn_j^{k+1}
\end{equation}
where $0\in \R^4/\Ga_j^{k+1}$ with $\Ga_j^{k+1}\su O(4)$ satisfying
\begin{equation}\label{equ4.5}
|\Ga_j^{k+1}| \leq C(v,D).
\end{equation}

Moreover, if $g_{ij}={\Phi_j^{k+1}}^*g$ is the pullback metric, then
\begin{equation}\label{equ5}
||g_{ij}-\de_{ij}||_{C^0}+\bar{r}_j^k\cdot|| \dl g_{ij}||_{C^0}\leq \ep(v)<0.1.
\end{equation}

 The reason why this construction ends in finitely many steps is because if for some indices $j,k,l$, the intersection $\Nn_l^{k+1}\cap \Bb_j^k\neq \emptyset$, then $|\Ga_l^{k+1}|\leq|\Ga_j^k|-1$. Therefore, after at most $|\Ga_j^2|\leq C(v,D)$ many steps, this process ends. As a result, we have the following decomposition theorem.

\begin{theorem}[\cite{cheeger2014regularity}, Theorem 8.64]\label{thm3_f}
Let $M$ be a 4-dimensional Riemannian manifold with $|Ric|\leq 3$, $\vol(M)>v>0$ and $\diam(M)\leq D$. Then $M$ admits a decomposition into bodies and necks
$$M=\Bb^1 \cup \bigcup_{j_2=1}^{N_2} \Nn_{j_2}^2\cup \bigcup_{j_2=1}^{N_2} \Bb_{j_2}^2\cup\dots\cup \bigcup_{j_k=1}^{N_k} \Nn_{j_k}^k\cup \bigcup_{j_k=1}^{N_k} \Bb_{j_k}^k,$$

such that the following conditions are satisfied:
\begin{enumerate}
\item If $x\in \Bb_i^j$, then $r_h(x)\geq r_0(v,D)\cdot \diam (\Bb_i^j)$, where $r_h$ is the harmonic radius and $r_0$ is a constant that only depends on $v$ and $D$.
\item Each $\Nn_i^j$ is diffeomorphic to $\R\z S^3/\Ga_i^j$ for some $\Ga_i^j\su O(4)$ with the order $|\Ga_i^j|<C(v,D)$.
\item $\Nn_i^j\cap \Bb_i^j$ is diffeomorphic to $\R\z S^3/\Ga_i^j$.
\item $\Nn_i^j\cap \Bb_{i'}^{j-1}$ is either empty or diffeomorphic to $\R\z S^3/\Ga_i^j$.
\item Each $N_i\leq N(v,D)$ and $k\leq k(v,D)$.
\end{enumerate}
\end{theorem}

\begin{remark}
In the statement of the above theorem, the constants $r_0(v,D)$, $N(v,D)$, $k(v,D)$ and $C(v,D)$ can be explicitly computed in terms of the constants in the ``$\ep$-regularity'' theorem \cite[Theorem2.11]{cheeger2014regularity} for manifolds in $\Mm(4,v,D)$.
\end{remark}

\begin{remark}\label{rm1}
For each neck $\Nn_j^{k+1}$, the ratio between the inner and outer radius of the annulus $r_j^k / \bar{r}_j^k$ may not be bounded above by any function of $v$ and $D$. Hence one may not cover a neck region with contractible metric balls described in Lemma~\ref{lem1_c} so that the number of balls in the covering is bounded above by a function of $v$ and $D$.
\end{remark}

Based on Theorem~\ref{thm3_f}, we are going to construct an open covering of $M$ so that the total number of the open sets in the covering is bounded by some function that only depends on $v$ and $D$. First note that each body $\Bb_j^k$ is covered by finitely many contractible balls as described in Lemma~\ref{lem1_c}. However, as described in Remark~\ref{rm1}, the metric annulus $A_{2\bar{r}_j^k, r_j^k} (x_j^k)$ in the neck region cannot be covered in the same way as the body regions. Instead, we are going to cover it by some trapezoids, such that each trapezoid is contractible in some larger trapezoids, which will be defined below.
\begin{definition}
For each neck $\Nn_j^{k+1}$, where $k\geq 1$, let $r_c(k+1,j)$ be the convexity radius of $S^3 / \Ga_j^{k+1}$ equipped with the standard metric $ds^2_{k+1,j}$. We cover $S^3 / \Ga_j^{k+1}$ by $B_{r_c(k+1,j)/4}(z_i)$ with $z_i \in S^3 / \Ga_j^{k+1}$ in an efficient way, so that the balls $B_{r_c(k+1,j)/16}(z_i)$ are pairwise disjoint. We define
\begin{align}
\label{equ5.1}& K^{k+1}_{j,i}=(2\sqrt{1-\ep}\bar{r}_j^k,(2-\sqrt{1-\ep})r_j^k) \times  B_{r_c(k+1,j)/4}(z_i), \\
\label{equ5.2}& \bar{K}^{k+1}_{j,i}=(\bar{r}_j^k/2,2r_j^k) \times  B_{r_c(k+1,j)}(z_i),
\end{align}
with the metric $dg^2_{k+1,j}=dr^2+r^2 ds^2_{k+1,j}$, where $r \in (\bar{r}_j^k/2,2r_j^k) $ and $\ep=\ep(v)<0.1$.
\end{definition}
Now the annulus $A_{\bar{r}_j^k/2,2r_j^k}(0)$ is covered by the open sets $\{ \bar{K}^{k+1}_{j,i} \}$, where $ 0\in \R^4/\Ga_j^{k+1}$. Moreover, $\bar{K}^{k+1}_{j,i}$ is a convex open subset of $A_{\bar{r}_j^k/2,2r_j^k}(0)$.

\begin{definition}
We define a trapezoid in $M$ to be
\begin{equation}\label{equ6}
T^{k+1}_{j,i} = \Phi_j^{k+1}(K^{k+1}_{j,i}).
\end{equation}
and a large trapezoid
\begin{equation}\label{equ7}
\bar{T}^{k+1}_{j,i} = \Phi_j^{k+1}(\bar{K}^{k+1}_{j,i})
\end{equation}
where $\Phi_j^{k+1}$ is the diffeomorphism~\eqref{equ4.1}.
\end{definition}

Note that since different necks are disjoint, trapezoids in different necks do not intersect. Moreover if $T^{k+1}_{j,i} \cap T^{k+1}_{j,l} \neq \emptyset$,  then $T^{k+1}_{j,i} \cup T^{k+1}_{j,l} \subset \bar{T}^{k+1}_{j,i}$.

Our next Lemma can be viewed as an analogue of Lemma~\ref{lem1_c} for the trapezoids.

\begin{lemma} \label{trapezoid contraction}
Let $T^{k+1}_{j,i}\su \bar{T}^{k+1}_{j,i}\su \Nn_j^{k+1} \su M$ be trapezoids defined in \eqref{equ6} and \eqref{equ7}. Then,
\begin{enumerate}
\item
Any two points $x$ and $y$ in $T^{k+1}_{j,i}$ can be connected by a curve in $T^{k+1}_{j,i}$ with length less than $3D$. Any two points $x$ and $y$ in $\bar{T}^{k+1}_{j,i}$ can be connected by a curve in $\bar{T}^{k+1}_{j,i}$ with length less than $3D$.
\item
For any two points $x$ and $y$ in $\bar{T}^{k+1}_{j,i}$, let $\ga$ be a minimizing geodesic connecting $x$ and $y$ in $M$. If $\ga$ is contained in the neck $\Nn_j^{k+1}$, then we can connect $x$ and $y$ by a curve in $\bar{T}^{k+1}_{j,i}$ with length less than $2 \length(\ga)$.
\item
For any closed curve $\ga:[0,1]\ra \bar{T}^{k+1}_{j,i}$, where $\ga(0)=\ga(1)=p$, there exists a contraction $H:[0,1]\z [0,1]\ra \bar{T}^{k+1}_{j,i} $ of $\ga$ to $p$, such that  the width of the homotopy $\om_H\leq 21D$ and $p$ is fixed during the contraction.
\end{enumerate}
\end{lemma}

\begin{proof}
\

\begin{enumerate}
\item Note that any two points in $K^{k+1}_{j,i}$ can be connected by a curve with length less than $5 r_j^k$. By equation~\eqref{equ5} and the fact that $2r_j^k \leq D$, we conclude that any two points in $T^{k+1}_{j,i}$ can be connected by a curve with length less than $\frac{5r_j^k}{\sqrt{1-\epsilon(v)}} \leq \frac{5r_j^k}{\sqrt{1-0.1}} \leq 3D$. If two points are in $\bar{T}^{k+1}_{j,i}$, the proof is similar.

\item Now suppose that a minimizing geodesic $\ga$ in $M$ connecting $x$ and $y$ is contained in the neck $\Nn_j^{k+1}$. Then $(\Phi_j^{k+1})^{-1}(\ga)$ is a curve connecting $(\Phi_j^{k+1})^{-1}(x)$ and $(\Phi_j^{k+1})^{-1}(y)$ in $A_{\bar{r}_j^k/2,2r_j^k}(0)$ with length less than $\sqrt{1+2 \ep(v)} \length(\ga)$.  Since $\bar{K}^{k+1}_{j,i}$ is convex in $A_{\bar{r}_j^k/2,2r_j^k}(0)$, there is a curve $\ga'$ in $\bar{K}^{k+1}_{j,i}$ connecting  $(\Phi_j^{k+1})^{-1}(x)$ and $(\Phi_j^{k+1})^{-1}(y)$ with length less than the length of $(\Phi_j^{k+1})^{-1}(\ga)$. Therefore, $\Phi_j^{k+1}(\ga')$ is a desired curve with length
\begin{equation*}
\begin{split}
\length(\Phi_j^{k+1}(\ga')) &\leq \frac{1}{\sqrt{1-\ep(v)}} \length([\Phi_j^{k+1}]^{-1}(\ga)) \\
 &\leq \frac{\sqrt{1+2 \ep(v)}}{\sqrt{1-\ep(v)}}\length(\ga) < 2\length(\ga).
\end{split}
\end{equation*}

\item For $(\bar{K}^{k+1}_{j,i}, dg^2_{k+1,j})$, let us consider a homotopy $F(t,x): [0,1] \times \bar{K}^{k+1}_{j,i} \ra \bar{K}^{k+1}_{j,i}$ with $F(0,x)=x$ and $F(1,x)=(2r_j^k, z_i)$ defined in the following way.

For $t \in [0, 1/2]$, we define $F(t,x)$ to be the deformation retraction of $\bar{K}^{k+1}_{j,i}$ onto $2r_j^k \times B_{r_{k+1,j}}(z_i)$. And for $t \in [1/2, 1]$, we define $F(t,x)$ to be a retraction of
$2r_j^k \times B_{r_{k+1,j}}(z_i)$ to $(2r_j^k, z_i)$ induced by the exponential map at $z_i \in S^3 / \Ga_j^{k+1}$ with the metric $(2r_j^k)^2 ds^2_{k+1,j}$.  Hence, for any closed curve in $ \bar{K}^{k+1}_{j,i}$, $F$ induces a contraction with width less than $4r_j^k$.  Consider $F_1=\Phi_j^{k+1} \circ F \circ (id \times [\Phi_j^{k+1}]^{-1}): [0,1] \times \bar{T}^{k+1}_{j,i} \ra \bar{T}^{k+1}_{j,i}$. For any closed curve $\ga$ in $\bar{T}^{k+1}_{j,i}$, the homotopy $F_1$ induces a contraction $H$ of $\ga$ to $q=\Phi_j^{k+1}(2r_j^k, z_i)$ such that
$$\om_H\leq \frac{4r_j^k}{\sqrt{1-\epsilon(v)}} \leq \frac{4r_j^k}{\sqrt{1-0.1}} \leq 3D.$$

Now suppose that $\ga$ is a curve in $\bar{T}^{k+1}_{j,i}$ and $p$ is point on $\ga$. In Step 1, the point $p$ is not fixed during the contraction.   We will describe a new homotopy by describing the image of the curve $\ga$ under the homotopy such that $p$ is fixed during the homotopy. Let $\si \su \bar{T}^{k+1}_{j,i}$ be the curve from $p$ to $q$ as described in (1).  $\length(\si) \leq 3D$.  Then $\ga$ is homotpic to $\si \cup [(-\si) \cup \ga \cup \si] \cup (-\si) $ with width bounded by $12D$. By Step 1, $\si \cup [(-\si) \cup \ga \cup \si] \cup (-\si) $ is homotopic to $\si \cup (-\si)$ with width bounded by $3D$. $\si \cup (-\si)$ is homotopic to $p$ with width bounded by $6D$. Hence the width of the contracton is $12D+3D+6D=21D$.
\end{enumerate}
\end{proof}

In the next lemma, we show that one can cover the manifold $M$ by contractible open balls and open trapezoids defined in equation (\ref{equ6}), such that the total number of the open sets in this covering is bounded by a function of $v$ and $D$.

\begin{lemma}\label{lem2_c}
Let $M$ be a 4-dimensional manifold that satisfies $|Ric|\leq 3$, $\vol(M)>v>0$ and $\diam(M)\leq D$ with the ``bubble tree'' decomposition
$$M=\Bb^1 \cup \bigcup_{j_2=1}^{N_2} \Nn_{j_2}^2\cup \bigcup_{j_2=1}^{N_2} \Bb_{j_2}^2\cup\dots\cup \bigcup_{j_k=1}^{N_k} \Nn_{j_k}^k\cup \bigcup_{j_k=1}^{N_k} \Bb_{j_k}^k,$$
as in Theorem~\ref{thm3_f}. Then $M$ admits a covering $\mathcal{O}$ that consists of contractible metric balls $\{B_{r(x_j)}(x_j)\}$ and trapezoids $\{T^k_{j,i}\}$ such that
\begin{enumerate}
  \item Each body region $\Bb_i^k$ is covered by some metric balls $\{B_{r(x_j)}(x_j)\}$, where $x_j\in \Bb_i^k$, $r(x_j)=\frac{1}{32\sqrt{1+2\cdot10^{-3}}}\x r_{h}(x_j)$ and $r_{h}(x_j)$ is the harmonic radius at $x_j$.
  \item Each neck region $A_{2\bar{r}_j^k, r_j^k} (x_j^k) \su \Nn_i^{k+1}$ is covered by some trapezoids $\{T^{k+1}_{j,i}\}$ which is defined in the equation~\eqref{equ6}.
  \item The total number of the open sets in $\mathcal{O}$ is bounded by some function $\ti{N}(v,D)$.
\end{enumerate}
\end{lemma}

\begin{proof}
Let us first consider the body regions in the decomposition of $M$. For each body $\Bb_j^k$, let $p\in\Bb_j^k$ and $d_j=\diam(\Bb_j^k)$.
let $U_j=\{x_i\}\su \Bb_j^k$ be a maximal subset such that $\{B_{R(x_i)/40}(x_i)\}$ are disjoint, $\Bb_j^k\su \cup_{i}B_{R(x_i)/4}(x)$ and $\cup_i B_{R(x_i)/40}\su B_{2d_j}(p)$, where $R(x_i)$ is the function defined in Lemma~\ref{lem1_c}.

We claim that the number of the elements in $U_j$ is bounded by a function $N_j(r_0)$, where $r_0=r_0(v,D)$ is the constant in Theorem~\ref{thm3_f}. Indeed, let $\vol_{-1}B_r$ denote the volume of a metric ball of radius $r$ in a 4-dimensional hyperbolic space of constant sectional curvature $\kappa=-1$. Let $\ep=R(x_i)/4$. Then, by Bishop-Gromov volume comparison theorem (See, for example, \cite{petersen2006riemannian}),
$$\frac{\vol(B_{10d_j}(x_i))}{\vol_{-1}B_{10d_j}}\leq\frac{\vol(B_{\ep/10}(x_i))}{\vol_{-1}B_{\ep/10}}.$$
And hence,
$$\#|U_j|\leq \frac{\sum\vol(B_{10d_j}(x_i))}{\vol(B_{2d_j}(p))}\leq \frac{\vol_{-1}B_{10d_j}}{\vol_{-1}B_{\ep/10}} \leq N_1(d_j/r_h).$$
By Theorem~\ref{thm3_f}, for any $x_i\in \Bb_j^k$, the ratio $d_j/r_h(x_i)<r_0(v,D)$. Therefore, we conclude that the number of the elements $\#|U_j|\leq N_1(r_0(v,D))$. Because the number of the $\Bb_j^k$ in the decomposition is bounded by $N(v,D) \times k(v,D)$, taking
 $$r(x_j)=R(x_j)/4=\frac{1}{32\sqrt{1+2\cdot10^{-3}}}\cdot r_h(x_j),$$
then the total number of the balls in the covering $\{B_{r(x_j)}(x_j)\}$ constructed above is bounded by $N(v,D) \times k(v,D) \times N_1(r_0(v,D))$.

Next we will show that for each annulus in the neck region, we have
\begin{equation}\label{intersection properties}
A_{2\bar{r}_j^{k}, r_j^{k}} (x_j^{k}) \su \cup_i T^{k+1}_{j,i} \su A_{(2-3\ep)\bar{r}_j^{k},(1+\frac{7}{2}\ep) r_j^{k}} (x_j^{k}).
\end{equation}
We first prove $A_{2\bar{r}_j^{k}, r_j^{k}} (x_j^{k}) \su \cup_i T^{k+1}_{j,i}$. Since $\cup_i K^{k+1}_{j,i}$ covers the annulus $A_{2\sqrt{1-\ep}\bar{r}_j^k,(2-\sqrt{1-\ep})r_j^k}(0)$, where $ 0\in \R^4/\Ga_j^{k+1}$, it suffices to prove that
$$(\Phi_j^{k+1})^{-1}(A_{2\bar{r}_j^{k}, r_j^{k}} (x_j^{k})) \su A_{2\sqrt{1-\ep}\bar{r}_j^k,(2-\sqrt{1-\ep})r_j^k}(0).$$
Let $S_r(0)$ be the sphere of radius $r$ in $\R^4$ and $S_r(x)$ be the sphere of radius $r$ at $x \in M$. We show that
\begin{align}
 & \dist_{\R^4/\Ga_j^{k+1}} \big((\Phi_j^{k+1})^{-1}(S_{2\bar{r}_j^k}(x_j^k)),S_{\bar{r}_j^k/2}(0)/ \Ga_j^{k+1}\big)  \label{inter radius} \\
\geq & 2\sqrt{1-\ep}\bar{r}_j^k -\bar{r}_j^k/2 = \dist_{\R^4/\Ga_j^{k+1}}\big(S_{2\sqrt{1-\ep}\bar{r}_j^k}(0)/ \Ga_j^{k+1},S_{\bar{r}_j^k/2}(0)/ \Ga_j^{k+1}\big).\nonumber
\end{align}
If $\ga$ is a curve that realizes the  $\dist_{\R^4/\Ga_j^{k+1}}\big([\Phi_j^{k+1}]^{-1}(S_{2\bar{r}_j^k}(x_j^k)),S_{\bar{r}_j^k/2}(0)/ \Ga_j^{k+1}\big) $ in $ \R^4/\Ga_j^{k+1}$. Then $\Phi_j^{k+1}(\ga) \su \Nn_j^k \cap \overline{B_{2\bar{r}_j^k}(x_j^k)}$. By equation (\ref{equ4}), we have
$$\length(\ga) \geq \sqrt{1-\ep} \cdot\length(\Phi_j^{k+1}(\ga)) \geq \sqrt{1-\ep} (2\bar{r}_j^k-\bar{r}_j^k/2) \geq 2\sqrt{1-\ep}\bar{r}_j^k -\bar{r}_j^k/2$$ Similarly, we have
\begin{align}\label{outer radius}
& \dist_{\R^4/\Ga_j^{k+1}}\big([\Phi_j^{k+1}]^{-1}(S_{r_j^k}(x_j^k)),S_{2 r_j^k}(0)/ \Ga_j^{k+1}\big)  \\
& \geq \dist_{\R^4/\Ga_j^{k+1}}\big([\Phi_j^{k+1}]^{-1}(S_{(2-\sqrt{1-\ep})r_j^k}(0)), S_{2 r_j^k}(0)/ \Ga_j^{k+1}\big). \nonumber
\end{align}
And the claim follows from the inequalities \eqref{inter radius} and \eqref{outer radius}. The second inclusion $\cup_i T^{k+1}_{j,i} \su A_{(2-3\ep)\bar{r}_j^{k},(1+\frac{7}{2}\ep) r_j^{k}} (x_j^{k}) $ can be proved similarly.

Now let us consider the neck regions. Consider the collection of all $\{S^3/\Ga_j^{k+1}\}$ with the standard metrics that appear in the bubble tree decomposition Theorem~\ref{thm3_f}.
The order of the group has a uniform upper bound $|\Ga_j^{k+1}| \leq C(v,D)$. Hence, there is a uniform volume lower bound $\vol(S^3)/C(v,D)$, and two-sided sectional curvature bound $1$ for all manifolds in the collection $\{S^3/\Ga_j^{k+1}\}$. Therefore, the convexity radius is bounded below in terms of $C(v,D)$ for all $\{S^3/\Ga_j^{k+1}\}$.

Now if $r_c(k+1,j)$ is the convexity radius of $S^3 / \Ga_j^{k+1}$, we cover $S^3 / \Ga_j^{k+1}$ by $B_{r_c(k+1,j)/4}(z_i)$ with $z_i \in S^3 / \Ga_j^{k+1}$ in an efficient way so that $B_{r_c(k+1,j)/16}(z_i)$ are pairwise disjoint. Then with the same volume comparison argument as above shows that the number balls in the covering of any $S^3 / \Ga_j^{k+1}$ is uniformly bounded above by $N_2(C(v,D))$. By the definition of the neck in equations (\ref{equ5.1}) and (\ref{equ6}), the number of trapezoids in each neck $\Nn_j^k$ is equal to the number of balls $B_{r_{k+1,j}/4}(z_i)$ to cover $S^3 / \Ga_j^{k+1}$. Hence, there are at most $N_2(C(v,D))$ trapezoids $T^{k+1}_{j,i}$ in the covering of each $A_{2\bar{r}_j^{k}, r_j^{k}} (x_j^{k})$. Then the total number of trapezoids in the necks is bounded by $N(v,D) \times k(v,D) \times N_2(C(v,D))$.  Therefore, there are at most $\ti{N}(v,D)=N(v,D) \times k(v,D) \times (N_1(r_0(v,D))+N_2(C(v,D)))$ open sets in the covering $\mathcal{O}$ of $M$.
\end{proof}

\section{Homotopy distance and simplicial approximation}\label{sec3}
In this section, we are going to first introduce a graph $\Si$ on the manifold $M$. We will show that given a curve $\ga\su M$, one can find its "simplicial approximation" in the graph $\Si$ with controlled "simplicial length". (See Definition~\ref{def4_s}.) The idea of this simplicial approximation is crucial in the proof of Theorem~\ref{thm2_w}.

Note that the proof of the existence part (Theorem~\ref{thm2_w}A) and an explicit formula in terms of certain constants (Theorem~\ref{thm2_w}B) will be based on different techniques.

The estimations in Lemma~\ref{lm1} to Lemma~\ref{lm5} will be mainly used in the proof of Theorem~\ref{thm2_w}A while Lemma~\ref{lm6} to Lemma~\ref{lm8} will be used in the proof of Theorem~\ref{thm2_w}B. Several techniques we used in Lemma~\ref{lm1} to Lemma~\ref{lm5} are due to R.~Rotman and her work \cite{rotman2000upper}.

Through out the section, we assume that $M\in \Mm(4,v,D)$. The graph $\Si$ is constructed from the covering $\mathcal{O}$ in Lemma~\ref{lem2_c} in the following three steps.

\underline{\textbf{Construction of the graph $\Si$}}
\begin{enumerate}
\item[1.]
By Lemma~\ref{lem2_c}, each body in $M$ is covered by some harmonic balls $\{B_{r(x_j)}(x_j)\}$, where $r(x)=R(x)/4=C\cdot r_h(x)$, for some constant $C$ and $r_h(x)$ is the harmonic radius at $x$. We define the center $x_j$ of the ball to be a vertex in the graph $\Si$.

If for some $i,j$, the intersection $B_{r(x_i)}(x_i)\cap B_{r(x_j)}(x_j)\neq \emptyset$ and $r(x_i)\geq r(x_j)$, then the union $B_{r(x_i)}(x_i)\cup B_{r(x_j)}(x_j)\su B_{R(x_i)}(x_i)$. In this case, we connect $x_i$ and $x_j$ with a minimizing geodesic segment $\ga$ in $M$. The triangle inequality implies that $\ga\su B_{R(x_i)}(x_i)$. We define $\ga$ to be an edge in $\Si$ connecting the vertices $x_i$ and $x_j$.
\item[2.]
Next, we conider the trapezoids $T^k_{j,i}$ in $\mathcal{O}$. For each trapezoid $T^k_{j,i}$, we choose a (any) point $x_{k,j,i} \in T^k_{j,i}$ to be a vertex in $\Si$. Since different necks are disjoint, the trapezoids in different necks do not intersect. Therefore, we only consider the intersection between trapezoids in the same neck. If $T^k_{j,i} \cap T^k_{j,l} \neq \emptyset$, then $T^k_{j,i} \cup T^k_{j,l} \subset \bar{T}^k_{j,i}\su \Nn_j^{k}$.  Let $y$ be a point in $ T^k_{j,i} \cap T^k_{j,l}$. By Lemma~\ref{trapezoid contraction}, we connect $x_{k,j,i}$ and $y$ by a curve $\ga_1 \su T^k_{j,i}$  and we connect $y$ and $x_{k,j,l}$ by a curve $\ga_2 \su T^k_{j,l}$. Let $\ga=\ga_1 \cup \ga_2 \su T^k_{j,i} \cup T^k_{j,l}$.  Then $\length(\ga) \leq 6D$ and we define this curve to be an edge in $\Si$ connecting the vertices $x_{k,j,i}$ and $x_{k,j,l}$.
\item[3.]
If $T^k_{j,i} \cap B_{r(x_l)}(x_l) \neq \emptyset$, let $y$ be a point in $ T^k_{j,i} \cap B_{r(x_l)}(x_l)$. Let $\ga_1$ be a minimizing geodesic connecting $x_l$ and $y$ in $M$. Note that because both $y$ and $x_{k,j,i}$ are in $T^k_{j,i}$, we can connect $y$ and $x_{k,j,i}$ by a curve $\ga_2$ as in Lemma~\ref{trapezoid contraction}. Let $\ga=\ga_1\cup \ga_2 \su T^k_{j,i} \cup B_{r(x_l)}(x_l)$ and $\length(\ga) \leq 4D$. We define $\ga$ to be an edge in $\Si$ connecting $x_l$ and $x_{k,j,i}$.

\end{enumerate}
\begin{figure}[htbp]
\begin{minipage}[c]{0.47\textwidth}
\centering\includegraphics[width=7cm]{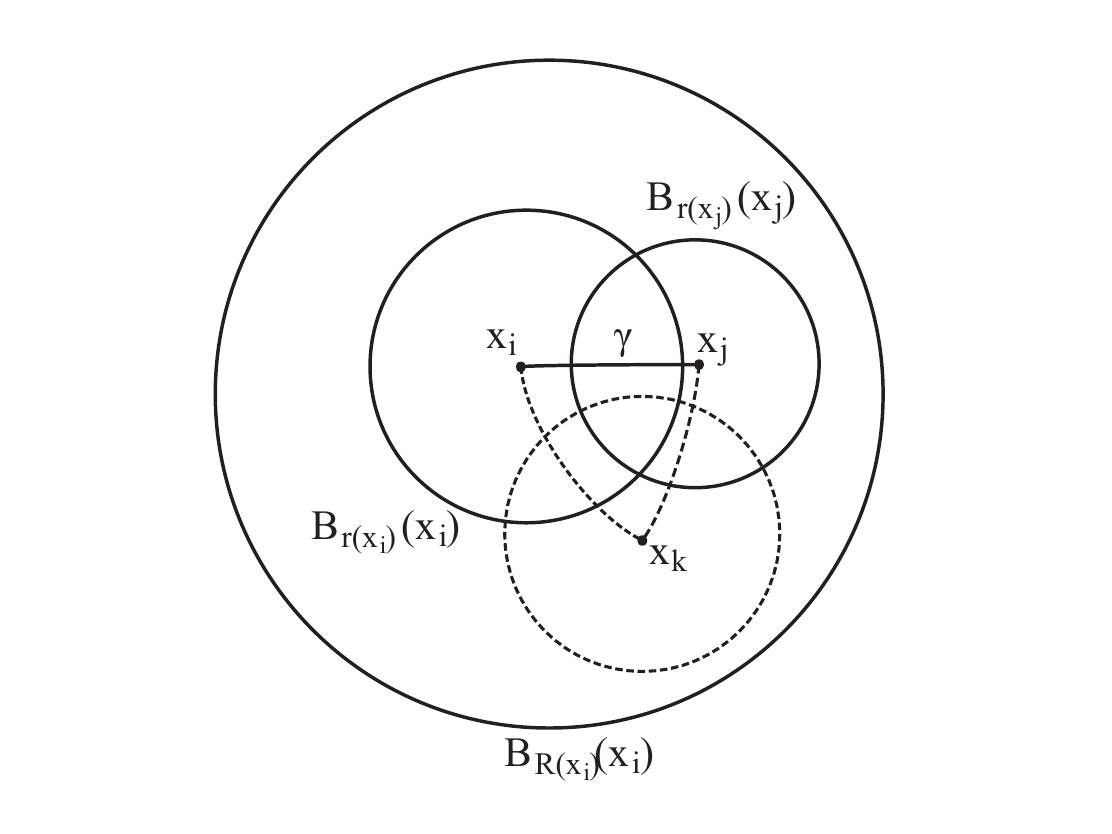}
\vspace*{6pt}
\caption{Construction of the graph $\Si$ I.}\label{fig1}
\end{minipage}
\begin{minipage}[c]{0.47\textwidth}
\centering\includegraphics[width=7cm]{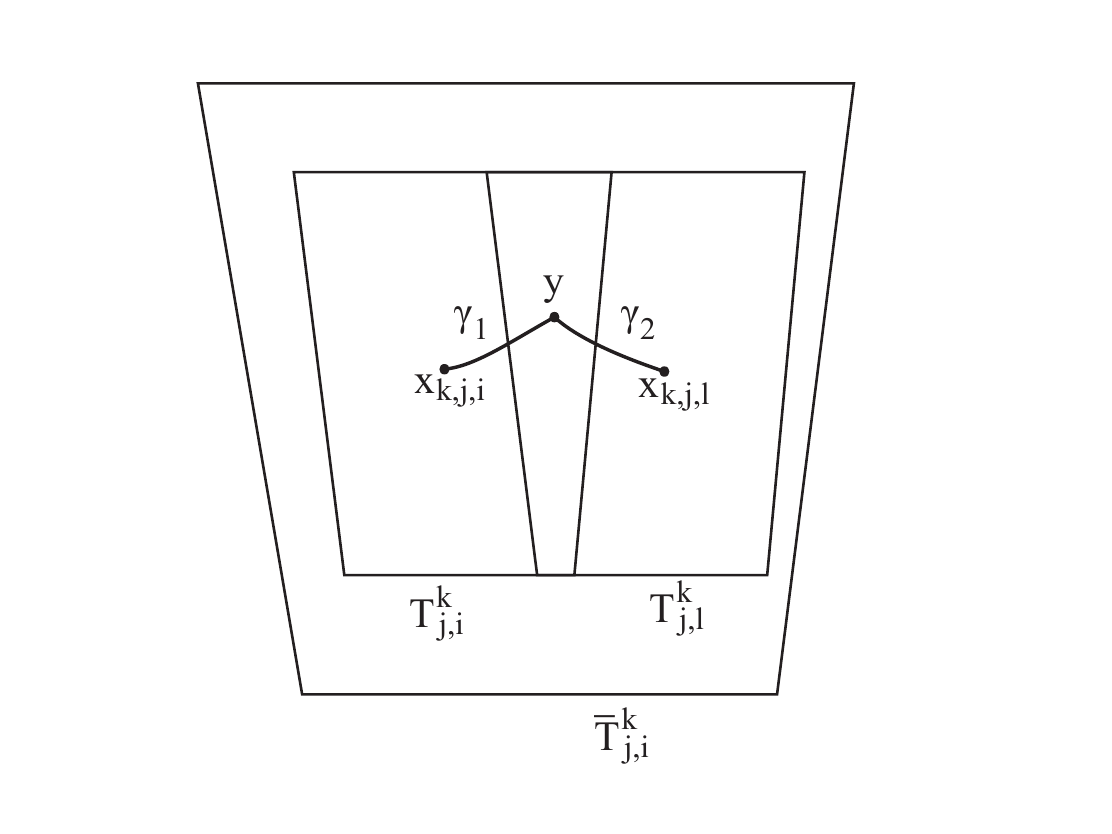}
\vspace*{6pt}
\caption{Construction of the graph $\Si$ II.}\label{fig2}
\end{minipage}
\end{figure}

\begin{figure}[htbp]
\centering\includegraphics[width=7cm]{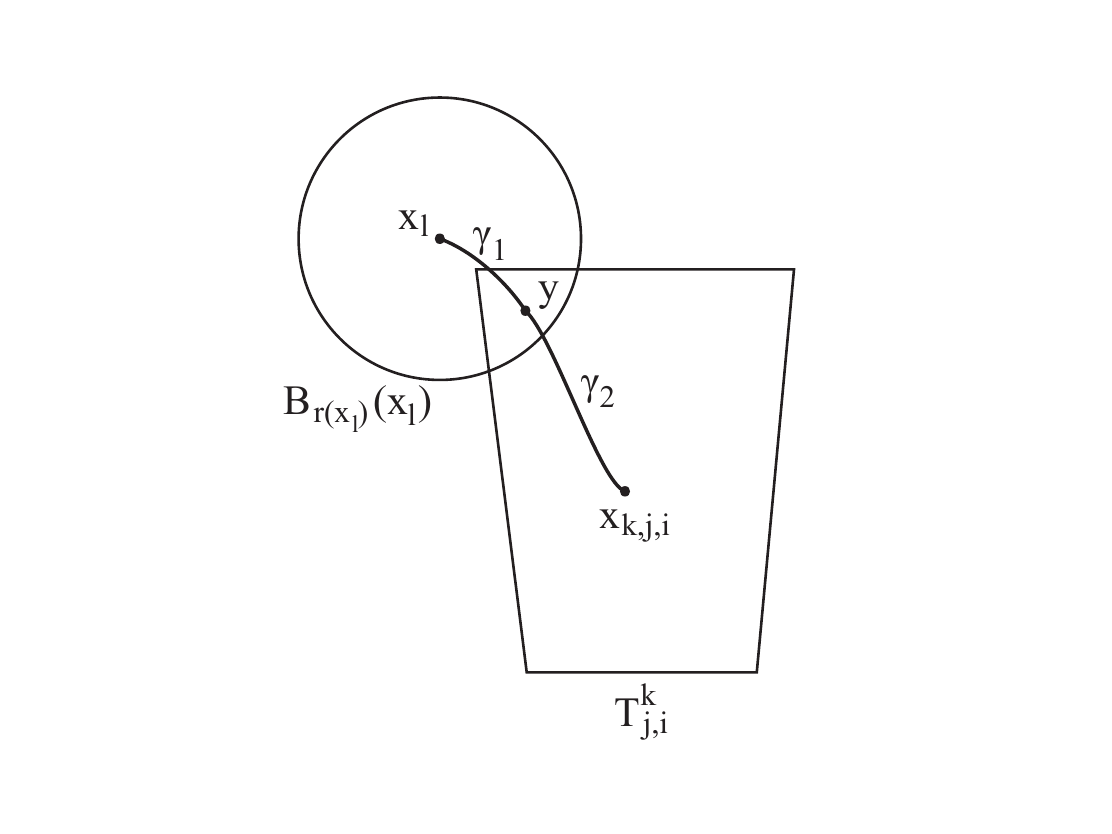}
\vspace*{6pt}
\caption{Construction of the graph $\Si$ III.}\label{fig2}
\end{figure}

We also call the points $x$, which we pick to be a vertex in $\Si$, the \textsl{center} of the open sets in $\mathcal{O}$. Note that the number of the edges in $\Si$ is bounded by $\ti{N}^2$, where $\ti{N}=\ti{N}(v,D)$ is the constant in Lemma~\ref{lem2_c}. In order to control the length of the curves in $\Si$, let us introduce the following definition.

\begin{definition}\label{def4_s}
 A simplicial curve $\al$ in $\Si$ is a simplicial map $\al:[0,1]_{\triangle}\ra \Si$, where $[0,1]_{\triangle}$ is a simplicial complex obtained by taking a partition $0<t_1<\dots<t_L<1$ of the interval $[0,1]$ and $L\geq0$ is an integer. We define the simplicial length $m(\al)$ of $\al$ to be the number of edges in $\al$. In other words, $m(\al)=L+1$. We call $\al$ a loop in $\Si$, if $\al(0)=\al(1)$.
\end{definition}

Lemma~\ref{lm1} below indicates that for any closed curve $\ga$ in the manifold $M$, one can find a curve $\ti{\ga}$ in $\Si$ which is homotopic to $\ga$ though a homotopy with bounded width. The curve $\ti{\ga}$ will be called the {\sl simplicial approximation} of the curve $\ga$.

\begin{remark}
  By the construction of the graph $\Si$, there is a natural inclusion map $\Si\hookrightarrow M$. Suppose that $\ti{\ga}$ is the simplicial approximation of a curve $\ga$. Sometimes we refer $\ti{\ga}$ as a piecewise smooth curve in $M$, which is the image of a simplicial curve under the inclusion map.
\end{remark}

\begin{lemma}\label{lm1}
For any curve $\ga:[0,1]\ra M$, there exists a simplicial curve $\ti{\ga}:[0,1]\ra \Si$ such that $\ga$ is homotopic to $\ti{\ga}$ through a homotopy $H$ with width $\om_H\leq 60 D$, where $D$ is the diameter of $M$.
\end{lemma}
\begin{proof}
Suppose $\ga:[0,1]\ra M$ is a closed curve. We are going to first decompose the curve into the open sets constructed above that cover bodies and necks. We choose a sufficiently fine subdivision $0=t_0<t_1<\dots<t_n=1$ of $[0,1]$ which satisfies the following condition:
\begin{enumerate}
\item
If $\ga(t_m)\in B_{r(x_i)}(x_i)$ and $\ga(t_{m+1})\in B_{r(x_j)}(x_j)$, then the intersection $B_{r(x_i)}(x_i)\cap B_{r(x_j)}(x_j)$ is nonempty, and $\ga([t_m,t_{m+1}])\su B_{r(x_i)}(x_i)\cup B_{r(x_j)}(x_j)$.

\item
If $\ga(t_m)\in T^k_{j,i}$ and $\ga(t_{m+1})\in T^k_{j,l}$, then the intersection $T^k_{j,i}\cap T^k_{j,l}$ is nonempty, and $\ga([t_m,t_{m+1}])\su T^k_{j,i}\cup T^k_{j,l}$.

\item
If $\ga(t_m)\in B_{r(x_i)}(x_i)$ and $\ga(t_{m+1})\in T^k_{j,l}$, then $\ga(t_{m+1}) \not\in \overline{B_{r(x_i)}(x_i)}$, where $\overline{B_{r(x_i)}(x_i)}$ is the closure of the metric ball, and $B_{r(x_i)}(x_i) \cap  T^k_{j,l} \not= \emptyset $, $\ga([t_m,t_{m+1}])\su B_{r(x_i)}(x_i) \cup T^k_{j,l}$.  Moreover, we require that there is $t'_m \in [t_m,t_{m+1}]$ such that $ \ga(t'_m)$ is in the boundary of $B_{r(x_i)}(x_i)$ ,  $\ga([t_m, t'_m]) \subset \overline{B_{r(x_i)}(x_i)}$ and $\ga([t'_m, t_{m+1}]) \subset T^k_{j,l}$.
The condition is similar if
$\ga(t_{m+1})\in B_{r(x_i)}(x_i)$ and $\ga(t_{m})\in T^k_{j,l}$.

\end{enumerate}

The loop $\ti{\ga}$ is constructed in the following way. Suppose that $\mathcal{O}$ is the covering of $M$ constructed in Lemma~\ref{lem2_c}. Based on the partition above, if $\ga(t_m)$ and $\ga(t_{m+1})$ are in two open sets in $\mathcal{O}$, then the intersection of these two open sets is non-empty and there is an edge in $\Si$ connecting the centers of the open sets. We pick $\ti{\ga}$ to be the union of edges in $\Si$ connecting the centers in the open sets which $\ga(t_m)$ lies in. Moreover, if $\ga$ is a closed curve, then $\ti{\ga}$ is a loop in $\Si$.

\begin{figure}[htbp]
\begin{minipage}[c]{0.47\textwidth}
\centering\includegraphics[width=7cm]{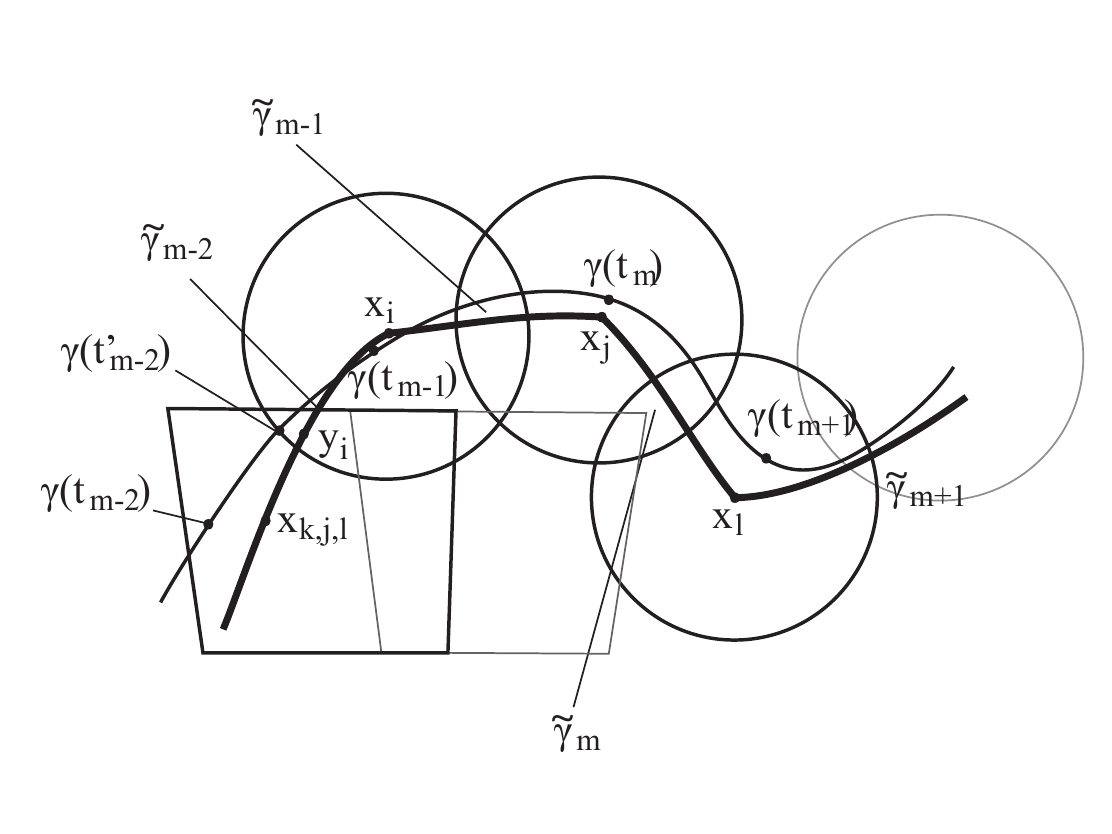}
\vspace*{6pt}
\caption{Take a subdivision of $\ga$ and choose $\ti{\ga}_k$.}\label{fig1}
\end{minipage}
\begin{minipage}[c]{0.47\textwidth}
\centering\includegraphics[width=7cm]{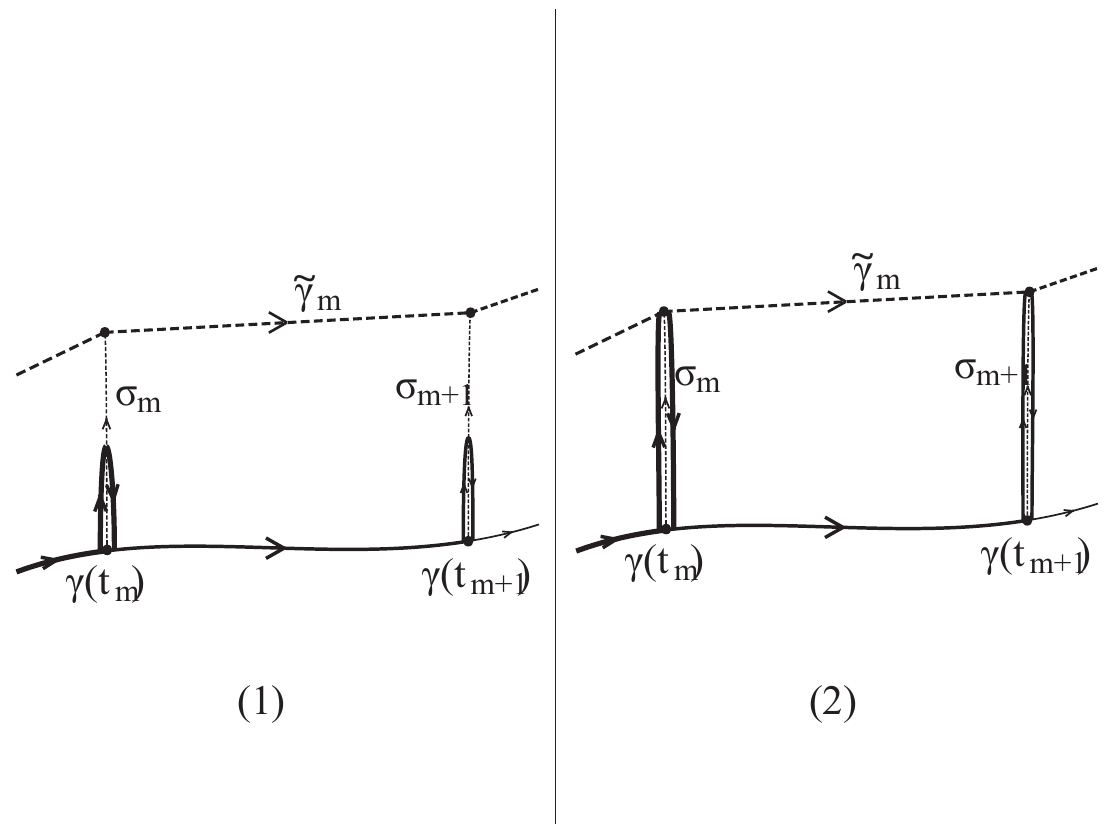}
\vspace*{6pt}
\caption{$\ga$ is homotopic to $\ga\cup_m (\si_m\cup (-\si_m))$.}\label{fig2}
\end{minipage}
\end{figure}

\begin{figure}[htbp]
\begin{minipage}[c]{0.47\textwidth}
\centering\includegraphics[width=7cm]{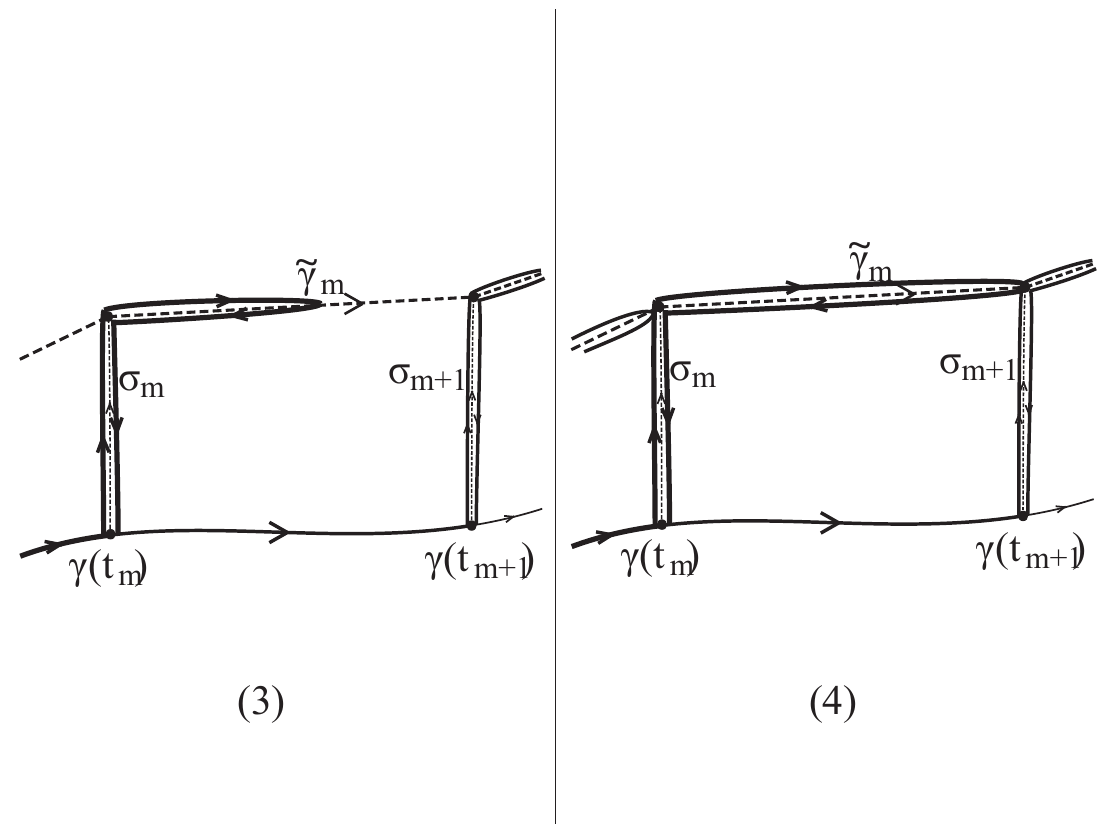}
\vspace*{6pt}
\caption{$\ga$ is homotopic to $\ga\cup_m (\si_m\cup \ti{\ga}_m \cup (-\ti{\ga}_m)\cup (-\si_m))$.}\label{fig3}
\end{minipage}
\begin{minipage}[c]{0.47\textwidth}
\centering\includegraphics[width=7cm]{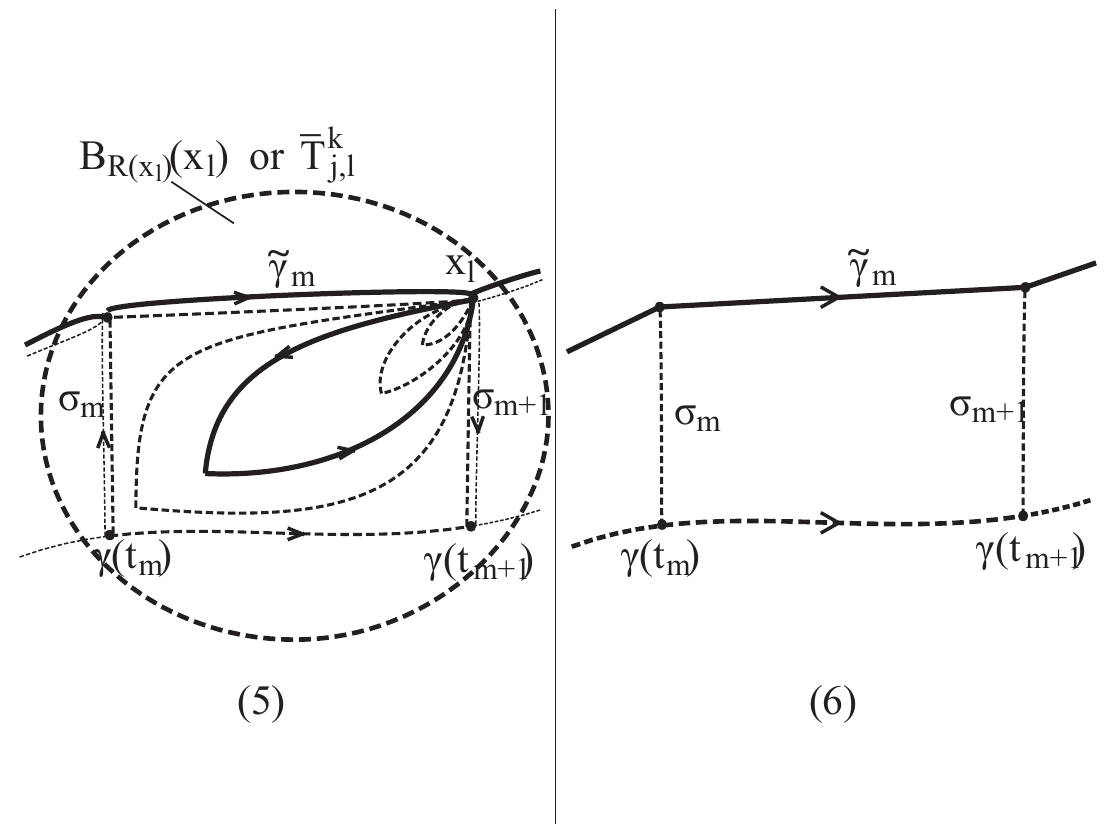}
\vspace*{6pt}
\caption{Contracting the loop $(-\ti{\ga}_m)\cup(-\si_m)\cup\ga([t_m,t_{m+1}])\cup\si_{m+1}$ and we get $\ti{\ga}$.}\label{fig4}
\end{minipage}
\end{figure}

We will show below that $\ga$ is homotopic to $\ti{\ga}$ through a homotopy of width bounded by $60D$. The construction is similar to the proof of \cite[Lemma 3.3]{rotman2000upper}. We will describe the homotopy by describing the image of the curve $\ga$ under the homotopy. The homotopy is constructed in the following three steps:

Step 1: Since $\ga(t_m)$ is in an open set in the covering, we connect $\ga(t_m)$ to the center of the open set by a curve $\si_m$. In particular, if $\ga(t_m)\in B_{r(x_i)}(x_i)$, then $\si_m$ is a minimizing geodesic between $\ga(t_m)$ and $x_i$ and $\length(\si_m) \leq D$. If $\ga(t_m)\in T^k_{j,l}$, then $\si_m$ is a curve between $\ga(t_m)$ and $x_i$ constructed in Lemma~\ref{trapezoid contraction} and $\length(\si_m) \leq 3D$. Now $\ga$ is homotopic to $\ga \cup_m (\si_m\cup (-\si_m))$ through a homotopy with width $\leq 2\length(\si_m)\leq 6D$. (See Figure~\ref{fig2}).

Step 2: Recall that when $\ga(t_m)$ and $\ga(t_{m+1})$ are in two open sets in $\mathcal{O}$, then the intersection of these two open sets is non-empty and there is an edge $\ti{\ga}_m$ in $\Si$ connecting the centers of the open sets. Based on the construction of $\Si$, $\length(\ti{\ga}_m) \leq 6D$. Hence $\ga \cup_m (\si_m\cup (-\si_m))$ is  homotopic to $\ga\cup_m (\si_m\cup \ti{\ga}_m \cup (-\ti{\ga}_m)\cup (-\si_m))$  with width $\leq 2\cdot 6D = 12D$. (See Figure~\ref{fig3}).

Step 3: We claim that each 4-gon $\ga([t_m,t_{m+1}]) \cup \si_{m+1} \cup (-\ti{\ga}_m) \cup (-\si_{m})$ can be contracted to the center of the open set where $\ga(t_{m+1})$ lies in. The width of the homotopy is bounded by $42D$.

There are three cases to be discussed.

\begin{enumerate}
\item
If $\ga(t_m)\in B_{r(x_i)}(x_i)$ and $\ga(t_{m+1})\in B_{r(x_j)}(x_j)$, then based on our construction,
 $\ga([t_m,t_{m+1}]) \cup\si_{m+1} \cup (-\ti{\ga}_m)\cup(-\si_m) \su B_{r(x_i)}(x_i) \cup B_{r(x_j)}(x_j)$.
If $r(x_i)\leq r(x_j)$, then $B_{r(x_i)}(x_i) \cup B_{r(x_j)}(x_j) \su B_{R(x_j)}(x_j)$ and
we can contract the 4-gon to the point $x_j$ within $B_{R(x_j)}(x_j)$. The width of the contraction, by Lemma~\ref{lem1_c} is bounded by $D$. (See Figure~\ref{fig4}). If $r(x_i) \geq r(x_j)$, then $\ga([t_m,t_{m+1}]) \cup\si_{m+1} \cup (-\ti{\ga}_m)\cup(-\si_m)$ is hompotic to $\ga([t_m,t_{m+1}]) \cup\si_{m+1} \cup (-\ti{\ga}_m) \cup\ti{\ga}_m \cup (-\ti{\ga}_m) \cup(-\si_m)$ with width bounded by $2D$. $\ga([t_m,t_{m+1}]) \cup\si_{m+1} \cup (-\ti{\ga}_m) \cup\ti{\ga}_m \cup (-\ti{\ga}_m) \cup(-\si_m)$ is contained in  $B_{r(x_i)}(x_i) \cup B_{r(x_j)}(x_j) \su B_{R(x_i)}(x_i)$, by Lemma~\ref{lem1_c} it is homotopic to $(-\ti{\ga}_m) \cup\ti{\ga}_m$ with width bounded by $D$. And $ (-\ti{\ga}_m) \cup\ti{\ga}_m$ can be contracted to $x_j$ with width bounded by $2D$. Hence,  if $r(x_i) \geq r(x_j)$, the 4-gon  $\ga([t_m,t_{m+1}]) \cup \si_{m+1} \cup (-\ti{\ga}_m) \cup (-\si_{m})$ can be contracted to $x_j$ with width bounded by $2D+D+2D=5D$.

\item
If $\ga(t_m)\in T^k_{j,i}$ and $\ga(t_{m+1})\in T^k_{j,l}$,  then based on our construction,
 $\ga([t_m,t_{m+1}]) \cup\si_{m+1} \cup (-\ti{\ga}_m)\cup(-\si_m) \su  T^k_{j,i} \cup T^k_{j,l} \su \bar {T}^k_{j,l}$. Hence, by Lemma~\ref{trapezoid contraction}, the 4-gon can be contracted to $x_{k,j,l}$ with width bounded by $21D$.

\item
Consider the case when $\ga(t_m)\in B_{r(x_i)}(x_i)$ and $\ga(t_{m+1})\in T^k_{j,l}$, (similarly when $\ga(t_{m+1})\in B_{r(x_i)}(x_i)$ and $\ga(t_{m})\in T^k_{j,l}$).

 First note that the intersection $B_{r(x_i)}(x_i) \cap  T^k_{j,l} \not= \emptyset$. By the definition of $T^k_{j,l}$, we have either $B_{r(x_i)}(x_i) \cap A_{r_j^{k-1},(1+\frac{7}{2}\ep) r_j^{k-1}} (x_j^{k-1})  \not= \emptyset$, or $B_{r(x_i)}(x_i) \cap A_{(2-3\ep)\bar{r}_j^{k-1},2\bar{r}_j^{k-1}} (x_j^{k-1}) \not= \emptyset$. We claim that $$B_{4r(x_i)}(x_i) =B_{R(x_i)}(x_i) \subset \Nn_j^k.$$ Suppose not. If $B_{r(x_i)}(x_i) \cap  A_{r_j^{k-1},(1+\frac{7}{2}\ep) r_j^{k-1}} (x_j^{k-1}) \not= \emptyset$ and $B_{4r(x_i)}(x_i) \not \su \Nn_j^k$, then $\Nn_j^k \subset B_{32r(x_i)}(x_i) \subset B_{r_h(x_i)}(x_i)$, which is a contradiction. On the other hand, if $B_{r(x_i)}(x_i) \cap A_{(2-3\ep)\bar{r}_j^{k-1},2\bar{r}_j^{k-1}} (x_j^{k-1}) \not= \emptyset$ and $B_{4r(x_i)}(x_i) \not \su \Nn_j^k$, then $\Bb^k_j \subset B_{32r(x_i)}(x_i)$, which violates the assumption in Lemma \ref{lem2_c}. Hence, in any case, we have $B_{R(x_i)}(x_i) \subset \Nn_j^k$.

Now recall that there is a point $y$ in $B_{r(x_i)}(x_i) \cap  T^k_{j,l}$, so that the edge $\ti{\ga}_m$ between $x_i$ and $x_{k,j,l}$ in $\Si$ consists of a minimizing geodesic $\ti{\ga}_{m,1}$ between $y$ and $x_i$ and a curve $\ti{\ga}_{m,2}$ in $\bar{T}^k_{j,l}$ between $y$ and $x_{k,j,l}$. Let us pick $t_m<t'_m<t_{m+1}$ such that $\ga(t'_m) \in \partial B_{r(x_i)}(x_i)$ as described in (3) above.  The geodesic distance in $M$ between $y$ and $\ga(t'_m)$ is less than $2r(x_i)$. Hence the minimizing geodesic between  $\ga(t'_m)$ and $y$ is contained in $B_{R(x_i)}(x_i) \subset \Nn_j^k$. By Lemma~\ref{trapezoid contraction}, we can connect $y$ and $\ga(t'_m)$ by a curve $\delta$ in $\bar{T}^k_{j,i}$ with length less than $2\x \text{distance}(y,\ga(t'_m))\leq 4r(x_i)$. Now, the triangle $x_i y \ga(t'_m)$ has circumference less than $6r(x_i)$, hence the 4-gon $(-\si_m) \cup \ga([t_m, t'_m]) \cup \delta \cup (-\ti{\ga}_{m,1})$ is in $B_{R(x_i)}(x_i)$, in particular $\delta \su B_{R(x_i)}(x_i) \cap  T^k_{j,l}$. Then, the contraction of  $\ga([t_m, t_{m+1}]) \cup\si_{m+1} \cup (-\ti{\ga}_m)\cup(-\si_m)$ to $x_{k,j,l}$ can be described in the following four steps.

\begin{figure}[htbp]
\centering\includegraphics[width=7cm]{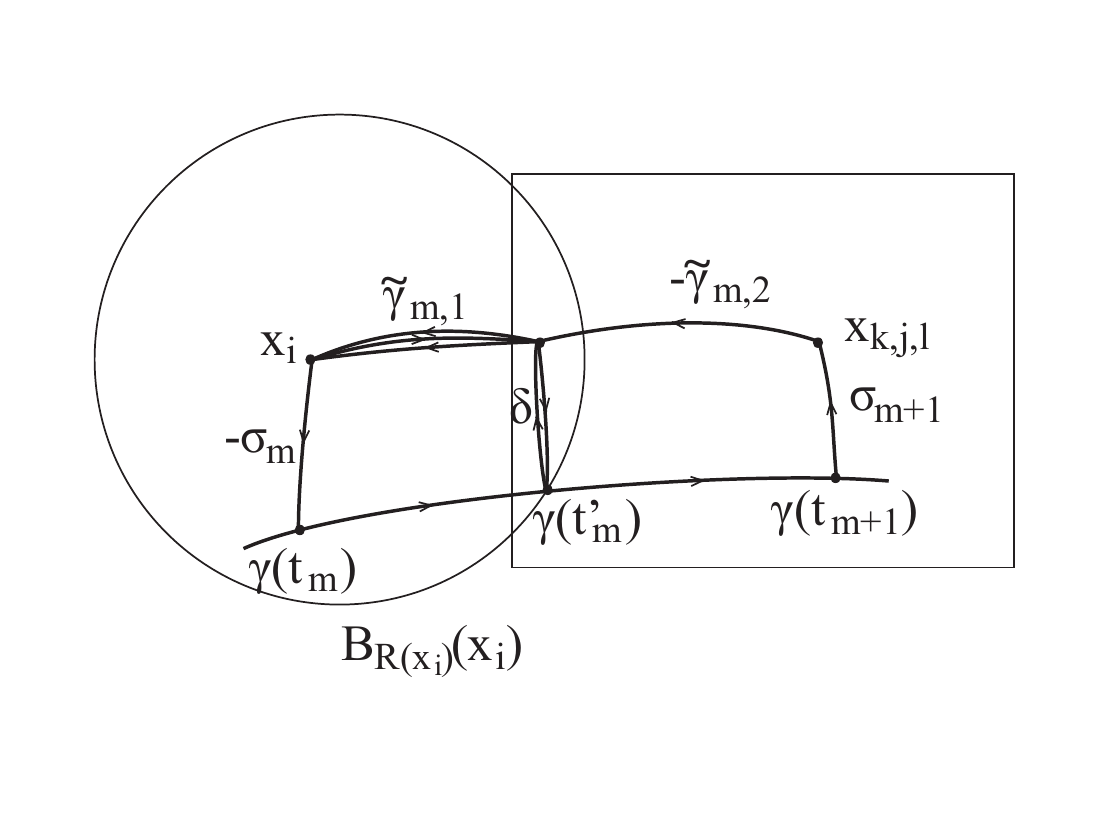}
\vspace*{6pt}
\caption{$\ga([t_m, t_{m+1}]) \cup\si_{m+1} \cup (-\ti{\ga}_m)\cup(-\si_m)$ is homotopic to $(-\si_m) \cup \ga([t_m, t'_m]) \cup [\delta \cup (-\ti{\ga}_{m,1})  \cup (\ti{\ga}_{m,1})  \cup (-\delta)] \cup  \ga([t'_m, t_{m+1}]) \cup\si_{m+1} \cup (-\ti{\ga}_{m,2}) \cup (-\ti{\ga}_{m,1})$}\label{ball trape case}
\end{figure}

First, note that $\ga([t_m, t_{m+1}]) \cup\si_{m+1} \cup (-\ti{\ga}_m)\cup(-\si_m) =(-\si_m) \cup \ga([t_m, t'_m]) \cup \ga([t'_m, t_{m+1}]) \cup\si_{m+1} \cup (-\ti{\ga}_{m,2}) \cup (-\ti{\ga}_{m,1}).$
It is homotopic to $(-\si_m) \cup \ga([t_m, t'_m]) \cup [\delta \cup (-\ti{\ga}_{m,1})  \cup (\ti{\ga}_{m,1})  \cup (-\delta)] \cup  \ga([t'_m, t_{m+1}]) \cup\si_{m+1} \cup (-\ti{\ga}_{m,2}) \cup (-\ti{\ga}_{m,1})$ with width bounded by $2 \length(\ti{\ga}_{m,1})+2\length(\delta) \leq 10 D$.  (See Figure~\ref {ball trape case} )

Second, we can contract $(-\si_m) \cup \ga([t_m, t'_m]) \cup \delta \cup (-\ti{\ga}_{m,1}) $ in $B_{R(x_i)}(x_i)$ and end up with $(\ti{\ga}_{m,1})  \cup (-\delta) \cup  \ga([t'_m, t_{m+1}]) \cup\si_{m+1} \cup (-\ti{\ga}_{m,2}) \cup (-\ti{\ga}_{m,1})$.By Lemma~\ref{lem1_c}, the width is bounded by $D$.

Third,$(\ti{\ga}_{m,1})  \cup (-\delta) \cup  \ga([t'_m, t_{m+1}]) \cup\si_{m+1} \cup (-\ti{\ga}_{m,2}) \cup (-\ti{\ga}_{m,1})$ is homotopic to $(-\delta) \cup  \ga([t'_m, t_{m+1}]) \cup\si_{m+1} \cup (-\ti{\ga}_{m,2}) $ with width bounded by $2 \length(\ti{\ga}_{m,1}) \leq 2D$.

At last, by Lemma~\ref{trapezoid contraction}, $(-\delta) \cup  \ga([t'_m, t_{m+1}]) \cup\si_{m+1} \cup (-\ti{\ga}_{m,2}) $  can be contracted to $x_{k,j,l}$ with width bounded by $21D$.

If we sum up the width in the above four steps, we conclude that the 4-gon $\ga([t_m,t_{m+1}]) \cup \si_{m+1} \cup (-\ti{\ga}_m) \cup (-\si_{m})$ can be contracted to $x_{k,j,l}$ with width bounded by $10D+D+2D+21D = 34D$.

The case when $\ga(t_{m+1})\in B_{r(x_i)}(x_i)$ and $\ga(t_{m})\in T^k_{j,l}$ can be discussed similarly. In this case, the 4-gon $\ga([t_m,t_{m+1}]) \cup \si_{m+1} \cup (-\ti{\ga}_m) \cup (-\si_{m})$ can be contracted to $x_i$ with width bounded by $42D$.
\end{enumerate}

If we combine Steps 1, 2 and 3, we conclude that $\ga$ is homotopic to $\ti{\ga}$ with width bounded by $6D+12D+42D=60D$.

\end{proof}

Given a curve $\ga$ in $M$, in general, its approximation $\ti{\ga}$ can have arbitrarily large simplicial length. Therefore we apply Lemma~\ref{lm2} below to decompose the curve $\ti{\ga}$ into the wedge of curves with bounded simplicial length.

\begin{lemma}\label{lm2}
Let $\al$ be a loop in $\Si$. Then $\al$ is homotopic to another  loop $\al'$ in $\Si$ through a homotopy $H$ with the following properties:
\begin{enumerate}
\item $\al'$ can be represented by the union of $m(\al')$ curves $\cup_{i=1}^{m(\al')} \al_i'$, where each $\al_i'$ is a  loop in $\Si$ with a common base point and the simplicial length $m(\al_i')\leq 2\ti{N}^2+1$.
\item $\om_H\leq 12\ti{N}^2D$.
\end{enumerate}
\end{lemma}
\begin{proof}
Suppose $\al:[0,1]\ra\Si$ is a  loop in $\Si$. We choose a partition  $\{t_i\}_{i=1}^{m(\al)}$ of $[0,1]$ such that for any $i$, $\al(t_i)$ is a vertex in $\Si$ and the arc $\al_i=\al([t_i,t_{i+1}])$ is an edge in $\Si$. (See Figure~\ref{fig5}.)

\begin{figure}[htbp]
\begin{minipage}[c]{0.47\textwidth}
\centering\includegraphics[width=7cm]{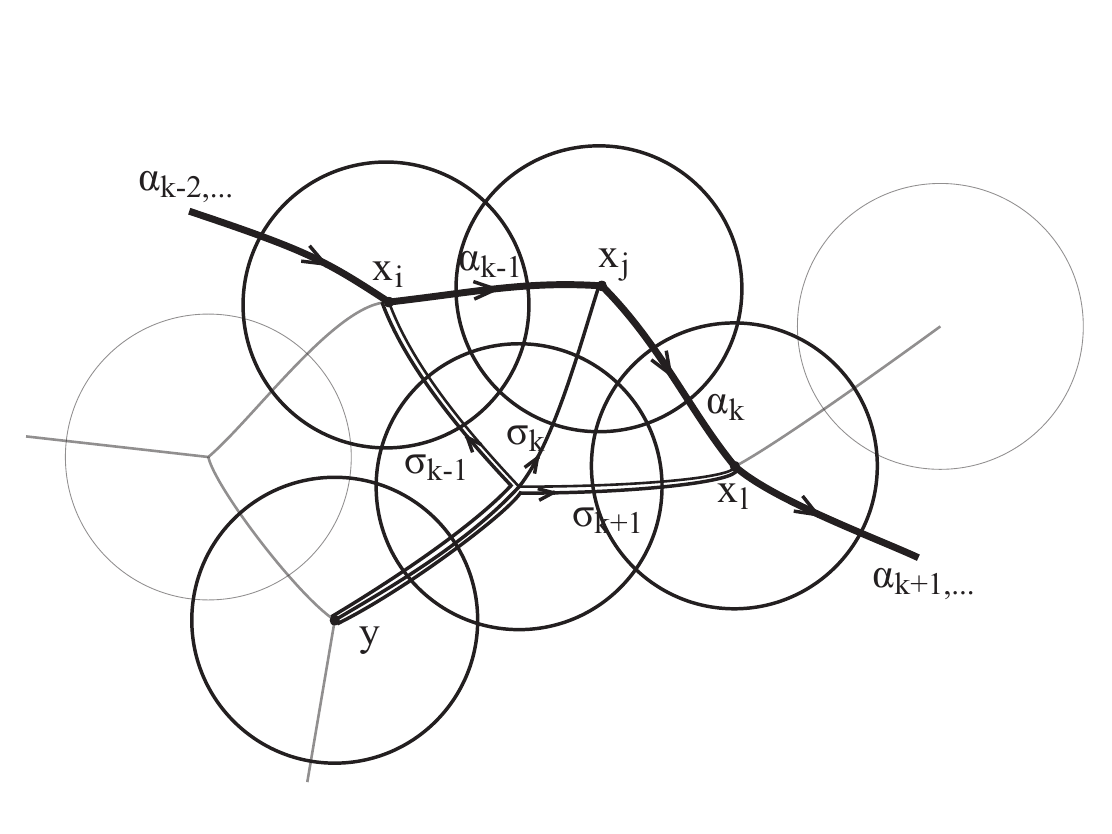}
\vspace*{6pt}
\caption{Subdivision of the curve $\al$ and pick $\si_k$.}\label{fig5}
\end{minipage}
\begin{minipage}[c]{0.47\textwidth}
\centering\includegraphics[width=7cm]{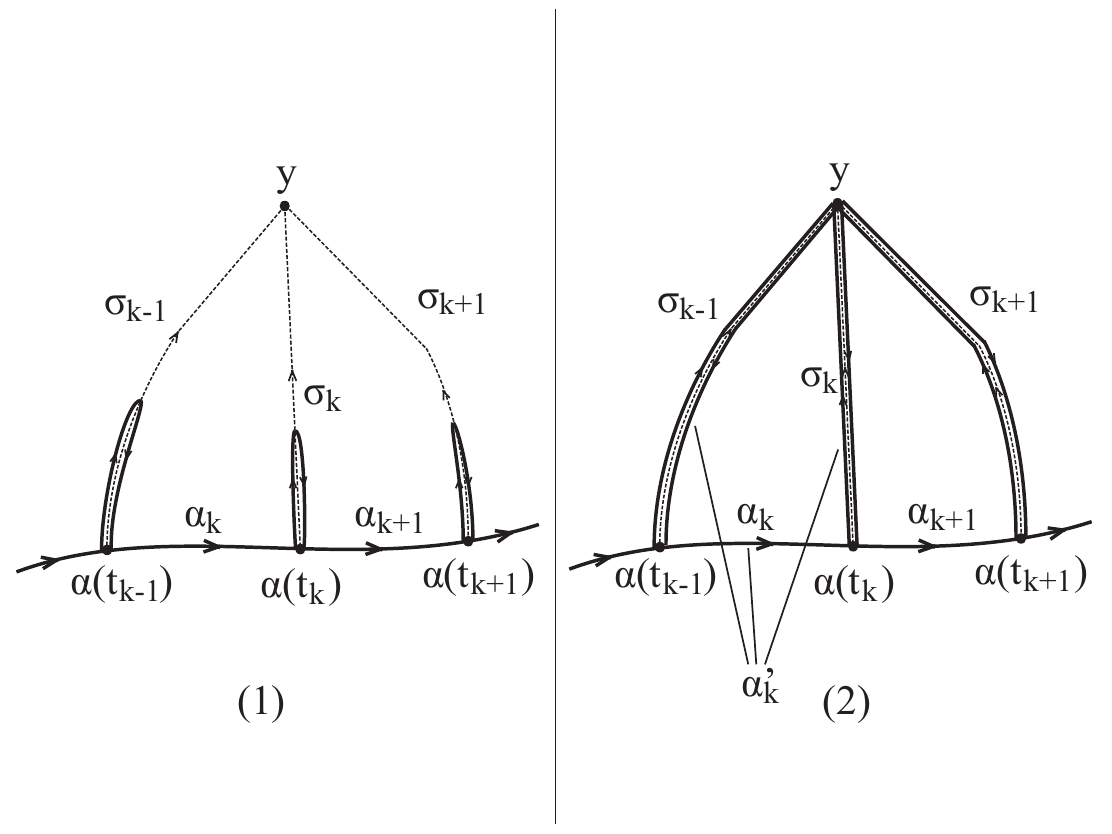}
\vspace*{6pt}
\caption{$\al$ is homotopic to $\al_1'\cup\al_2'\cup\dots\cup\al_{m(\al)}'$.}\label{fig6}
\end{minipage}
\end{figure}

Now we pick any vertex $y\in \Si$ and we connect $y$ with the vertices $\al(t_i)$ by a curve $\si_i \su \Si$. Note that this is possible since $\Si$ is path connected. Furthermore, the simplicial length $m(\si_i)$ is bounded by the total number of the edges $\ti{N}^2$ in $\Si$. Let $\al_i'=\si_i\cup \al_i\cup (-\si_{i+1})$. Then $\al$ is homotopic to $\al_1'\cup\al_2'\cup\dots\cup\al_{m(\al)}'$. (See Figure~\ref{fig6}.)  Since the length of each edge in $\Si$ is bounded by $6D$, the width of the homotopy is bounded by $2\cdot \length(\si_i)$, which is bounded by $12 \ti{N}^2D$.
\end{proof}

In the next Lemma, we show that how to homotope the curves that are "close" to each other in the  graph $\Si$.

\begin{lemma}\label{lm3}
Let $x_1,x_2,x_3$ be three vertices in $\Si$. Let $\al_i:[0,1]\ra \Si$, $i=1,2$, be two piecewise  loops in $\Si$ with the following properties: (See Figure~\ref{fig7}(1).)
\begin{enumerate}
\item For some $0\leq t_1<t_2\leq 1$, $\al_1(t_1)=x_1$ and $\al_1(t_2)=x_2$.
\item For some $0\leq t_1'<t_2'<t_3'\leq 1$, $\al_2(t_1')=x_1$, $\al_2(t_2')=x_3$ and $\al_2(t_3')=x_2$.
\item If the point $x_i$ belongs to some open sets $B_{r(x_i)}(x_i)$ or $T^k_{j,i}$ in $\mathcal{O}$, respectively, then the intersection between the open sets, pairwise, is non-empty. (See Figure below)
\item $\al_1([0,t_1])$=$\al_2([0,t_1'])$ and $\al_1([t_2,1])=\al_2([t_3',1])$.
\end{enumerate}
Then $\al_1$ is homotopic to $\al_2$ through a homotopy $H$ with width $\om_H\leq 66D$.
\end{lemma}
\begin{proof}
The proof is similar to Lemma~\ref{lm1}.
Let us denote $\tau_1=\al_2([t_1',t_2'])$ and $\tau_2=\al_2([t_2',t_3'])$. Based on the construction of edges in $\Si$, $\length(\tau_1) \leq 6D$ and $\length(\tau_2) \leq 6D$. First, $\al_1$ is homotopic to $\al_1|_{[0,t_1]}\cup \tau_1\cup (-\tau_1)\cup \al_1|_{[t1,1]}$ through a homotopy with width bounded by $2\cdot\length(\tau_1)\leq 12D$. (See Figure~\ref{fig7}(2).) We then homotope the curve to $\al_1|_{[0,t_1]}\cup \tau_1 \cup \tau_2 \cup (-\tau_2) \cup (-\tau_1)\cup \al_1|_{[t1,1]}$ by a homotopy with width bounded by $2\cdot\length(\tau_2)\leq 12D$. (See Figure~\ref{fig8}(3).)

\begin{figure}[htbp]
\begin{minipage}[c]{0.47\textwidth}
\centering\includegraphics[width=7cm]{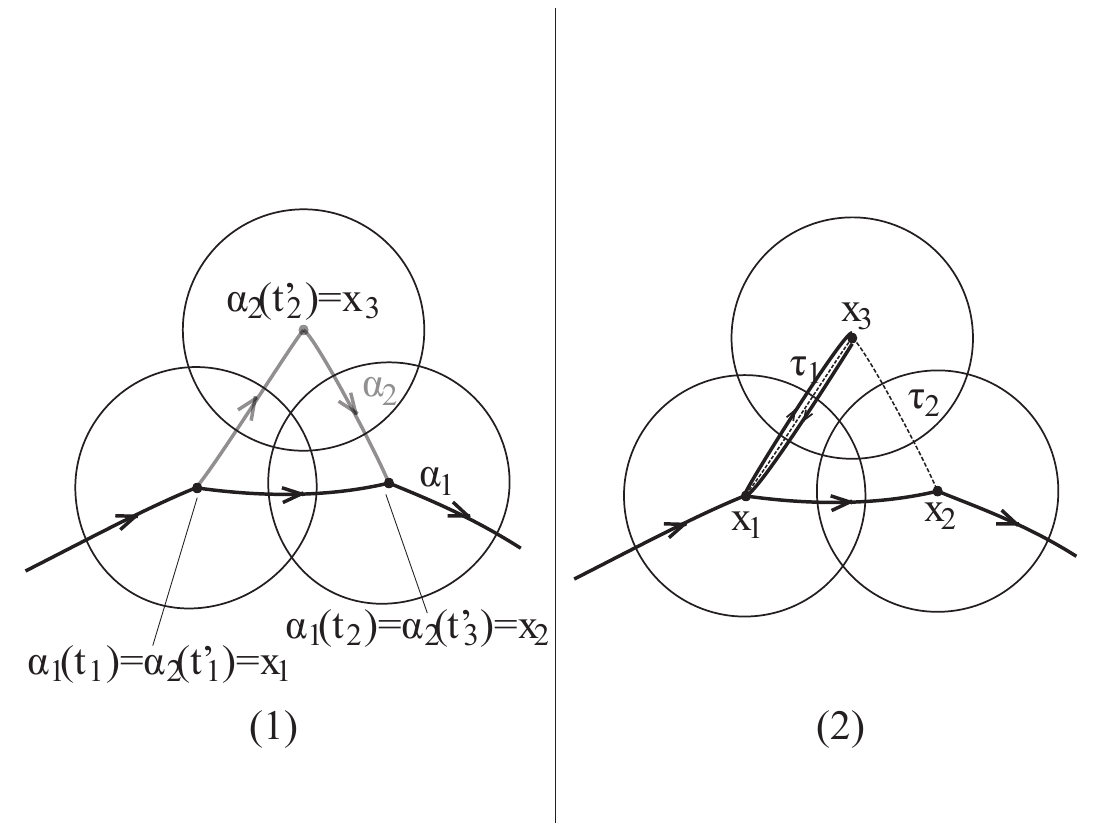}
\vspace*{6pt}
\caption{$\al_1$ is homotopic to $\al_1|_{[0,t_1]}\cup \tau_1\cup (-\tau_1)\cup \al_1|_{[t1,1]}$.}\label{fig7}
\end{minipage}
\begin{minipage}[c]{0.47\textwidth}
\centering\includegraphics[width=7cm]{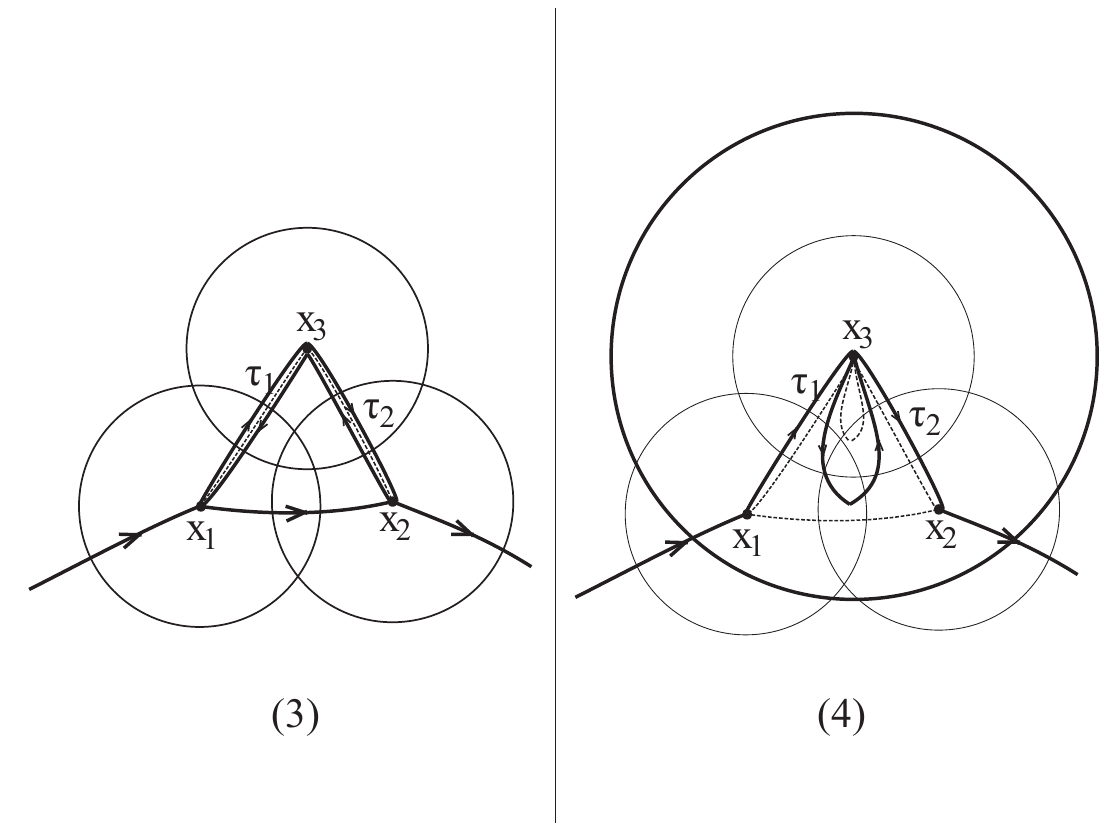}
\vspace*{6pt}
\caption{Contracting a triangle $(-\tau_2)\cup(-\tau_1)\cup \al_1([t_1,t_2])$}\label{fig8}
\end{minipage}
\end{figure}

At last, we show that the triangle $(-\tau_2)\cup(-\tau_1)\cup \al_1([t_1,t_2])$ can be contracted to a vertex with width bounded by $42D$ . Similar to the proof of Lemma~\ref{lm1}, there are several cases to be discussed.

\begin{enumerate}

\item If all three open sets are balls $B_r(x_i)(x_i)$, WLOG, suppose $B_{r(x_3)}(x_3)$ is the ball with the largest radius in $\{B_{r(x_i)}(x_i)\}$, $i=1,2,3$. In this case, we will contract the triangle to the vertex $x_3$. By our construction of the radius $r(x_i)$, we have $\cup_{i=1}^3 B_{r(x_i)}(x_i)\su B_{R(x_3)}(x_3)$, where $R(x_i)$ is the function in Lemma~\ref{lem1_c}. Now by Lemma~\ref{lem1_c}, any loop based at $x_3$ in $B_{R(x_3)}(x_3)$ can be contracted to the base point $x_3$ through a homotopy with width bounded by $D$. (See Figure~\ref{fig8}(4).)

\item If all three open sets are trapezoids $T^k_{j,i}$, and $x_2$ is the center of $T^k_{j,2}$, then $\cup_{i=1}^3 T^k_{j,i} \su \bar{T}^k_{j,2}$. By the construction of $\Si$, $(-\tau_2)\cup(-\tau_1)\cup \al_1([t_1,t_2])$ is contained in $\bar{T}^k_{j,2}$. By Lemma~\ref{trapezoid contraction} (3), we can contract $(-\tau_2)\cup(-\tau_1)\cup \al_1([t_1,t_2])$ to $x_2$ with width bounded by $21D$.

\item If one of the three open sets is a trapezoid $T^k_{j,i}$, we assume, WLOG, that $x_1\in T_{k,j,1}$, and $r(x_3)\geq r(x_2)$. With the same argument as in the case (1) above and the case (3) in the proof of Lemma~\ref{lm1}, $B_{r(x_2)}\su B_{R(x_3)}(x_3)$ and $B_{R(x_3)}(x_3)\su \Nn_{j}^k$. With the same construction in Lemma~\ref{lm1} (3) (also Lemma~\ref{trapezoid contraction}), we can contract the triangle $(-\tau_2)\cup(-\tau_1)\cup \al_1([t_1,t_2])$ to $x_3$ through a homotopy with width bounded by $42D$.

\item If two of the three open sets are trapezoids $T^k_{j,i}$, WLOG, say $x_1\in T_{k,j,1}$, $x_2\in T_{k,j,2}$ and $x_3\in B_{r(x_3)}(x_3)$. Note that in this case we still have $B_{R(x_3)}(x_3)\su \Nn_{j}^k$. And with the same argument in Lemma~\ref{lm1} (2) and (3), one can contract $(-\tau_2)\cup(-\tau_1)\cup \al_1([t_1,t_2])$ to a point using a homotopy with width bounded by $34D$.
\end{enumerate}

If we take all previous steps into account, then $\al_1$ is homotopic to $\al_2$ through a homotopy $H$ with width $\om_H\leq 12D+12D+42D \leq 66D$.
\end{proof}

Our next step is to show some general results about the simplicial approximation of a loop $\al\su \Si$ in the nerve of the covering $\Nn(M)$, which will be used to obtain an upper bound in Theorem~\ref{thm2_w}A. Recall that the nerve $\Nn(M)$ of the open covering $\mathcal{O}=\cup_{i=1}^{\ti{N}}\mathcal{O}_i$ of $M$ is a simplicial complex where the vertices corresponds to the open sets $\{\mathcal{O}_i\}$, and the $n-$simplicies corresponds to the non-empty intersections $\{\cap_{k=1}^{n} \mathcal{O}_k\}$.

\begin{lemma}\label{lm4}
For any loop $\al\su \Si$, there is a simplicial approximation $S$ of $\al$ such that $S$ is a simplicial 1-chain in the nerve $\Nn(M)$ and the number of 1-simplices in $S$ is $m(\al)$.
\end{lemma}
\begin{proof}
Let $\al:[0,1]\ra \Si$ be a loop. As before, we choose a partition $\{t_i\}_{i=1}^{m(\al)}$ of $[0,1]$ such that $\al(t_i)$ is a vertex in $\Si$ and $\tau_i:=\al([t_i,t_{i+1}])$ is an edge in $\Si$ for any $i$.

Let $f:M\ra \Nn(M)$ be the natural map obtained by using a partition of unity subordinate to the covering $\{\mathcal{O}_i\}$. Based on our construction , if $\tau_i$ connects the centers of two open sets $\mathcal{O}_k$ and $\mathcal{O}_{k'}$, then $\tau_i \su \mathcal{O}_k \cup \mathcal{O}_{k'}$, the image $f(\tau_i)$ is contained in a simplex $\De$ and the 1-simplex $\overline{f(\al(t_i))f(\al(t_{i+1}))}$ is an edge of $\De$. Because $\De$ is contractible, $\tau_i$ and $\overline{f(\al(t_i))f(\al(t_{i+1}))}$ are homotopic. We apply this simplicial approximation for all $i$, where $1\leq i\leq m(\al)$. Then we will obtain a simplicial 1-chain $S$ by taking the sum of all $\overline{f(\al(t_i))f(\al(t_{i+1}))}$. The number of 1-simplices in $S$ is $m(\al)$.
\end{proof}

\begin{lemma}\label{lm5}
Let $\Nn(M)$ be the nerve of the covering $\mathcal{O}$ of $M$. Let $S\su \Nn(M)$ be a simplicial complex with $m$ simplices and $\ga\su S$ a closed simplicial curve in $S$. Suppose $\ga$ is contractible through a simplicial homotopy $H$ in $S$ and the number of the 1-simplex in $\ga$ is $l$, counting with multiplicity, then there is an increasing function $F(m,l)$ such that the image of the simplicial homotopy $H$ consists of no more than $F(m,l)$ 2-simplices.
\end{lemma}
\begin{proof}
The argument is essentially the same with the proof of \cite[Lemma~3.5(b)]{rotman2000upper}. We include the proof for the sake of completeness. For every positive integer $m$ and $l$, there are only finitely many simplicial complexes $S$ with no more than $m$ simplices in $\Nn(M)$ and finitely many contractible closed simplicial curves $\ga$ in $S$ with no more than $l$ many $1-$simplices. By taking the maximum over all pairs $\ga$ and $S$ of the number of 2-simplices in the optimal homotopy contracting $\ga$ in $S$, we obtain an increasing function of $m$ and $l$.
\end{proof}

The proof of Theorem~\ref{thm2_w}B will be based on a slightly different construction than the proof of Theorem~\ref{thm2_w}A. In order to obtain an explicit estimate, we are going to extend the graph $\Si$ so that it captures more geometric information of the manifold.

More specifically, we define a graph $\Ga$ whose vertices are still the centers of the open sets in $\mathcal{O}$. For each body $\Bb^k_i$, if $\{B_{r(x_j)}(x_j)\}$ are all the open balls in $\mathcal{O}$ that cover $\Bb^k_i$, then we connect any two distinct vertices $x_l$ and $x_j$ in this covering by a minimizing geodesic and we define it to be an edge in $\Ga$. Note that if there are more than one geodesics connecting $x_l$ and $x_j$, we just pick one of them. The edges in $\Ga$ connecting vertices in $T^k_{j,i}$ or $T^k_{j,i}$ and $B_{r(x_l)}(x_l)$, when their intersection is non-empty, remain the same as the graph $\Si$.

\begin{remark}
Note that $\Si\su\Ga$ is a subgraph. Let us extend Definition~\ref{def4_s} to the graph $\Ga$. Then the conclusion of Lemma~\ref{lm2} is still true if we replace $\Si$ by $\Ga$. (See Lemma~\ref{lm7}.)
\end{remark}

The key idea in the proof of Theorem~\ref{thm2_w}B is that we would like to approximate some ``good'' curve, i.e, minimizing geodesics, in $M$ by a curve in $\Ga$ with controlled simplicial length. To see this, we are going to first partite a minimizing geodesic into bodies and necks with controlled number of pieces. We will then show, respectively in Lemma~\ref{lm6} and Lemma~\ref{approx in neck}, that how to estimate the simplicial length of the approximation of a minimizing geodesic in body and neck. These estimation results are combined in Lemma~\ref{approx minimizing geodesic}.

\begin{definition}
Let $B_{r}(x)$ be an open metric balls in $M$ and $\overline{B_{r}(x)}$ its closure. Let $\ga:[0,1] \ra M$ be a curve in $M$. We define \textsl{the first intersection point} of $\ga$ with $\overline{B_{r}(x)}$ to be a point $\ga(a)$ with
\begin{equation}
a=\{\max_{t \in [0,1]} t | \ga([0,t)) \su M \setminus \overline{B_{r}(x)} \}.
\end{equation}
Note that if $\ga(0) \in M \setminus \overline{B_{r}(x)}$, then the first intersection point is on the boundary of $B_{r}(x)$. If $\ga(0) \in \overline{B_{r}(x)}$, then we define the first intersection point to be $\ga(0)$. Similarly, we define \textsl{the last intersection point} of $\ga$ with $\overline{B_{r}(x)}$ to be $\ga(b)$ with
\begin{equation}
b=\{\min_{t \in [0,1]} t | \ga((t,1]) \su M \setminus \overline{B_{r}(x)} \}.
\end{equation}
Similarly, if $\ga(1) \in M \setminus \overline{B_{r}(x)}$, then the last intersection point is on the boundary of $B_{r}(x)$. If $\ga(1) \in \overline{B_{r}(x)}$, then we define the last intersection point to be $\ga(1)$. Moreover, it follows from the definition that if the first intersection point exists, then the last intersection point exists and we have $0 \leq a \leq b \leq 1$.
\end{definition}

Follows from this definition, we have:

\begin{lemma}\label{intersection point}
Let $B_{r_1}(x)$ and $B_{r_2}(x)$ be two open metric balls in $M$ with $r_2 \geq 2r_1$. Suppose that $\ga: [0,1] \ra M$ is a minimizing geodesic. If $\ga(a)$ and $\ga(b)$ are the first and the last intersection points of $\ga$ with $\overline{B_{r_1}(x)}$ respectively, then $\ga([a,b]) \su B_{r_2}(x)$.
\end{lemma}
\begin{proof}
If $\ga(a)=\ga(b)$, then the statement is true trivially. If $\ga(a) \not=\ga(b)$, and suppose that $\ga([a,b]) \not\su B_{r_2}(x)$, then there is a point $y \in \ga([a,b])$ such that $y \in M \setminus B_{r_2}(x)$. Hence $\length(\ga([a,b])) >2 r_1$. However, since $\ga(a),\ga(b) \in \overline{B_{r_1}(x)}$, the distance between them $d(\ga(a),\ga(b)) \leq 2 r_1<\length(\ga([a,b])) $. This contradicts to that $\ga$ is minimizing.
\end{proof}

Now let us introduce a partition of a minimizing geodesic in $M$.

\begin{lemma}[Partition Lemma] \label{partition of geodesic}
Let $M\in\Mm(4,v,D)$. Suppose that $M$ has a ``bubble tree'' decomposition as in Theorem~\ref{thm3_f}. Let $\ga: [0,1] \ra M$ be a minimizing geodesic. Then there is a partition of $\ga$ into at most $4\ti{N}(v,D)+1$ geodesic segments such that each segment is either in a body or in $\overline{A_{\bar{r}_j^k, 2r_j^k} (x_j^k)}\subset \Nn_j^{k+1}$ for some neck $\Nn_j^{k+1}$. Moreover, there are at most $2$ geodesic segments in each $\overline{A_{\bar{r}_j^k, 2r_j^k} (x_j^k)}$.
\end{lemma}
\begin{proof}
Recall that in equations (\ref{equ3}) and (\ref{equ4}), for each neck $\Nn_j^{k+1}$, we have
\begin{equation}
A_{\bar{r}_j^k /2, 2 r_j^k} (x_j^k)   \subset \Nn_j^{k+1} \subset A_{(1-\ep)\bar{r}_j^k /2, 2(1+\ep) r_j^k} (x_j^k), \nonumber
\end{equation}
and for each body $\Bb_i^{k}$,  we have
\begin{equation}
\Bb_i^{k}=B_{2\bar{r}_i^{k-1}}(x_i^{k-1})\setminus \cup_j B_{r_j^{k}}(x_j^{k}). \nonumber
\end{equation}
In each neck $\Nn_j^{k+1}$, we can choose the first and the last intersection points of $\ga$  with $\overline{B_{r_j^k}(x_j^k)}$ and $\overline{B_{\bar{r}_j^k}(x_j^k)}$ respectively. Suppose that those points are $\{\ga(t_i)\}_{i=1}^K$ and $0\leq t_1 \leq t_1 \leq, \ldots, \leq t_K \leq 1$, then $\{\ga([t_i,t_{i+1}])\}_{i=0}^{K}$ with $t_0=0$ and $t_{K+1}=1$ form the partiton of $\ga$. Each segment $\ga([t_i,t_{i+1}])$ is either in a body or in the closure of $\overline{A_{\bar{r}_j^k, 2r_j^k} (x_j^k)}$ follows from Lemma~\ref{intersection point} our definition of the first and last intersection point. Note that by Lemma~\ref{lem2_c}, the number of necks is bound above by $\ti{N}(v,D)$ and there are at most $4$ intersection points in each neck. Hence, we partite $\ga$ into at most $4\ti{N}(v,D)+1$ pieces.
\end{proof}

We first show that any minimizing geodesic in a body region can be approximate by a curve in $\Ga$ with controlled simplicial length. For a body $\Bb^k_i$, let us denote
$$r(\Bb^k_i)=\inf\{r_h(x)|x \in \Bb^k_i\}.$$

\begin{lemma}\label{lm6}
 Let $\ga:[0,1]\ra M$ be a curve in $\Bb^k_i$ with length $l$. Then there exists a curve $\al \su \Ga$ with simplicial length bounded by $64 l/r(\Bb^k_i)$ such that $\ga$ is homotopic to $\al$ through a homotopy with width bounded by $9D$. In particular, if $\ga$ is a minimizing geodesic of $M$, then the simplicial length of $\al$ does not exceed $64/r_0(v,D)$, where $r_0(v,D)$ is defined in (1) of Theorem \ref{thm3_f}.
\end{lemma}

\begin{proof}
Suppose $\{B_{r(x_j)}(x_j)\}$ are the open sets in $\mathcal{O}$ that cover $\Bb^k_i$. Let us first choose a partition of the curve $\ga$ in the following way. We will take a partition $0=t_0<t_1<\dots<t_n=1$ of $[0,1]$ inductively. Let $\ga(0)=\ga(t_0)\in B_{r(x_0)}(x_0)$. If the length of $\ga\geq r(x_0)$, we choose $t_1$ such that $\length (\ga([t_0,t_1]))=r(x_0).$ Suppose we have already chosen $t_0,t_1,\dots,t_i$, and $\ga(t_i)\in B_{r(x_i)}(x_i)$ for some vertex $x_i$. If $\length(\ga([t_i,1]))\geq r(x_i)$, we choose $t_{i+1}$ such that $\length(\ga([t_i,t_{i+1}]))=r(x_i)$. (See Figure~\ref{fig9}.) Otherwise, we just choose $t_{i+1}=1$ and then $\ga([t_i,1])\su B_{r(x_i)}(x_i)$. Note that it is possible that $\ga(t_{i+1})\in B_{r(x_i)}(x_i)$. In this case we just let $x_{i+1}=x_i$ be the same point.

\begin{figure}[htbp]
\begin{minipage}[c]{0.47\textwidth}
\centering\includegraphics[width=7cm]{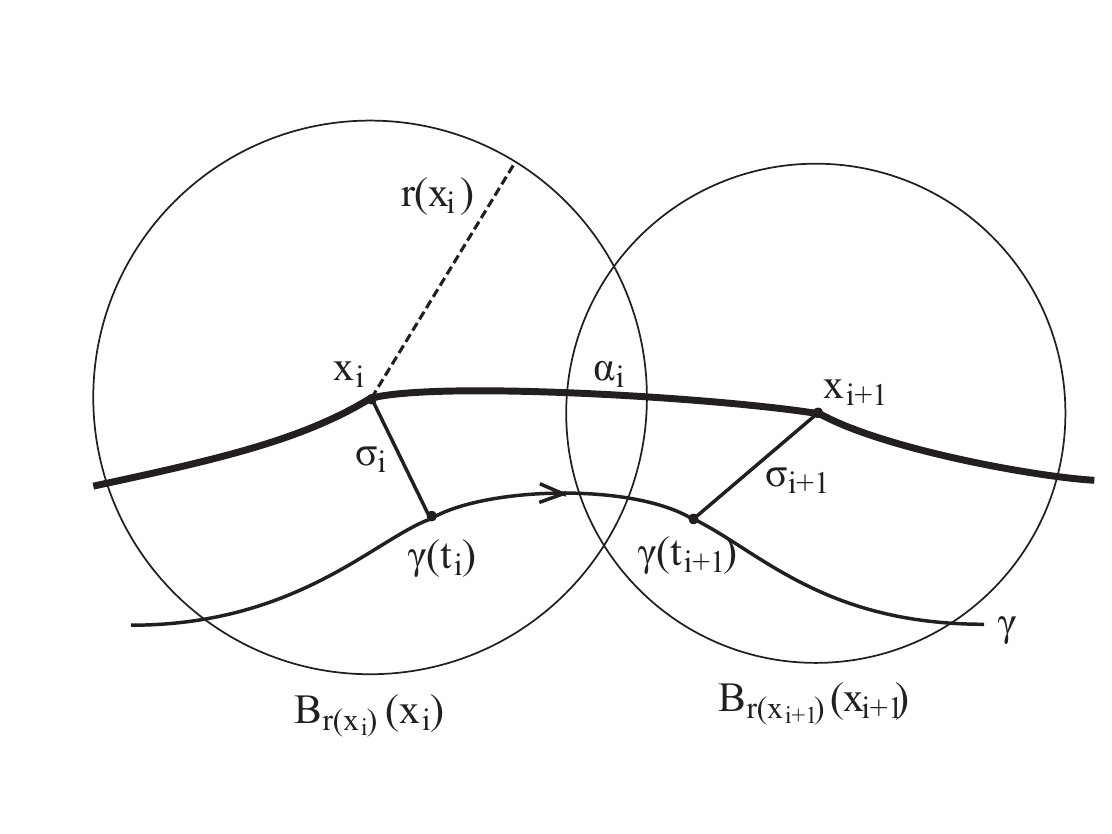}
\vspace*{6pt}
\caption{Pick $\ga(t_{i})$ so that $\length(\ga([t_i,t_{i+1}]))=r(x_i)$.}\label{fig9}
\end{minipage}
\end{figure}

For each arc $\ga([t_i,t_{i+1}])$, suppose that $\ga(t_i)\in B_{r(x_i)}(x_i)$ and $\ga(t_{i+1})\in B_{r(x_{i+1})}(x_{i+1})$. Let $\si_i$ (resp. $\si_{i+1}$) be a minimizing geodesic connecting $x_i$ (resp. $x_{i+1}$) and $\ga(t_i)$ (resp. $\ga(t_{i+1})$) and let $\al_i$ denote the edge in $\Ga$ connecting $x_i$ and $x_{i+1}$. If $x_i=x_{i+1}$, $\al_i$ is just a point curve.

We consider the four-gon $\ga([t_i,t_{i+1}])\cup \si_{i+1}\cup \al_i\cup \si_i$, (or triangle when $x_i=x_{i+1}$). Assume that $r({x_i})\geq r(x_{i+1})$, then this four-gon is contained in $B_{3r(x_i)}(x_i)\su B_{R(x_i)}(x_i)$, because $\length(\al_i)=d(x_i,x_{i+1})\leq \length(\si_i)+\length(\si_{i+1})+\length(\ga([t_i,t_{i+1}]))\leq 3r(x_i)$.

We take the approximating curve $\al$ to be $\al_1\cup\al_2\cup\dots\al_n$. Since each four-gon $\ga([t_i,t_{i+1}])\cup \si_{i+1}\cup \al_i\cup \si_i$ is contained in the ball $B_{R(x_i)}(x_i)$ (or $B_{R(x_{i+1})}(x_{i+1})$, if $r(x_{i+1})\geq r(x_i)$), by the same argument as Step 1, Step2 and Step 3 (1) in  Lemma~\ref{lm1}, $\ga$ is homotopic to $\al$ by a homotopy with width bounded by $2D+2D+5D =9D$.

Note that the simplicial length of $\al$ is bounded by the number of segements in the partition of $\ga$. By Lemma \ref {lem2_c}, $$r(x_j)=\frac{1}{32\sqrt{1+2\cdot10^{-3}}}\x r_{h}(x_j) \geq \frac{1}{32\sqrt{1+2\cdot10^{-3}}}\x r(\Bb^k_i)$$
Then the number of segements in the partition of $\ga$ is bounded by $$32\sqrt{1+2\cdot10^{-3}} l/r(\Bb^k_i) \leq 64l/r(\Bb^k_i) .$$
Now suppose that $\ga$ is a minimizing geodesic of $M$ contained in $\Bb_i^{k}$. Recall that for the body $\Bb_i^{k}$, we have
\begin{equation}
\Bb_i^{k}=B_{2\bar{r}_i^{k-1}}(x_i^{k-1})\setminus \cup_j B_{r_j^{k}}(x_j^{k}). \nonumber
\end{equation}
Thus, a minimizing geodesic of $M$ in $\Bb_i^{k}$ has length bounded by $2\bar{r}_i^{k-1}$. Moreover, by (1) of Theorem \ref{thm3_f}, $\diam(\Bb_i^{k})/r(\Bb^k_i) =2\bar{r}_i^{k-1}/r(\Bb^k_i) \leq 1/r_0(v,D)$. Hence, the conclusion follows.

\end{proof}

Given a minimizing geodesic $\ga$ of $M$ in a body, Lemma~\ref{lm6} tells us that we can partite $\ga$ into controlled numbers of pieces, so that each piece is contained in a contractible ball $B_{R(x_i)}(x_i)$. Consequently, we can homotope $\ga$ to a curve in $\Ga$ within the controlled width. However, such argument can not be applied to a minimizing geodesic in a neck. Indeed, the trapezoids in a neck are thin and long. So it is possible that a minimizing geodesic goes in and out a large trapezoid $\bar{T}^{k+1}_{j,i}$ many times which can not be bounded by any function of $v$ and $D$. To resolve this problem, we are going to use the geometry of the neck. Let us introduce the following result which is a combination of Proposition 3.22 in \cite{gromov2007metric} and Corollary 6.3 in \cite{nabutovsky2013length}.


\begin{lemma} \label{Gromov lemma}
Let $M$ be a closed Riemannian manifold with diameter $d$ and let $p \in M$. If the fundamental group $\pi_1(M,p)$ is finite with order $l$, then there exists generators $\{g_1, \ldots, g_K\}$  of $\pi_1(M,p)$ such that $\length(g_i) \leq 2d$, for $1 \leq i \leq K$. Moreover, any element $g \in \pi_1(M,p)$ can be represented in a word of those generators with the length of the word bounded by $l/2$.
\end{lemma}

With the above lemma, we can show that the a minimizing geodesic in the neck region is homotopic to a curve with bounded simplicial length in $\Ga$.

\begin{lemma}\label{approx in neck}
Suppose that $\ga:[0,1] \ra \overline{A_{\bar{r}_j^k, 2r_j^k} (x_j^k)}\subset \Nn_j^{k+1}$ for some neck $\Nn_j^{k+1}$. Then $\ga$ is path homotopic to a curve $\ga'$ through a path homotopy with width bounded by $7D$. Moreover,  there exists a curve $\al\su \Ga$ with simplicial length bounded by $\frac{20+2C(v,D)}{r_0(v,D)}+4$ such that $\ga'$ is homotopic to $\al$ through a homotopy with width bounded by $60D$. Here, the constants $C(v,D)$ and $r_0(v,D)$ are defined in (1) of Theorem \ref{thm3_f}.
\end{lemma}

\begin{proof}
First recall that the neck satisfies
\begin{equation}
A_{\bar{r}_j^k /2, 2 r_j^k} (x_j^k)   \subset \Nn_j^{k+1} \subset A_{(1-\ep)\bar{r}_j^k /2, 2(1+\ep) r_j^k} (x_j^k), \nonumber
\end{equation}
and there is a diffeomorphism $\Phi_j^{k+1}: A_{\bar{r}_j^k/2,2r_j^k}(0)\ra \Nn_j^{k+1}$, where  $0\in \R^4/\Ga_j^{k+1}$ and $\Ga_j^{k+1}\su O(4)$ with $|\Ga_j^{k+1}| \leq C(v,D)$. Moreover, if $g_{ij}={\Phi_j^{k+1}}^*g$ is the pullback metric, then we have equation (\ref{equ5})
\begin{equation}
||g_{ij}-\de_{ij}||_{C^0}+\bar{r}_j^k\cdot|| \dl g_{ij}||_{C^0}\leq \ep(v)<0.1. \nonumber
\end{equation}

Let $S_r(0)$ be the sphere of radius $r$ in $\R^4$ and $S_r(x)$ be sphere of radius $r$ at $x \in M$. We choose $\lambda$, such that
\begin{equation}
\bar{r}_j^k/2 < \lambda <\dist_{\R^4/\Ga_j^{k+1}}\big([\Phi_j^{k+1}]^{-1}(S_{\bar{r}_j^k}(x_j^k)),0\big) \quad \text{and}\quad \lambda<\bar{r}_j^k. \nonumber
\end{equation}
Note that we have $S_{\lambda}(0)/\Ga_j^{k+1} \subset  A_{\bar{r}_j^k/2,2\bar{r}_j^k}(0)$ and $\Phi_j^{k+1}(S_{\lambda}(0)/\Ga_j^{k+1})\su \overline{B_{\bar{r}_j^k}(x_j^k)} \cap\Nn_j^{k+1} $.

Let us define a deformation retraction $H(x,t)$ of $A_{\bar{r}_j^k/2,2r_j^k}(0)$ onto $S_{\lambda}(0)/\Ga_j^{k+1}$ along the radical direction:
\[
H(x,t)=\frac{x}{\|x\|}[(\lambda-\|x\|)t+\|x\|],  \quad \text{for } x \in A_{\bar{r}_j^k/2,2r_j^k}(0) \text{ and } 0 \leq t \leq 1.
\]

Define $\ti{\ga}(t):[0,1] \ra M$ as follows:
\[
\ti{\ga}(t)=
\begin{cases}
  \Phi_j^{k+1} \circ H ([\Phi_j^{k+1}]^{-1}(\ga(0)),3t)  \quad 0 \leq t \leq \frac{1}{3}, \\
  \Phi_j^{k+1} \circ H ([\Phi_j^{k+1}]^{-1}(\ga(t)),1)  \quad \frac{1}{3} \leq t \leq \frac{2}{3},  \\
  \Phi_j^{k+1} \circ H ([\Phi_j^{k+1}]^{-1}(\ga(1)),3-3t) \quad \frac{2}{3} \leq t \leq 1.
\end{cases}
\]
Obviously, $[\Phi_j^{k+1}]^{-1}(\ti{\ga})$  is path homotopic to $[\Phi_j^{k+1}]^{-1}(\ga)$ through a straight line homotopy in radical direction in $A_{\bar{r}_j^k/2,2r_j^k}(0)$. The width of this homotopy is bounded by $2r_j^k$. Thus, by equation (\ref{equ5}) $\ga$ is path homotopic to $\ti{\ga}$ within width $\frac{2r_j^k}{\sqrt{1-\ep(v)}} \leq 3D$.
Note that $\ti{\ga}([\frac{1}{3},\frac{2}{3}]) \subset \Phi_j^{k+1}(S_{\lambda}(0)/\Ga_j^{k+1})$.

\

\textbf{Claim.} Any curve $\ga:[0,1] \ra S_{\lambda}(0)/\Ga_j^{k+1}$ is path homoptic to a curve $\si$ in $S_{\lambda}(0)/\Ga_j^{k+1}$ with $\length(\si) \leq \lambda \pi (C(v,D)+2)$. Moreover, the width of the homotopy is bounded by $\pi \lambda$.

\begin{proof}[Proof of the claim.]
Let us connect $\ga(0)$ and $\ga(1)$ by a minimizing geodesic $g_0$ in $S_{\lambda}(0)/\Ga_j^{k+1}$. Then $\ga g^{-1}_0$ is a loop based at $\ga(0)$. By Lemma~\ref{Gromov lemma}, $\ga g^{-1}_0$ is path homotopic to $g_1g_2 \cdots g_l$, where $g_i \in \pi_1(S_{\lambda}(0)/\Ga_j^{k+1}, \ga(0))$, for $1 \leq i \leq l$ and $l \leq |\Ga_j^{k+1}|/2$.

Furthermore, $\length(g_i) \leq 2 \diam(S_{\lambda}(0)/\Ga_j^{k+1})\leq 2 \pi \lambda $, for $0 \leq i \leq l$. Thus, $\ga$ is path homotopic to $\si=g_1g_2 \cdots g_l g_0$. The $\length(\si) \leq 2\pi \lambda ( |\Ga_j^{k+1}|/2+1) \leq \lambda \pi (C(v,D)+2)$.

Now, $\ga \si^{-1}$ is a homotopically trivial loop in $\pi_1(S_{\lambda}(0)/\Ga_j^{k+1}, \ga(0))$. We lift $\ga$ and $\si$ to the universal covering space $S_{\lambda}(0)$ by a local isometry $\phi$. $\phi(\ga)$ and $\phi(\si)$ are path homotopic in $S_{\lambda}(0)$ within width bounded $\pi \lambda$. Therefore, the covering map induces a path homotopy between $\ga$ and $\si$ with width bounded $\pi \lambda$. This proves the claim.
\end{proof}
Now, it follows from the above claim and equation (\ref{equ5})  that there is a curve $\tilde{\si}: [\frac{1}{3},\frac{2}{3}]\ra \Phi_j^{k+1}(S_{\lambda}(0)/\Ga_j^{k+1})$ such that $\length(\tilde{\si}) \leq \frac{\lambda \pi (C(v,D)+2)}{\sqrt{1-\ep(v)}} \leq 4\lambda (C(v,D)+2)$. Moreover, $\tilde{\si}$
is path homotopic to $\ti{\ga}[\frac{1}{3},\frac{2}{3}]$ within width bounded by $\frac{\pi \lambda}{\sqrt{1-\ep(v)}} \leq 4\lambda$.

\begin{figure}[htbp]\label{projection in neck}
\centering\includegraphics[width=7cm]{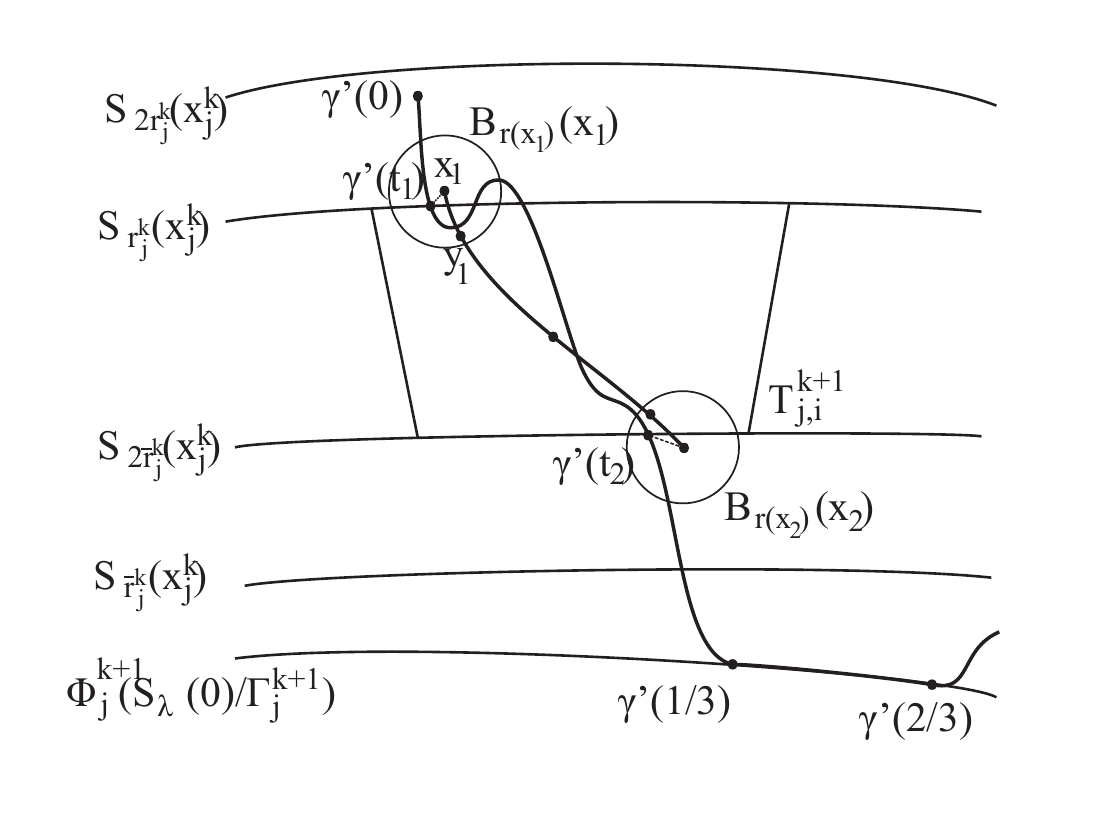}
\vspace*{6pt}
\caption{The curve $\ga'$ and its approximation.}\label{projection in neck}
\end{figure}

Define $\ga'(t):[0,1] \ra M$ as follows: $\ga'(t)=\tilde{\ga}(t)$ for $0 \leq t \leq \frac{1}{3}$, $\ga'(t)=\tilde{\si}(t)$ for $\frac{1}{3} \leq t \leq \frac{2}{3}$,  and $\ga'(t)=\tilde{\ga}(t)$ for $\frac{2}{3} \leq t \leq 1$. Since $\lambda<\bar{r}_j^k\leq D$, $\ga$ is path homotopic to $\ga'$ within width bounded by $3D+4\lambda \leq 7D$.

Without loss of generality, assume that  $\ga(0) \in S_{2r_j^k}(x_j^k)$. By our construction, $\ga'([0,\frac{1}{3}]) \su \bar{T}^{k+1}_{j,i}$. Let us pick a point $t_1$ such that $\ga'([0,t_1)) \su M \setminus B_{r^k_j}(x_j^k)$ and a point $t_2$ such that $\ga'((t_2,\frac{1}{3}]) \su B_{2\bar{r}_j^k}(x_j^k)$. (In the general case when $\ga(0) \not \in S_{2r_j^k}(x_j^k)$, $t_1$ may or may not exist. But the discussion is the same.) With the same method as in Lemma \ref{lm1}, $\ga'([t_1,t_2])$ is homotopic to two edges in $\Ga$ through a homotopy with width bounded by $60D$.(See Figure~\ref{projection in neck}).

We estimate the length of $\ga'([0,t_1])$. Note that $\ga'(0)$ and $\ga'(t_1)$ can be connected by a curve $\eta_1$ in $\Nn_j^{k+1}$ with length bounded by $2r_j^k$. Hence,
\begin{equation}
\length ([\Phi_j^{k+1}]^{-1}\big( \ga'([0,t_1]) \big)) \leq \length \big( [\Phi_j^{k+1}]^{-1}(\eta_1) \big) \leq \sqrt{1+2\ep(v)} \length(\eta_1) \leq 4r_j^k. \nonumber
\end{equation}
And $\length ( \ga'([0,t_1])) \leq \frac{4r_j^k}{\sqrt{1-\ep(v)}} \leq 8r_j^k$. Note that $\ga'([0,t_1])$ is in the body $\Bb_l^{k}$ and $\diam(\Bb_l^{k}) >2r_j^k$. By Lemma \ref{lm6} and (1) of Theorem \ref{thm3_f}, $\ga'([0 , t_1])$ is homotopic a curve in $\Ga$ with simplicial length bounded by $\frac{4}{r_0(v,D)}$ through a homotopy with width bounded by $9D$.

Next, we estimate the length of $\ga'([t_2,\frac{1}{3}])$. Let $\eta_2$ be a curve that realizes the distance in $M$ between $\ga'(t_2)$ and $\Phi_j^{k+1}(S_{\lambda}(0)/\Ga_j^{k+1})$, then $\eta_2 \su \Nn_j^{k+1}$ and $\length(\eta_2) \leq 2\bar{r}_j^k (x_j^k)$. Hence,
\begin{equation}
\length ([\Phi_j^{k+1}]^{-1}\big( \ga'([t_1,\frac{1}{3}]) \big)) \leq \length \big( [\Phi_j^{k+1}]^{-1}(\eta_2) \big) \leq \sqrt{1+2\ep(v)} \length(\eta_2) \leq 4\bar{r}_j^k (x_j^k). \nonumber
\end{equation}
And $\length ( \ga'([t_2,\frac{1}{3}])) \leq \frac{4\bar{r}_j^k (x_j^k)}{\sqrt{1-\ep(v)}} \leq 8\bar{r}_j^k (x_j^k)$.

 Similarly, assume $\ga(1) \in S_{r_j^k}(x_j^k)$, we find the point $t_4$ such that $\ga'((t_4,1]) \su M \setminus B_{r^k_j}(x_j^k)$ and the point $t_3$ such that $\ga'((t_3,\frac{1}{3}]) \su B_{2\bar{r}_j^k}(x_j^k)$. Hence, $\ga'([t_3,t_4])$ is homotopic to two edges in $\Ga$ through a homotopy with width bounded by $60D$. Same argument shows that $\ga'([t_4 , 1])$ is homotopic a curve in $\Ga$ with simplicial length bounded by $\frac{4}{r_0(v,D)}$ through a homotopy with width bounded by $9D$. And we have $\length (\ga'([\frac{2}{3},t_3])) \leq 8\bar{r}_j^k (x_j^k)$. Therefore,
 $$\length (\ga'([t_2 , t_3]))  \leq 8\bar{r}_j^k (x_j^k)+8\bar{r}_j^k (x_j^k)+4\lambda (C(v,D)+2) \leq 16\bar{r}_j^k (x_j^k)+4\bar{r}_j^k (x_j^k)(C(v,D)+2)$$

 Note that $\ga'([t_2 , t_3]) \su \Bb_j^{k+1}$ and $\diam( \Bb_j^{k+1}) = 2\bar{r}_j^k (x_j^k)$. By Lemma \ref{lm6} and (1) of Theorem \ref{thm3_f}, $\ga'([t_2 , t_3])$ is homotopic a curve in $\Ga$ with simplicial length bounded by $\frac{12+2C(v,D)}{r_0(v,D)}$ through a homotopy with width bounded by $9D$. If we take  $\ga'([0,t_1])$, $\ga'([t_1,t_2])$, $\ga'([t_3,t_4])$  and $\ga'([t_4,1])$ into account, the conclusion follows.
\end{proof}

Now we can apply the previous lemmas to obtain a controlled simplicial approximation for minimizing geodesics.

\begin{lemma}\label{approx minimizing geodesic}
 Let $\ga:[0,1]\ra M$ be a minimizing geosdesic. Then there exists a curve $\al \su \Ga$ with simplicial length bounded by $(\frac{360+4C(v,D)}{r_0(v,D)}+8)\ti{N}(v,D)$ such that $\ga$ is homotopic to $\al$ through a homotopy with width bounded by $67D$.
\end {lemma}
\begin{proof}
First, we apply Lemma~\ref{partition of geodesic} to partite $\ga$ into geodesic segements $\{\ga_i\}$ in the bodies and the region $\overline{A_{\bar{r}_j^k, 2r_j^k} (x_j^k)}\subset \Nn_j^{k+1}$ in each neck. There are at most $4\ti{N}(v,D)+1$ geodesic segements in the bodies and at most $2\ti{N}(v,D)$ segements in the regions $\{\overline{A_{\bar{r}_j^k, 2r_j^k} (x_j^k)}\}$.

For each geodesic segment $\ga_i$ in the body, we apply Lemma \ref {lm6}. So, there is a $\al_i \su \Ga$ with simplicial length bounded by $64/r_0(v,D)$ and $\ga_i$ is homotopic to $\al_i$ with width bounded by $9D$. For each geodesic segment $\ga_i$ in $\overline{A_{\bar{r}_j^k, 2r_j^k} (x_j^k)}$, we apply Lemma \ref{approx in neck}.  So, there is a $\al_i \su \Ga$ with simplicial length bounded by $\frac{20+2C(v,D)}{r_0(v,D)}+4$ and $\ga_i$ is homotopic to $\al_i$ with width bounded by $60D+7D=67D$. If we combine all the pieces together, we have $\ga=\cup_i \ga_i$ is homotopic to
$\al=\cup_i \al_i$ with width bounded by $67D$. The simplicial length of $\al$ is bounded by $(4\ti{N}(v,D)+1)\times 64/r_0(v,D)+ (\frac{20+2C(v,D)}{r_0(v,D)}+4) \times 2\ti{N}(v,D)  \leq (\frac{360+4C(v,D)}{r_0(v,D)}+8)\ti{N}(v,D)$.
\end{proof}

We will now proceed to general curves. Similar to the case of minimizing geodesic, we are going to first introduce a partition of a general curve in $M$ in Lemma~\ref{length reduction}(1). In Lemma~\ref{length reduction}(2) and (3) we will show some rough estimate of the relation between the number of the segments in the partition and the length of the curve.

\begin{lemma} \label{length reduction}
\

\begin{enumerate}
\item For any curve $\ga:[0,1] \ra M$, there is a partition $\mathcal{P}$ of $\ga=\cup_i \ga_i$, such that each $\ga_i$ is either in a body or in $\overline{A_{\bar{r}_j^k, 2r_j^k} (x_j^k)}\subset \Nn_j^{k+1}$ for some neck. Based on this partition, one can construct a simplicial curve $\ti{\ga}:[0,1]\ra $ in $\Ga$, such that $\ga$ is homotopic to $\ti{\ga}$ with width bound by $67D$.
\item
 If there is a $\overline{A_{\bar{r}_j^k, 2r_j^k} (x_j^k)}$ which contains at least $11$ segments of the partition $\mathcal{P}$ , then there is a curve $\ga'$ such that $\ga(0)=\ga'(0)$ and $\ga(1)=\ga'(1)$. Moreover, $\length(\ga') \leq \length(\ga)-\bar{r}_j^k$.
\item
Suppose that $\ga:[0,1] \ra \Bb^k_i$ is a curve in some body. We approximate $\ga$ by a simplicial curve $\al$ in $\Ga$ as descibed in Lemma \ref{lm6}. If there is an edge in $\al$ that appears more than $4$ times, then there is a curve $\ga'$ such that $\ga(0)=\ga'(0)$ and $\ga(1)=\ga'(1)$. Moreover, $\length(\ga') \leq \length(\ga)-\frac{1}{32\sqrt{1+2\cdot10^{-3}}}\cdot r_h$, where $r_h$ is the harmonic radius of $M$.
\end{enumerate}
\end{lemma}

\begin{remark}
The simplicial length of the curve $\ti{\ga}$ in Lemma~\ref{length reduction}(1) may not be bounded by any function of $v$ and $D$.
\end{remark}

\begin{proof}

\

\begin{enumerate}
\item
We define the partition $\mathcal{P}$ by taking a partition $0=t_0<t_1<\dots<t_n=1$ of $[0,1]$ inductively. Let $\ga(0)=\ga(t_0)$. For an odd number $i$, choose $t_i$ such that $\ga(t_i)$ on the boundary of $\overline{A_{2\bar{r}_{j_i}^{k_i}, r_{j_i}^{k_i}} (x_{j_i}^{k_i})}$ and $\ga([t_{i-1},t_i])$ is  either in $\overline{A_{2\bar{r}_{j_i}^{k_i}, r_{j_i}^{k_i}} (x_{j_i}^{k_i})}$ or its complement $\overline{A_{2\bar{r}_{j_i}^{k_i}, r_{j_i}^{k_i}} (x_{j_i}^{k_i})} ^c$. Moreover, $\ga([t_{i-1},t_i])$ does not intersect any other $\overline{A_{2\bar{r}_{j}^{k}, r_{j}^{k}} (x_j^k)}$. We then choose $\ga(t_{i+1})$ on the boundary of $\overline{A_{\bar{r}_{j_i}^{k_i}, 2r_{j_i}^{k_i}} (x_{j_i}^{k_i})}$ such that $\ga([t_{i},t_{i+1}]) \su \overline{A_{\bar{r}_{j_i}^{k_i}, 2r_{j_i}^{k_i}} (x_{j_i}^{k_i})}$.

Note that $\ga([t_0,t_1])$ and $\ga([t_{n-1},t_n])$ can be either in a body or some $\overline{A_{\bar{r}_j^k, 2r_j^k} (x_j^k)}$. For the rest of the segments, $\ga([t_{i},t_{i+1}])$ is contained in some $\overline{A_{\bar{r}_j^k, 2r_j^k} (x_j^k)}$ for $i$ odd, and $\ga([t_{i},t_{i+1}])$ is in some body for $i$ even. Now by Lemma~\ref{lm6} and Lemma~\ref{approx in neck}, one can piecewise homotope $\ga$ to a curve $\ti{\ga}$ in $\Ga$ through a homotopy of width bounded by $67D$.

\item
Now, suppose that there is a $\overline{A_{\bar{r}_j^k, 2r_j^k} (x_j^k)}$ which contains at least $11$ segments in the partition $\mathcal{P}$ of $\ga$. Note that $\ga(0)$ and $\ga(1)$ may not be on the boundaries of $\overline{A_{\bar{r}_j^k, 2r_j^k} (x_j^k)}$ or $\overline{A_{2\bar{r}_j^k, r_j^k} (x_j^k)}$. Thus, among those segments, there are at least 9 of them whose endpoints are on the boundries of the annuli. There are four possible types.

We define the segments to be of Type I, if one of their endpoint is in the boundary of the ball $B_{\bar{r}_j^k}(x_j^k)$. Similarly, we define the segments to be of Type II, if if one of their endpoint is in the boundary of the ball $B_{2{r}_j^k}(x_j^k)$. (See Figure~\ref{type I (1)} to Figure~\ref{type II (2)})

\begin{figure}[htbp]
\centering
\begin{minipage}[c]{0.47\textwidth}
\centering
\includegraphics[width=7cm]{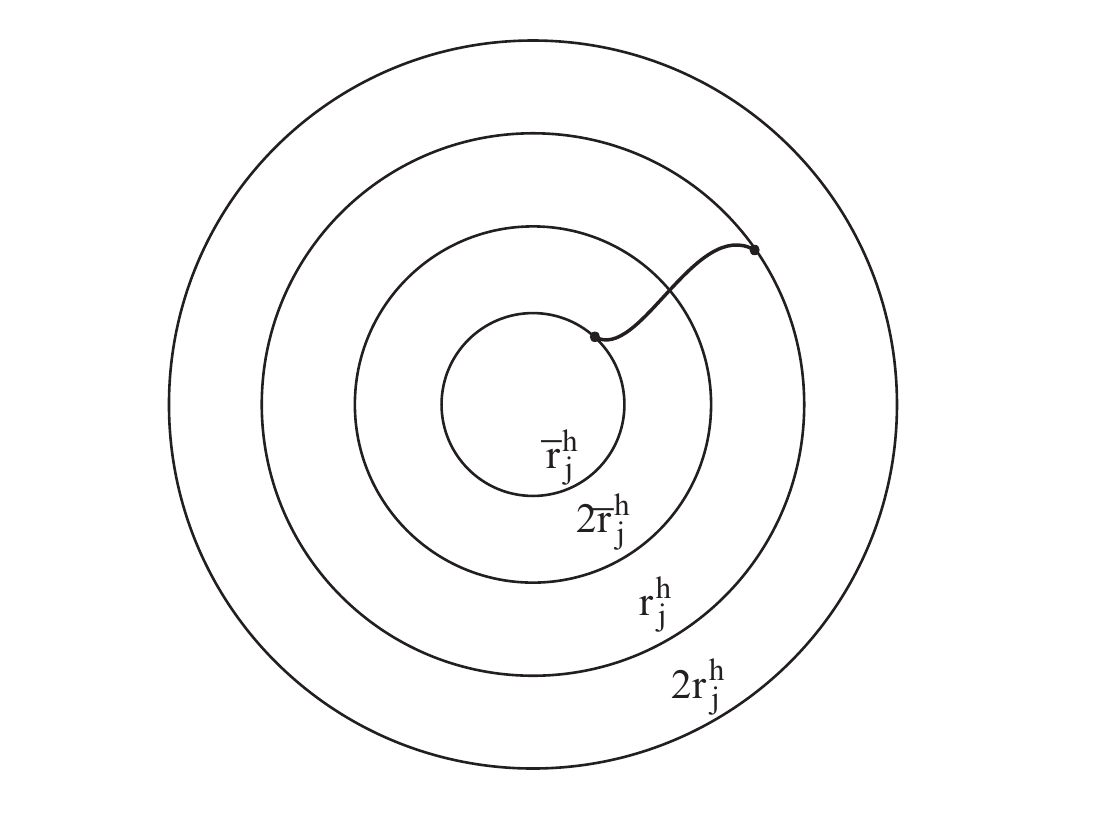}
\vspace*{6pt}
\caption{Type I segment (1).}\label{figty}\label{type I (1)}
\end{minipage}
\begin{minipage}[c]{0.47\textwidth}
\centering
\includegraphics[width=7cm]{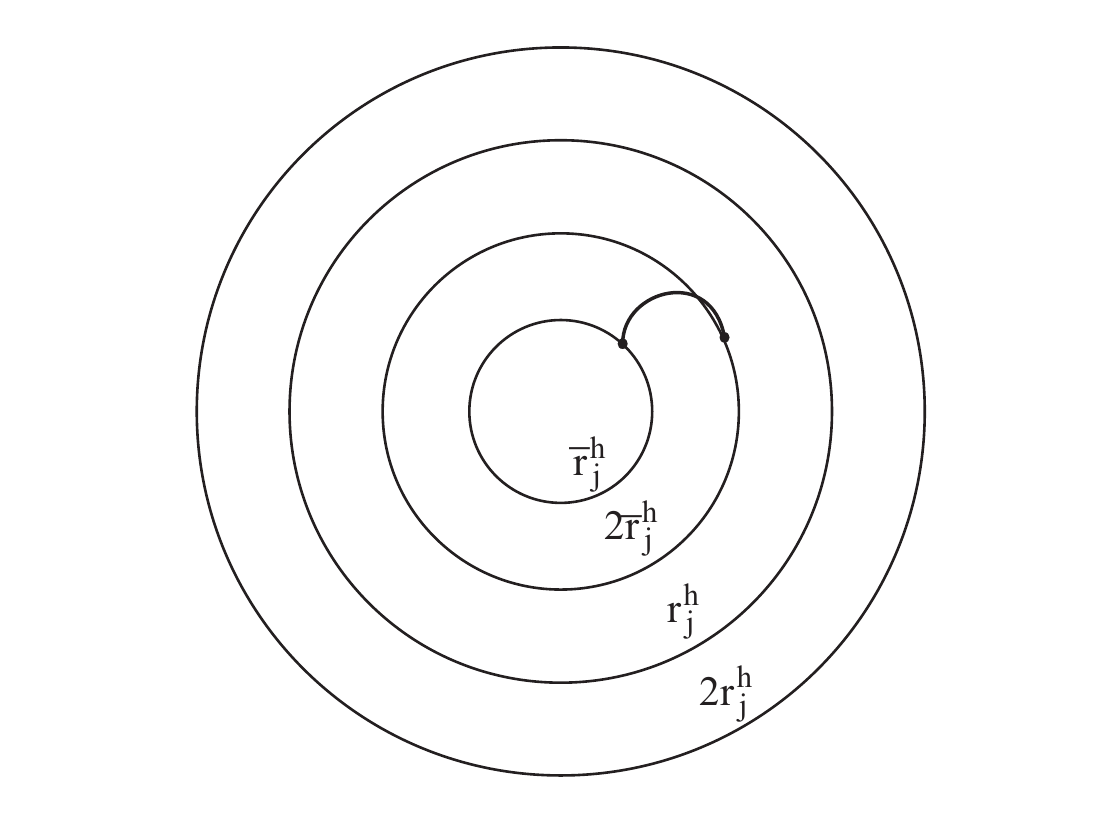}
\vspace*{6pt}
\caption{Type I segment (2)}\label{type I (2)}
\end{minipage}
\end{figure}

\begin{figure}[htbp]
\begin{minipage}[c]{0.47\textwidth}
\centering\includegraphics[width=7cm]{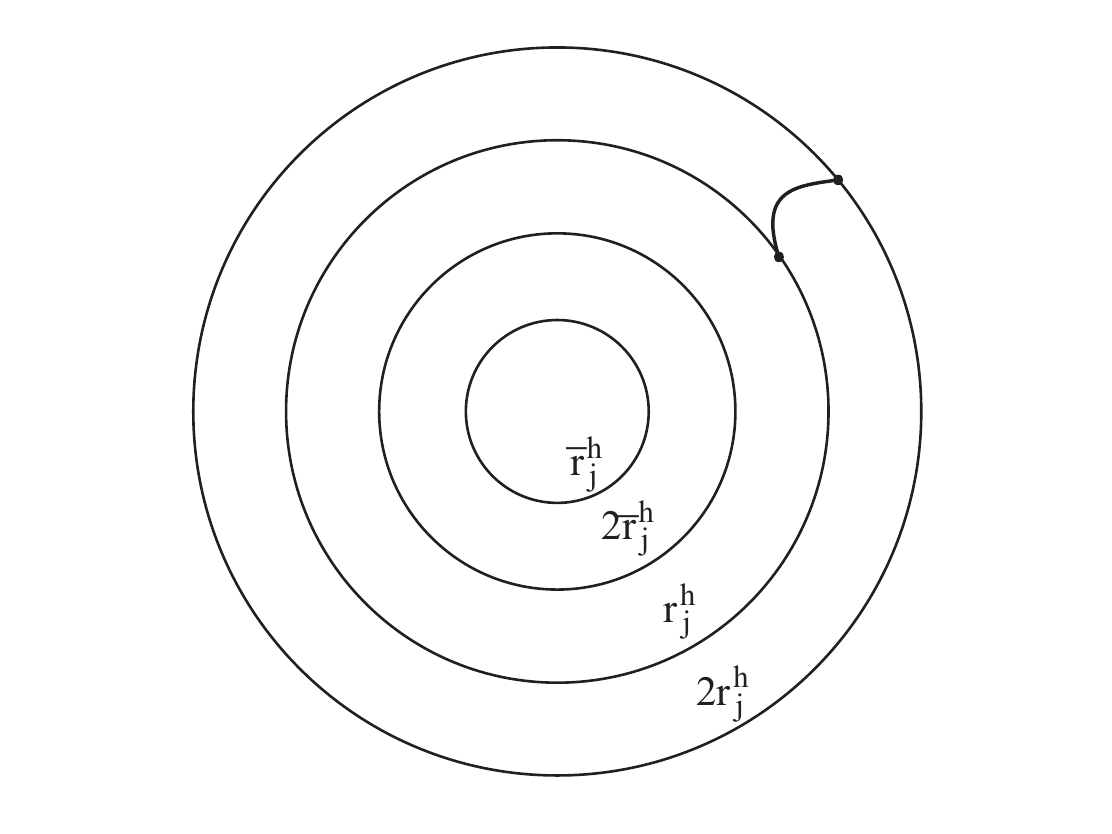}
\vspace*{6pt}
\caption{Type II segment (1)} \label{type II (1)}
\end{minipage}
\begin{minipage}[c]{0.47\textwidth}
\centering\includegraphics[width=7cm]{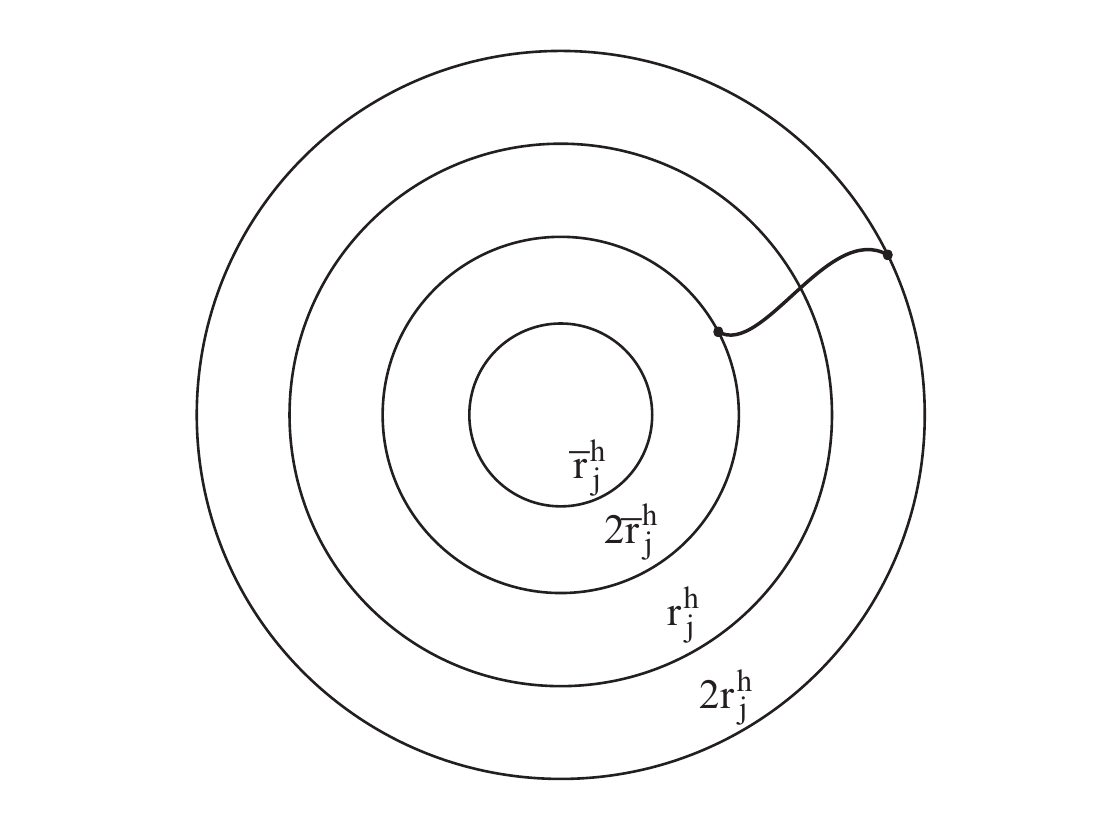}
\vspace*{6pt}
\caption{Type II segment (2)}\label{type II (2)}
\end{minipage}
\end{figure}

\begin{figure}[htbp]
\begin{minipage}[c]{0.47\textwidth}
\centering\includegraphics[width=7cm]{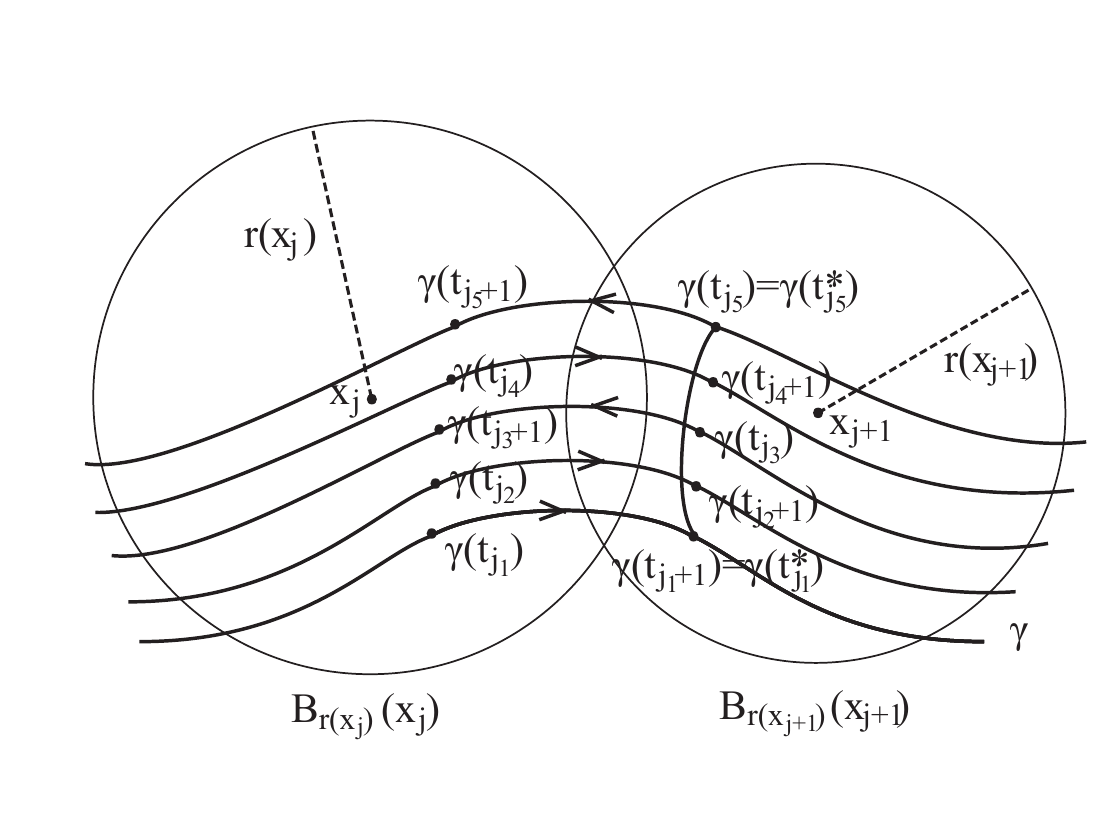}
\vspace*{6pt}
\caption{The distance $d(\ga(t_{j_1}^*), \ga(t_{j_5}^*))\leq 2r(x_{j+1})$.}\label{fig10}
\end{minipage}
\end{figure}

The segments of Type I have lengths greater than or equal to $\bar{r}_j^k$, while the segments of Type II have lengths greater than or equal to $r_j^k$. Among those $9$ segments , at least $5$ of them are of Type I or of Type II.  Suppose there are $5$ segments of Type I, namely, $\ga([t_{k_1},t_{k_2}]), \ldots, \ga([t_{k_9},t_{k_{10}}])$, where $t_{k_1} <, \cdots, < t_{k_{10}}$. Moreover, suppose  that $\ga(t^*) \in \{\ga(t_{k_1}),\ga(t_{k_2}) \}$ is the one with the shorter distance to $x_j^k$, and $\ga(t^{**})\in \{\ga(t_{k_9}),\ga(t_{k_{10}}) \}$ is the one with the shorter distance to $x_j^k$. Define $\ga'_1$ to be the minimizing geodesic from $\ga(t^*)$ to $x_j^k$. Define $\ga'_2$ to be the minimizing geodesic from $x_j^k$ to $\ga(t^{**})$. Since $\ga([t^*,t^{**}])$ contains at least $3$ segments of Type I, $\length(\ga'_1 \cup \ga'_2) =2\bar{r}_j^k <3 \bar{r}_j^k \leq \length(\ga([t^*,t^{**}]))$. Hence, define $\ga'=\ga([0,t^*]) \cup \ga'_1 \cup \ga'_2 \cup \ga([t^*,1])$ and the conclusion follows. It can be proved similarly for the case where there are $5$ segments of Type II.

\item
Suppose $\{B_{r(x_j)}(x_j)\}$ are the open sets in $\mathcal{O}$ that cover $\Bb^k_i$. Let $0=t_0<t_1<\dots<t_n=1$ be the partition of $[0,1]$ which is described in Lemma \ref{lm6}. Moreover, suppose that $\al_j$ connecting $x_j$ and $x_{j+1}$ is an edge in $\al$ and it appears more than 5 times in $\al$. Without loss of generality, assume $r(x_j)\geq r(x_{j+1})$. By our construction, there are at least 5 segments $\ga([t_{j_k},t_{j_k+1}]),k=1,2,\dots,5$ which are approximated by the edge $\al_j$. Assume that $t_{j_1}<t_{j_1+1}<t_{j_2}<\dots<t_{j_5+1}$, then the arc $\ga([t_{j_1+1},t_{j_5}])$ contains three segments $\{\ga([t_{j_k},t_{j_k+1}])\}_{k=2}^4$. Let $t_{j_k}^*\in\{t_{j_k},t_{j_k+1}\}$ such that $\ga(t_{j_k}^*)\in B_{r(x_{j+1})}(x_{j+1})$. (See Figure~\ref{fig10}.) Then $\length(\ga([t_{j_1}^*,t_{j_5}^*]))\geq \length(\ga([t_{j_1+1},t_{j_5}]))\geq 3r(x_{j+1})$. Define $\ga'_1$ to be the minimizing geodesic from $\ga(t_{j_1}^*)$ to $x_{j+1}$. Define $\ga'_2$ to be the minimizing geodesic from $x_{j+1}$ to $\ga(t_{j_5}^*)$. Then $\length(\ga'_1 \cup \ga'_2) =2r(x_{j+1}) <3r(x_{j+1}) \leq \length(\ga([t_{j_1}^*,t_{j_5}^*]))$. Hence, we define $\ga'=\ga([0,t_{j_1}^*]) \cup \ga'_1 \cup \ga'_2 \cup \ga([\ga(t_{j_5}^*),1])$. Note that for any $i$, $r(x_i) \geq \frac{1}{32\sqrt{1+2\cdot10^{-3}}}\cdot r_h$ and the conclusion follows.
\end{enumerate}
\end{proof}

The following Lemma~\ref{lm7} apply to a closed curve in $M$ and can be viewed as a generalization of Lemma~\ref{lm2}.

\begin{lemma}\label{lm7}
Let $Z(v,D)=100\ti{N}^3(\frac{20+2C(v,D)}{r_0(v,D)}+4)+10\ti{N}^2$. Let $\gamma:[0,1]\rightarrow M$ be a closed curve in $M$ which is not a closed geodesic. Let $p \in \gamma$ be a base point. Then $\gamma$ is homotopic to the wedge of some curves $\gamma_1 \vee \gamma_2 \vee\dots\vee\gamma_{k}$ based at $p$ with width bounded by $2D$. Each $\gamma_i$ is homopoted to a simplicial curve $\ti{\ga}_i \su \Ga$ with width bounded by $67D$. The simplicial length of the approximation $\ti{\ga_i}$ of each $\ga_i$ is bounded by $Z(v,D)$. Moreover, suppose that $\mathcal{P}$ is a partition of $\ga$ as in Lemma~\ref{length reduction} (1), then each $\ga_i$ satisfies the following properties.
\begin{enumerate}
\item
If there is a $\overline{A_{\bar{r}_j^k, 2r_j^k} (x_j^k)}\subset \Nn_j^{k+1}$ which contains at least $22$ segments of the partition $\mathcal{P}$, then the length of each $\gamma_i$ is less than or equal to $\length(\gamma)-\frac{1}{32\sqrt{1+2\cdot10^{-3}}}\cdot r_h/2$, where $r_h$ is the harmonic radius of $M$.
\item
Suppose that each $\overline{A_{\bar{r}_j^k, 2r_j^k} (x_j^k)}$ contains no more than $21$ segments of the partition $\mathcal{P}$. If the simplicial length $m(\tilde{\gamma})$ of $\tilde{\gamma}$ exceeds $Z$, then then the length of each $\gamma_i$ is less than or equal to $\length(\gamma)-\frac{1}{32\sqrt{1+2\cdot10^{-3}}}\cdot r_h/2$.
\end{enumerate}
\end{lemma}

\begin{proof}
Let $\ti{C}=\frac{1}{32\sqrt{1+2\cdot10^{-3}}}$ and $\ga(0)=p$. Let $0=s_0=s_{k+1}<s_1\dots<s_k=1$ be a partition of $[0,1]$ such that $\length(\gamma([s_j,s_{j+1}])\leq \tilde{C}\cdot r_h/4$. For each $j=1,2,\dots,k$, we connect $\gamma(0)$ and $\gamma(s_j)$ by a minimizing geodesic $\sigma_j$ and denote by $\gamma_j$ the loop $\sigma_j\cup \gamma([s_j,s_{j+1}])\cup(-\sigma_{j+1})$. With the same argument as in Lemma~ \ref{lm2}, $\gamma$ is homotopic to $\gamma_1\cup \gamma_2\cup\dots\cup \gamma_k$ through a homotopy with width bounded by $2\max_i(\length (\sigma_i))\leq 2D$.

By the assumption in Lemma \ref{lem2_c}, $\tilde{C}\cdot r_h/4 <\bar{r}_j^k$, hence $\gamma([s_j,s_{j+1}])$ is either in a body or in some $\overline{A_{\bar{r}_j^k, 2r_j^k} (x_j^k)}$. If $\gamma([s_j,s_{j+1}])$ is in a body, then it is homotopic to one edge $\ti{\ga}_{j,j+1}$ in $\Ga$ with width bounded by $9D$ . If $\gamma([s_j,s_{j+1}])$ is in some $\overline{A_{\bar{r}_j^k, 2r_j^k} (x_j^k)}$, then it is homotopic to a simplicial curve $\ti{\ga}_{j,j+1}$ in $\Ga$ with width bounded by $67D$ and the simplicial length of $\ti{\ga}_{j,j+1}$ is bounded by $\frac{20+2C(v,D)}{r_0(v,D)}+4$. By Lemma \ref{approx minimizing geodesic}, $\sigma_j$ is homotopic to a simplicial curve $\ti{\si}_j \su \Ga$  with width bounded by $67D$, and the simplicial length of $\ti{\si}_j$ is bounded by $(\frac{360+4C(v,D)}{r_0(v,D)}+8)\ti{N}(v,D)$. Define $\ti{\ga}_j=\ti{\si}_j \cup \ti{\ga}_{j,j+1} \cup \ti{\si}_{j+1}$. The simplicial length of $\ti{\ga}_j$ is bounded by
$$2 \cdot  (\frac{360+4C(v,D)}{r_0(v,D)}+8)\ti{N}(v,D)+ \frac{20+2C(v,D)}{r_0(v,D)}+4 \leq  (\frac{740+10C(v,D)}{r_0(v,D)}+20)\ti{N}(v,D) \leq Z(v,D). $$

Next, we prove that, for both case (1) and (2) in the statement of the lemma, we have
\begin{equation} \label{sigma j}
\length(\sigma_j) \leq \max\{\length(\gamma[0,s_j]), \length(\gamma[s_j,1])\}-\tilde{C}\cdot r_h/2,
\end{equation}
and
\begin{equation}\label{sigma j+1}
\length(\sigma_{j+1}) \leq \max\{\length(\gamma[0,s_{j+1}]), \length(\gamma[s_{j+1},1])\}-\tilde{C}\cdot r_h/2.
\end{equation}

In both cases, we will only show equation (\ref{sigma j}). Equation (\ref{sigma j+1}) can be proved in the same way.

Case~(1): Suppose that the inequality (\ref{sigma j}) fails, then both $\gamma([0,s_j])$ and $\gamma([s_j,1])$ has length bounded by $\length(\sigma_j)+\tilde{C}\cdot r_h/2=d(\gamma(0),\gamma(s_j))+\tilde{C}\cdot r_h/2$. Under the assumption of (1), one of $\gamma([0,s_j])$ and $\gamma([s_j,1])$ contains at least $11$ segments in $\overline{A_{\bar{r}_j^k, 2r_j^k} (x_j^k)}$. Suppose that $\gamma([0,s_j])$ does. Then by Lemma \ref{length reduction}(2), there is a curve $c$ connecting $\ga(0)$ and $\ga(s_j)$ such that
\begin{equation}
\length(c) \leq \length(\ga([0,s_j]))-\bar{r}_j^k \leq \length(\sigma_j)+\tilde{C}\cdot r_h/2-\bar{r}_j^k. \nonumber
\end{equation}
By assumption in Lemma \ref{lem2_c}, $\tilde{C}\cdot r_h/2 <\bar{r}_j^k$. Thus, we have $\length(c)<\length(\sigma_j)$, which contradicts to that $\sigma_j$ is minimizing.

Case~(2): Suppose that the inequality (\ref{sigma j}) fails, then both $\gamma([0,s_j])$ and $\gamma([s_j,1])$ has length bounded by $\length(\sigma_j)+\tilde{C}\cdot r_h/2=d(\gamma(0),\gamma(s_j))+\tilde{C}\cdot r_h/2$.  Under the assumption of (2), one of $\gamma([0,s_j])$ and $\gamma([s_j,1])$ contains a curve $\al$ that is approximated by a simplicial subcurve $\tilde{\al} \su \ti{\ga}$ such that the simplicial length of $\ti{\al}$ is bounded by $50\ti{N}^3(\frac{20+2C(v,D)}{r_0(v,D)}+4)+5\ti{N}^2$.  Assume that $\al \su\gamma([0,s_j])$.  Since there are at most $\ti{N}^2$ edges in $\Ga$, there is an edge $\ti{\al}_i \su \ti{\al}$ which appears at least $21\ti{N}(\frac{20+2C(v,D)}{r_0(v,D)}+4)+5$ times.  However, there are no more than $21$ segments of the partition $\mathcal{P}$ in each $\overline{A_{\bar{r}_j^k, 2r_j^k} (x_j^k)}\subset \Nn_j^{k+1}$ and each segment is appoximated by a simplicial subcurve in $\ti{\ga}$ with simplicial length bounded by $\frac{20+2C(v,D)}{r_0(v,D)}+4$. The number of necks is bounded by $\ti{N}$. Hence, there is a $\ga'_i$ in the partition $\mathcal{P}$, such that $\ga'_i \su \al \su \gamma([0,s_j])$ and $\ga'_i$ is in body. Moreover, $\ga'_i$ is approximated by a simplicial curve with an edge that appears at least $5$ times. Then by Lemma \ref{length reduction}(3), there is a curve $c$ connecting $\ga(0)$ and $\ga(s_j)$ such that
\begin{equation}
\length(c) \leq \length(\ga([0,s_j]))-\tilde{C}\cdot r_h \leq \length(\sigma_j)+\tilde{C}\cdot r_h/2-\tilde{C}\cdot r_h. \nonumber
\end{equation}
Thus, we have $\length(c)<\length(\sigma_j)$, which, again, leads to a contradiction.

Therefore, we have proved that the length of $\gamma_j=\sigma_j\cup \gamma([s_j,s_{j+1}])\cup(-\sigma_{j+1})$ is always bounded by $\length(\gamma)-\tilde{C}\cdot r_h/2$.
\end{proof}

Finally, we show that during the homotopy, one can break the curve into several small curves while the total width of the homotopy can be still controlled.

\begin{lemma}\label{lm8}
Let $\ga:[0,1]\ra M$ be a closed curve and $\ga(0)=\ga(1)=p$. Suppose that $\ga=\vee_{i=1}^n \al_i$, where each $\al_i$ is a closed curve with base point $p$. If each $\al_i$ can be contracted to a point in $M$ through a homotopy with width bounded by $W_i$, then there exists a homotopy $H(s,t)$ such that $H(s,0)=\ga$ and $H(s,1)=p$. The width $\om_H$ of this homotopy is bounded by $2\cdot \max_i W_i$.
\end{lemma}

\begin{proof}
Let us denote by $H_i(s,t)$ the homotopy contracting each $\al_i$ and let $p_i=H_i(s,1)$. By our assumption, the curves $\al_i$ have a common base point $\al_i(0)=p$. Let $\si_i=H_i(0,\cdot):[0,1]\ra M$ be the trajectory of $p$ in the homotopy $H_i$. We first homotope the curve $\ga=\vee_{i=1}^n \al_i$ to $\cup_{i=1}^n \si_i\cup (-\si_i)$ through the curves $\cup_{i=1}^n \si_i([0,t])\cup H_i([0,1],t)\cup (-\si_i([0,t]))$, for $t\in[0,1]$. The width of this homotopy is bounded by $\max_i W_i$. Then we contract $\cup_{i=1}^n \si_i\cup (-\si_i)$ to the base point $p$. The total width of this homotopy is bounded by $2\cdot \max_i W_i$.
\end{proof}

\section{Width of the homotopy and length of the shortest closed geodesic}
In this section, we will prove our main results Theorem~\ref{thm1_l} and Theorem~\ref{thm2_w}. We will prove Theorem~\ref{thm2_w}A and B separately, and then we use the result of Theorem~\ref{thm2_w} to show Theorem~\ref{thm1_l}.

Recall that Theorem~\ref{thm2_w}A states that if $M\in\Mm(4,v,D)$, any closed curve $\ga\su M$ can be contracted to a point through a homotopy $H$ with width $\om_H\leq \Omega(v,D)$, where $\Omega$ is a function which only depends on volume $v$ and diameter $D$.

\begin{proof}[Proof of Theorem~\ref{thm2_w}A]
Given a four dimensional manifold $M$ satisfies the above conditions, we first construct a finite covering $\mathcal{O}$ of $M$ as in Lemma~\ref{lem2_c} and a  graph $\Si$ from this covering as it is in the beginning of the Section~\ref{sec3}.

Let $\ga:[0,1]\ra M$ be any closed curve. By Lemma~\ref{lm1}, there is a piecewise loop $\al\su \Si$ such that $\ga$ is homotopic to $\al$ through a homotopy $H_1$ with $\om_{H_1}\leq 60D$.

Since the simplicial length $m(\al)$ of $\al$ maybe unbounded in terms of $v$ and $D$, our second step is to apply Lemma~\ref{lm2} to break $\al$ into
$m(\al)$ many small curves so that the simplicial length of each small curve is no more than $2\ti{N}^2+1$, where $\ti{N}=\ti{N}(v,D)$ is the number of the balls in the covering of $M$.

In fact, by Lemma~\ref{lm2}, the curve $\al$ is homotopic to $\al'=\al_1\cup\dots\cup\al_{m(\al)}$ through a homotopy $H_2$ such that $m(\al_i)\leq 2\ti{N}^2+1$, for all $i$, and the width $\om_{H_2}\leq 12\ti{N}^2D$. Let $p\in \cap_i \al_i$ be a vertex in $\Si$. Because $M$ is simply-connected, by Lemma~\ref{lm8}, if each curve $\al_i$ can be contracted to a point $p_i\in M$ through a homotopy with width bounded by some function $W$, then the curve $\al'$ can be contracted to the vertex $p$ through a homotopy with width bounded by $2W$.

In order to contract $\al_i$, we apply Lemma~\ref{lm4} to find a simplicial approximation $S_i$ of $\al_i$ in the nerve $\Nn(M)$. By Lemma~\ref{lm4} the 1-chain $S_i$ is contractible in $M$, hence also contractible in $\Nn(M)$ and the number of the 1-simplices in $S_i$ is bounded by $2\ti{N}^2+1$.

We then apply Lemma~\ref{lm5} to control the width of the homotopy contracting $S_i$ in $\Nn(M)$. In fact, because the number of the vertices in $\Nn(M)$ is $\ti{N}$, which, by Lemma~\ref{lem2_c}, is a constant that only depends on the volume bound $v$ and diameter bound $D$ of $M$, the number of the possible intersections of the balls is bounded by a function of $\ti{N}$. Therefore, the function $F$ in Lemma~\ref{lm5} is a function $F(\ti{N})=F(v,D)$.

The simplicial homotopy contracting $S_i$ can be realized through a sequence of closed simplicial curves $\{\si_i\}$ such that $\si_1=S_i$ and $\si_{i+1}-\si_i$ is the boundary of a 2-simplex. Therefore, there are at most $F(\ti{N})$ many such curves in the sequence.

Now each curve $\si_i$ can be realized by a loop $\be_i$ in $\Si$ by connecting the corresponding vertices. We then get a sequence of  loops $\{\be_i\}$ in $\Si$ such that any two consecutive curves $\be_i$ and $\be_{i+1}$ satisfies the condition in Lemma~\ref{lm3}, because the corresponding $\si_i$ and $\si_{i+1}$ are only differed by the boundary of some $2-$simplex. By Lemma~\ref{lm3}, $\be_i$ is homotopic to $\be_{i+1}$ through a homotopy with width bounded by $66D$. Hence we conclude that $\al_i$ can be contracted to a point through a homotopy with width bounded by $66 F(\ti{N})D$.

Finally, by connecting the homotopies $H_1$, $H_2$ and the homotopy contracting $\al'$, we obtain that our original curve $\ga$ can be contracted to a point in $M$ through a homotopy with width bounded by a function $\OM(v,D)$.
\end{proof}

If we assume that there is no closed geodesic on $M$ of which the length is less than $4D$, we may improve the above construction to get an expression of $\OM(v,D)$ in terms of $v$, $D$ and the function $\ti{N}(v,D)$ in Lemma~\ref{lem2_c}. Note that in this case, if we apply certain curve shortening algorithm to a curve of length shorter than $4D$, we are able to contract the curve to a point. Let us first introduce the following notation.

\begin{definition}
Let $\al:[0,1]\ra M$ be a closed contractible curve in a Riemannian manifold $M$, and $H(s,t):[0,1]\z[0,1]\ra M$ a homotopy contracts $\al$ with $H(s,0)=\al(s)$ and $H(s,1)=point\in M$. For any $0\leq t_1\leq t_2\leq 1$ We define
$$\om_H([t_1,t_2])=\sup_{s\in[0,1]} \length(H(s,[t_1,t_2])).$$
\end{definition}

\begin{proof}[Proof of Theorem~\ref{thm2_w}B]
Let $\ga:[0,1]\ra M$ be a closed curve in $M$. Assume that there is no closed geodesic on $M$ of which the length is less than $4D$. We are going to show that the curve $\ga$ can be contracted to a point through a homotopy with width bounded by a function of $v$, $D$ and $\ti{N}$.

Let $\ti{\ga}$ be the approximation of $\ga$ as it is in Lemma~\ref{length reduction} (1). Let
$$Z(v,D)=100\ti{N}^3(\frac{20+2C(v,D)}{r_0(v,D)}+4)+10\ti{N}^2$$
be the same constant as in Lemma~\ref{lm7}. If $m(\ti{\ga})\geq Z(v,D)$, we first apply Lemma~\ref{lm7} to homotope $\ga$ to a loop $\ga_1\vee\dots\vee\ga_{k}$ through a homotopy with width bounded by $2D$, where $k$ is a positive integer and each $\ga_i$ is a loop with length $\leq 4D$.

By Lemma~\ref{lm8}, if each $\ga_i$ can be contracted to a point through a homotopy with width $\om_i$, then one can contract $\ga$ to a point through a homotopy with width bounded by $\max_i 2\om_i$. In our construction below, all curves $\ga_i$ will be contracted in the same way. Therefore, without lost of generality, let us consider the contraction of the curve $\ga_1$. The contraction of $\ga_1$ will be constructed in two steps.

We first construct a family of curves $\{\ga_{a_1\dots a_n}\}$ parameterized by a finite tree $\Tt$ associate to the curve $\ga_1$. We will then choose, from this family $\{\ga_{a_1\dots a_n}\}$, a bounded number of curves to construct a homotopy that contracts $\ga_1$.

The family  $\{\ga_{a_1\dots a_n}\}$ we are going to construct satisfies the following properties:
\begin{enumerate}
\item The family of the curves $\{\ga_{a_1\dots a_n}\}$ is parameterized by a finite tree $\Tt$ in the following way. The root of the tree $\Tt$ is identified with the curve $\ga_1$. Each vertex of $\Tt$ corresponds to a curve $\ga_{a_1\dots a_n}$. For the index $a_1a_2\dots a_n$, the curve $\ga_{a_1\dots a_n}$ is a child of the curve $\ga_{a_1\dots a_{n-1}}$ in $\Tt$.
\item For each curve $\ga_{a_1\dots a_n}$, there is an associated base point $p_{a_1\dots a_n}$ and a homotopy $H_{a_1\dots a_n}$ such that:
\begin{enumerate}
\item If $\ga_{a_1\dots a_{n}}$ has only one child $\ga_{a_1\dots a_n,1}$, then $\ga_{a_1\dots a_n}$ is homotopic to $\ga_{a_1\dots a_n,1}$ through $H_{a_1\dots a_n}$ with width bounded by $138 D$.
\item If $\ga_{a_1\dots a_{n-1}}$ has $k$ children $\{\ga_{a_1\dots a_{n-1},i}\}_{i=1}^k$, where $k\geq 2$, then the curves in $\{\ga_{a_1\dots a_{n},i}\}_{i=1}^k$ have a common base point $p_{a_1\dots a_n,1}=\dots=p_{a_1\dots a_n,k}$ and $\ga_{a_1\dots a_{n-1}}$ is homotopic to the wedge $\vee_{i=1}^k \ga_{a_1\dots a_{n-1},i}$ through $H_{a_1\dots a_n}$ with width bounded by $138D$.
\item If $\ga_{a_1\dots a_n}$ has no children, then it is a point curve in $M$.
\end{enumerate}
\item For each curve $\ga_{a_1\dots a_n}$, there is an approximation $\ti{\ga}_{a_1\dots a_n}$ in $\Ga$ which is homotopic to $\ga_{a_1\dots a_n}$ with width bounded by $67D$. The length of the curve $\ga_{a_1\dots a_n}$ is bounded by $4D$ and the simplicial length of the curve $\ti{\ga}_{a_1\dots a_n}$ is bounded by $Z(v,D)$.
\end{enumerate}
\begin{claim}
We claim that for a curve $\ga_1$, if there exits a family of curves $\{\ga_{a_1\dots a_n}\}$ satisfies the above $(1)-(3)$, then $\ga_1$ can be contracted to a point through a homotopy with width bounded by $\OM (v,D)=414 D\cdot((2(\ti{N}^2+1)^{Z}+1)$.
\end{claim}
\begin{proof}[Proof of the claim.]
Let $h$ denote the height of the tree $\Tt$. We first show that if there is a family of the curves satisfies the above $(1)$ and $(2)$, then $\ga_1$ can be contracted to a point through a homotopy with width bounded by $414Dh$. Then we will use $(3)$ to show that one can always form a new tree $\Tt'$ from $\Tt$ such that the above $(1)-(3)$ hold and the height of the tree $\Tt'$ is bounded by $2(\ti{N}^2+1)^{Z}+1$.

The construction of the homotopy is similar to the proof of Lemma~\ref{lm8}.
If $\gamma_{a_1 \dots a_{n-1}}$ has only one child $\gamma_{a_1 \dots a_{n-1},1}$, let $\sigma_{a_1 \dots a_{n-1}}$ be the trajectory of $p_{a_1 \dots a_{n-1},1}$ under the homotopy $H_{a_1 \dots a_{n-1}}$, then the curve $\gamma_{a_1 \dots a_{n-1}}$ is homotopic to $\sigma_{a_1 \dots a_{n-1}} \cup \gamma_{a_1 \dots a_{n-1},1} \cup -\sigma_{a_1 \dots a_{n-1}}$ through a homotopy with width bounded by $2\cdot 138D$, since the length of $\sigma_{a_1 \dots a_{n-1}}$ is bounded by $138D$. (See Figure~\ref{fig12}.)

\begin{figure}[htbp]
\centering\includegraphics[width=7cm]{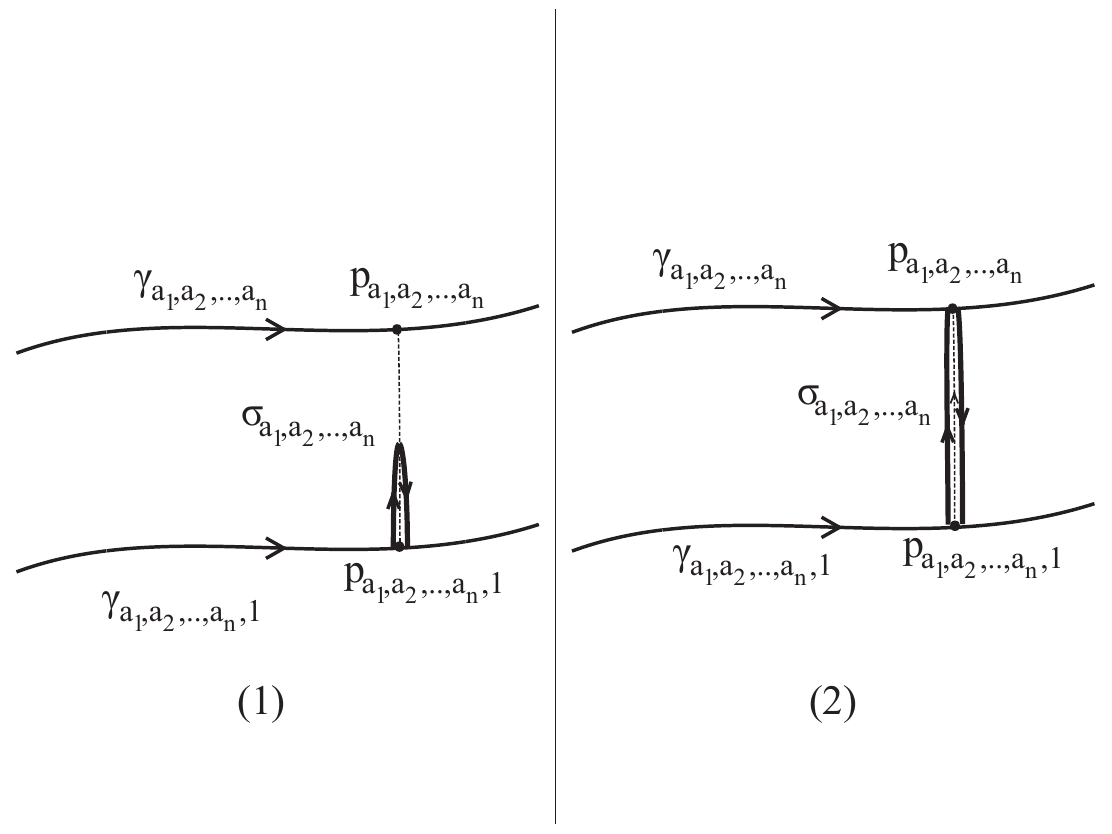}
\vspace*{6pt}
\caption{Homotope $\gamma_{a_1 \dots a_{n-1},1}$ to $\sigma_{a_1 \dots a_{n-1}} \cup \gamma_{a_1 \dots a_{n-1},1} \cup -\sigma_{a_1 \dots a_{n-1}}$ }\label{fig12}
\end{figure}

Similarly, if $\gamma_{a_1 \dots a_{n-1}}$ is homotopic to $\gamma_{a_1 \dots a_{n-1},1} \vee  \dots  \vee \gamma_{a_1 \dots a_{n-1},k}$, then $\gamma_{a_1 \dots a_{n-1}}$ is homotpic to $\sigma_{a_1 \dots a_{n-1}} \cup (\gamma_{a_1 \dots a_{n-1},1} \vee  \dots  \vee \gamma_{a_1 \dots a_{n-1},k}) \cup -\sigma_{a_1 \dots a_{n-1}}$ through a homotopy with width bounded by $2\cdot 138D$, where $\sigma_{a_1 \dots a_{n-1}}$ is the trajectory of $p_{a_1 \dots a_{n-1}}$ under $H_{a_1 \dots a_{n-1}}$.

Therefore, $\gamma_1$ is homotopic to the curves $\sigma_1\cup (\vee_i \ga_{1i})\cup (-\si_1)$, $\si_1\cup \{\cup_i(\si_{1i} \cup (\vee_j\ga_{1ij})\cup (-\si_{1i}))\}\cup (-\si_1)$,\dots, and $\cup(\sigma_{a_1 \dots a_n} \cup -\sigma_{a_1 \dots a_n})$. (See Figure~\ref{fig13}, \ref{fig14})

\begin{figure}[htbp]
\begin{minipage}[c]{0.47\textwidth}
\centering\includegraphics[width=7cm]{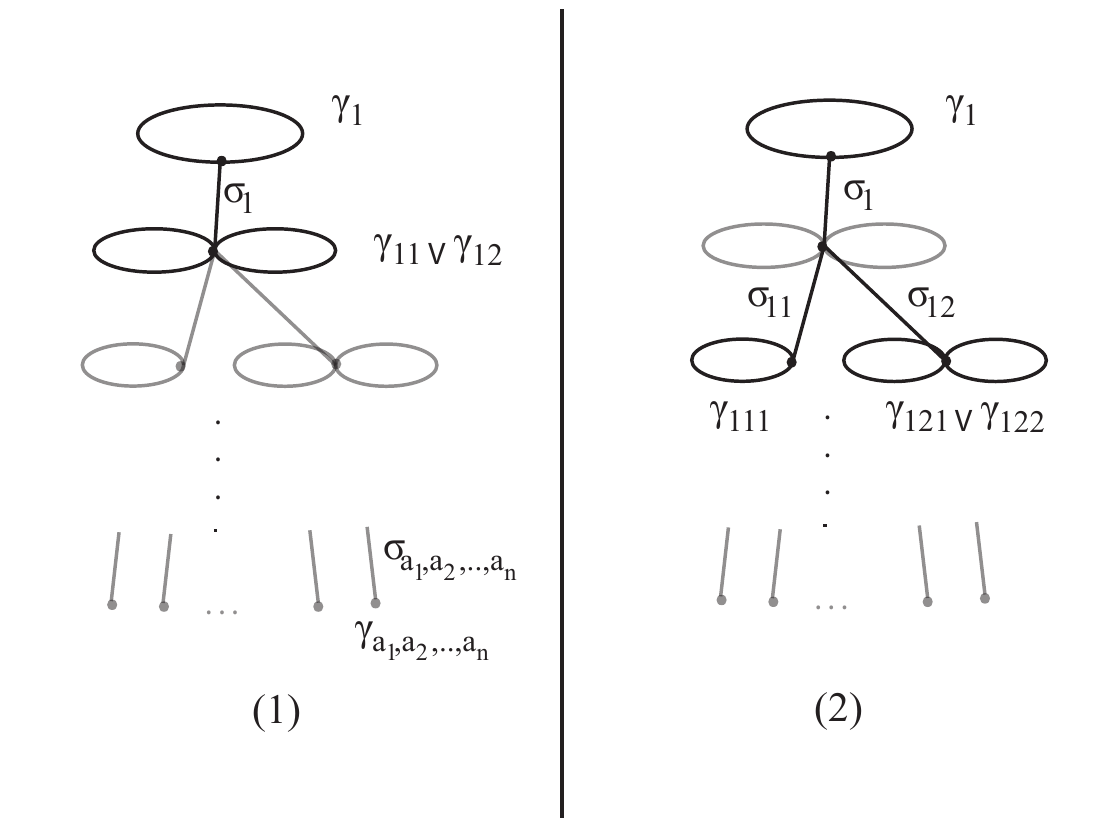}
\vspace*{6pt}
\caption{$\ga_1$ is homotopic to (1) $\sigma_1\cup (\vee_i \ga_{1i})\cup (-\si_1)$, (2) $\si_1\cup \{\cup_i(\si_{1i} \cup (\vee_j\ga_{1ij})\cup (-\si_{1i}))\}\cup (-\si_1)$}\label{fig13}
\end{minipage}
\begin{minipage}[c]{0.47\textwidth}
\centering\includegraphics[width=7cm]{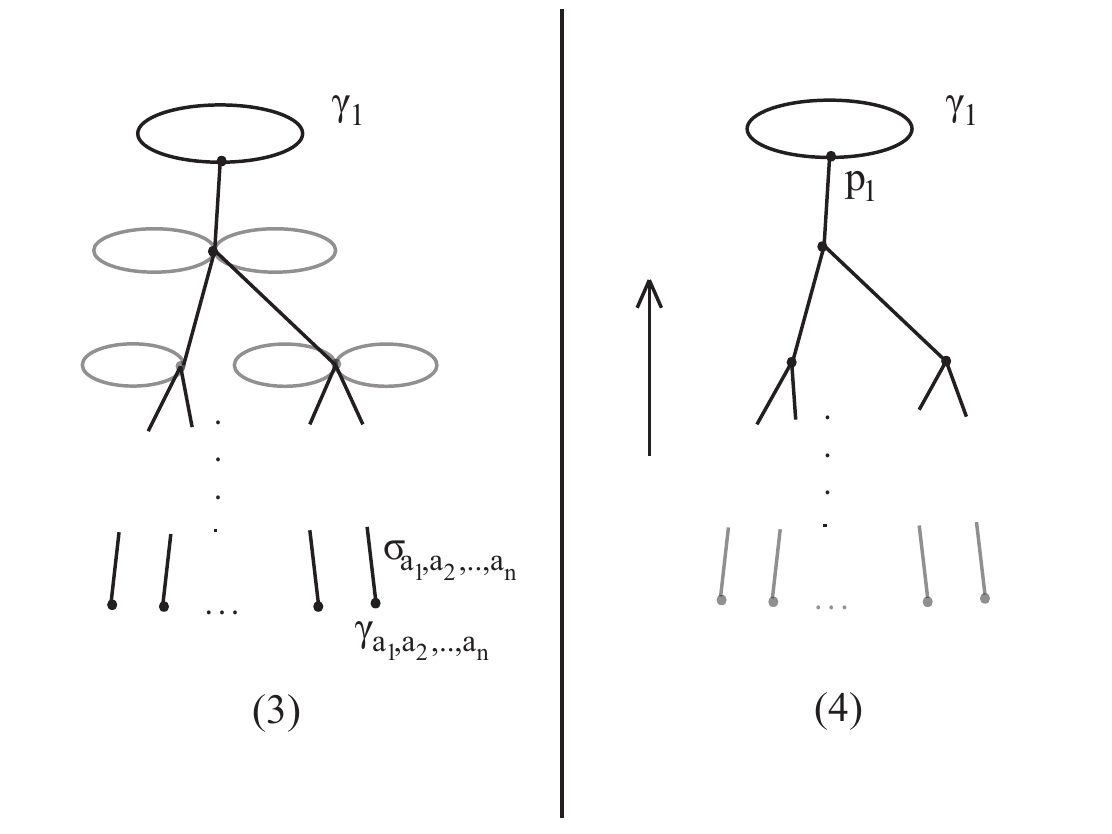}
\vspace*{6pt}
\caption{$\ga_1$ is homotopic to (3) $\cup(\sigma_{a_1 \dots a_n} \cup -\sigma_{a_1 \dots a_n})$, (4) $p_1$}\label{fig14}
\end{minipage}
\end{figure}

The width of a homotopy between $\ga_1$ and $\cup(\sigma_{a_1 \dots a_n} \cup -\sigma_{a_1 \dots a_n})$ is bounded by $276Dh$, where $h$ is the height of $\Tt$. We then contract $\cup(\sigma_{a_1 \dots a_n} \cup -\sigma_{a_1 \dots a_n})$ to $p_1$ by contracting every pair $\sigma_{a_1 \dots a_n} \cup -\sigma_{a_1 \dots a_n}$. The width of this homotopy is bounded by $138D h$. Now by combining the above homotopies together, we obtain a homotopy that contracts $\ga_1$ to $p_1$ with width bounded by $414Dh$.

Let us now construct a new tree such that the height of the tree is bounded by $2(\ti{N}^2+1)^{Z}+1$. The idea is that since, by our construction in Lemma~\ref{lm6}, $\ga_{a_1\dots a_n}$ is homotopic to its approximation $\ti{\ga}_{a_1\dots a_n}$, if two curves $\ga_{a_1\dots a_n}$, $\ga_{b_1\dots b_k}$, where $k> n$, in the family have the same approximation curve $\ti{\ga}_{a_1\dots a_n}=\ti{\ga}_{b_1\dots b_k}$, then we may homotope $\ga_{a_1\dots a_n}$ to $\ga_{b_1\dots b_k}$ through a homotopy with width bounded by $2 \cdot 67D=134D$. In this case, we can form a new tree by connecting the vertex $\ga_{a_1\dots a_n}$ with $\ga_{b_1\dots b_k}$, and delete the vertices in between. Note that the base point condition in $(2)$ may not be satisfied in this situation. However, since the length of the curve $\ga_{b_1\dots b_k}$ is bounded by $4D$, we may homotope the image of $p_{a_1\dots a_n}$ along the curve $\ga_{b_1\dots b_k}$ to $p_{b_1\dots b_k}$. The width of this homotopy is bounded by $4D$. Then in total, the width is bounded by $134D+4D=138D$.

Note that in the {graph}  $\Ga$, the total number of the edges is bounded by $\ti{N}^2$ and hence the number of the curves in $\Ga$ with simplicial length bounded by $Z(v,D)$ is bounded by $N_0:=(\ti{N}^2+1)^{Z}$. If the height of the tree $\Tt$ is greater than $2N_0+1$, then there exits $\ga_{a_1\dots a_n},\ga_{b_1\dots b_k}$ such that $k\geq n+2$ and $\ti{\ga}_{a_1\dots a_n}=\ti{\ga}_{b_1\dots b_k}$. In this case, we replace the subtree with root $\ga_{a_1\dots a_n}$ by connecting $\ga_{a_1\dots a_n}$ with $\ga_{b_1\dots b_k}$ followed with the subtree with root $\ga_{b_1\dots b_k}$. Note that in this case the height of the new subtree with root $\ga_{a_1\dots a_n}$ is reduced by at least one. If the height of the tree is greater than $2N_0+1$, one can always apply the above algorithm to reduce the height of a subtree by at least one. Since $\Tt$ has only finitely many subtrees, after finitely many steps, the height of the tree is decreased by at least one. Therefore, we conclude that the height $h$ can bounded by $2N_0+1$ and hence the width of the homotopy that contracts $\ga_1$ is bounded by $414D\cdot(2N_0+1)$.
\end{proof}

In the rest of the proof, we are going to construct the family of the curves $\{\ga_{a_1\dots a_n}\}$ which is parameterized by a tree $\Tt$ that satisfies the above properties. The idea of this construction is to apply curve shortening to the curve $\ga_1$ and we apply Lemma~\ref{lm7} to get a bouquet of circles when (3) is not satisfied. This family will be constructed inductively.

We apply Birkhoff curve shortening process for free loops (BPFL) to the curve $\ga_1:[0,1]\ra M$. (See \cite{birknoff1960dynamical}, \cite{croke1988area} or \cite{nabutovsky2013length} for detailed discussion about Birkhoff curve shortening process). Recall that during the BPFL, we first take a partition of $0=s_0=s_{n+1}<s_1<s_2\dots<s_n=1$ of $[0,1]$ such that for every $j$, $\ga_1([s_j,s_{j+1}])$ is contained in a half of the injectivity radius at $\ga_1(s_j)$. We join the consecutive midpoints of the arc $\ga_1([s_j,s_{j+1}])$ by a unique minimizing geodesic and obtain a closed piecewise geodesic $\ga_1'$. Then there is a length non-increasing homotopy from $\ga_1$ to $\ga_1'$. And then we apply the same process to $\ga_1'$.

Eventually, the curve $\ga_1$ will either converge to a closed geodesic or a point in $M$. By our assumption on the length of the closed geodesic, the first case is impossible, hence BPFL induces a contraction of $\ga_1$. However, it is worth to note that the width of this contraction may not be bounded by any function of $v$ and $D$.

We denote by $H:[0,1]\z[0,1]\ra M$ the contraction of $\ga_1$ obtained by BPFL such that $H(s,0)=\ga_1(s)$ and $H(s,1)$ is a point in $M$. For a sufficiently large $n$, we take a partition $0=t_0<t_1\dots<t_n=1$ of the second interval of the domain of $H$ such that when $j>1$, the width $\om_H(t_j,t_{j+1})\leq D$. We denote by $\ga_1^j$ the curve $H(\cdot,t_j)$. Let $\ti{\ga}_1^j$ be the approximation of $\ga_1^j$ in $\Ga$ in Lemma~\ref{length reduction} (1) with width of the homotopy bounded by $67D$.

When $j=0$, by our assumption, the simplicial length $m(\ti{\ga}_1^0)$ is bounded by $Z(v,D)$. Let $j_1$ be the first index such that $m(\ti{\ga}_{1}^{j_1})>Z$. In this case, we apply Lemma~\ref{lm7} to homotope $\ga_{1}^{j_1}$ to the wedge of the curves $\ga_{11}^{j_1}\vee\dots\vee\ga_{1k}^{j_1}$ such that:

\begin{enumerate}
\item The width of this homotopy is bounded by $2D$.
\item The length of each $\ga_{1i}^{j_1}$ is bounded by $\length(\ga_1^{j_1})-\frac{1}{32\sqrt{1+2\cdot10^{-3}}}\cdot r_h$, where $r_h$ is the harmonic radium of $M$.
\item  Each $\ga_{1i}^{j_1}$ can be homopoted to a simplicial curve $\ti{\ga}_{1i}^{j_1} \su \Ga$ with width bounded by $15D$. The simplicial length of each $\ti{\ga}_{1i}^{j_1}$ is bounded by $Z(v,D)$.
\end{enumerate}

Note that this furthur implies that the curve $\ga_{1}^{j_1-1}$ is homotopic to $\ga_{11}^{j_1}\vee\dots\vee\ga_{1k}^{j_1}$ with width bounded by $D+2D=3D$. Now we pick the first $j_1-1$ curves $\ga_{\underbrace{1\dots1}_i}=\ga_1^{i}$, for $i=1,2,\dots,j_1-1$. And for $l=1,2,\dots,k$, set $\ga_{1\dots1l}=\ga_{1l}^{j_1}$ and $\ti{\ga}_{1\dots1l}=\ti{\ga}_{1l}^{j_1}$. We then apply the same construction to each $\ga_{1\dots1l}$. Eventually, we are going to obtain a family of curves $\{\ga_{a_1\dots a_n}\}$ which is parameterized by a tree $\Tt$ satisfies the above conditions.

It remains to show that constructed in this way, the hight of the tree $\Tt$ is finite. Indeed, because during the BPFL, one will end up at a point after finite time. And every time we apply Lemma~\ref{lm7} to the curve, the length is decreased by a definite amount $\frac{1}{32\sqrt{1+2\cdot10^{-3}}}\cdot r_h$. In other words, Lemma~\ref{lm7} can be applied for at most $\length(\ga_1)/(\frac{1}{32\sqrt{1+2\cdot10^{-3}}}\cdot r_h)$ times and hence the hight of $\Tt$ is finite.
\end{proof}

We may now proceed to the proof of Theorem~\ref{thm1_l}.

\begin{proof}[Proof of Theorem~\ref{thm1_l}]
Let $M\in\Mm(4,v,D)$. Suppose that there is no closed geodesic of length $\leq 4D$. Then Theorem~\ref{thm2_w} implies that every loop in $M$ may be contracted via a homotopy with width bounded by $\OM=\OM(v,D)$. This further implies that the depth $S_p(M,4D)\leq \max\{4D,2\OM(v,D)+2D\}$.

Let $\OM_p(M)$ denote the space of continuous maps $\{S^1\ra M\}$ based at $p\in M$ and $\OM^E_pM$ the subspace where every curve is of length $\leq E$. Now by taking the integer $k=2$ in Theorem~\ref{thm4_q}, we conclude that for every positive integer $m$, every map $f:S^{m}\ra \OM_pM$ is homotopic to a map $\ti{f}:S^m\ra \OM^F_pM$, where
$$F=F(m,v,D)=10\cdot m+D+(2m-1)\cdot\max\{L,2\OM+2D\}.$$
And in particular, the length of a shortest periodic geodesic does not exceed $F(m,v,D)$.

Finally, since our manifold is simply-connected, suppose it is $(l-1)-$connected but not $l-$connected for $l\geq 2$, the above argument shows that there is a periodic geodesic of length $\leq F(l,v,D)$. However, since we know that $H^4(M)\neq0$, by Hurewicz theorem (see \cite[Theorem~4.32]{hatcher2002algebraic}), if $M$ is $3-$connected then $\pi_4(M)\cong H^4(M)\neq 0$, and hence we can take $F(v,D)=F(4,v,D)$ in the above argument.
\end{proof}

\section*{Acknowledgement}
The authors are grateful to Alexander Nabutovsky and Regina Rotman for suggesting this problem and numerous helpful discussions.
We thank Vitali Kapovitch for useful discussions about manifolds with bounded Ricci curvature. We thank Robert Haslhofer for useful discussions about the $\ep$-regularity theorem. We also thank Aaron Naber for answering several questions about his work \cite{cheeger2014regularity} with Jeff Cheeger.

\bigskip
\bibliographystyle{alpha}
\bibliography{mybib}

\end{document}